\documentclass{siamart171218}
\usepackage{amsmath}
\usepackage{amssymb}
\usepackage{enumerate}
\usepackage{algorithm,algorithmic}
\usepackage{caption}

\usepackage{pgfplots, pgfplotstable, booktabs, colortbl, array}
\usepackage{multirow}

\usepackage{float}
\floatstyle{plaintop}
\restylefloat{table}

\usepackage{todonotes}

\newcommand{\ignore}[1]{}
\def\mystrut#1{\rule{0cm}{#1}}
\usepackage[leqno]{amsmath}
\usepackage{amsfonts}
\usepackage{amssymb}
\newtheorem{remark}{Remark}
\newcommand{\norm}[1]{\|#1\|}
\newcommand{\abs}[1]{|#1|}

\def\shiftcholqr{shiftedCholeskyQR}

\newcommand{\uu}{{\bf u}}


\begin{document}

\title{Shifted CholeskyQR
for computing the QR factorization of ill-conditioned matrices }

\author{
Takeshi Fukaya\thanks{
Hokkaido University, Hokkaido, Japan (\email{fukaya@iic.hokudai.ac.jp})}
\and 
Ramaseshan Kannan\thanks{Arup, 3 Piccadilly Place, Manchester M1 3BN, United Kingdom. (\email{Ramaseshan.Kannan@arup.com})} \and
 Yuji Nakatsukasa\thanks{National Institute of Informatics, 2-1-2 Hitotsubashi, Chiyoda-ku, Tokyo 101-8430, Japan (\email{nakatsukasa@nii.ac.jp})} 
\and  Yusaku Yamamoto\thanks{The University of Electro-Communications, Tokyo, Japan / JST CREST, Tokyo, Japan (\email{yusaku.yamamoto@uec.ac.jp})}
\and Yuka Yanagisawa\thanks{Waseda university, Waseda Research Institute for Science and Engineering, Tokyo, Japan (\email{yuuka@aoni.waseda.jp})}
}

\date{}

\headers{Shifted CholeskyQR}{Fukaya, Kannan,  Nakatsukasa, Yamamoto and Yanagisawa}

\maketitle

\begin{abstract}
The Cholesky QR algorithm is an efficient communication-minimizing algorithm for computing the QR factorization of a tall-skinny matrix. 
Unfortunately it has the 
inherent numerical instability and breakdown when the matrix is ill-conditioned. 
A recent work establishes that the instability can be cured 
by repeating the algorithm twice (called CholeskyQR2).
However, the applicability of CholeskyQR2 is still limited by the requirement that the Cholesky factorization of the Gram matrix runs to completion, 
which means it does not always work for matrices $X$ with $\kappa_2(X)\gtrsim {\uu}^{-\frac{1}{2}}$ where ${\uu}$ is the unit roundoff.
In this work we extend the applicability to $\kappa_2(X)=\mathcal{O}(\uu^{-1})$ by introducing a shift to the computed Gram matrix so as to guarantee the Cholesky factorization $R^TR= A^TA+sI$ 
succeeds numerically. We show that the computed $AR^{-1}$ has reduced condition number $\leq {\uu}^{-\frac{1}{2}}$, for which CholeskyQR2 safely computes the QR factorization, yielding a computed $Q$ of orthogonality $\|Q^TQ-I\|_2$ and 
residual $\|A-QR\|_F/\|A\|_F$ both $\mathcal{O}({\uu})$. 
Thus we obtain the required QR factorization by essentially running Cholesky QR thrice. 
We extensively analyze the resulting algorithm \shiftcholqr3\ to reveal its excellent numerical stability.
\shiftcholqr3 is also highly parallelizable, and applicable and effective also when working in an oblique inner product space. We illustrate our findings through experiments, in which we achieve significant (up to x40) speedup over alternative methods.
\end{abstract}


\section{Introduction}

Computing the QR factorization $X=QR$ is required in various applications in scientific computing. 
The Cholesky QR algorithm computes the factorization by:
\begin{eqnarray}
&& A  = X^{\top}X, \label{eq:gram}\\
&&R = \mbox{chol}(A), \label{cholA}\\
&& Q  =  XR^{-1}, 
\end{eqnarray}
where  $\mbox{chol}(A)$ denotes the Cholesky factor of $A$. 
Cholesky QR is a communication-avoiding algorithm
whose communication cost is equivalent to that of the TSQR algorithm, which has been devised specifically to reduce communication for the QR factorization of tall-skinny matrices~\cite{implementqr}. 
Cholesky QR has the advantage
over TSQR that its arithmetic cost is about half and that its
reduction operator is addition, while that of TSQR is a QR
factorization of a small matrix~\cite{fukaya2014choleskyqr2}. 
As a result, Cholesky QR usually runs faster than TSQR.
However, Cholesky QR is rarely used in practice because of its
instability: the distance from orthogonality of its computed $Q$
grows rapidly with the condition number of the input matrix.
By contrast, TSQR is unconditionally stable. 

A recent work by the authors~\cite{yamamotoetna2015}
 establishes that the instability can be cured significantly by repeating the algorithm twice (called CholeskyQR2).
However, the applicability of CholeskyQR2 is still limited by the requirement that the Cholesky factorization of the Gram matrix runs to completion, 
which means it does not always work for matrices $A$ with $\kappa_2(X)=\mathcal{O}({\uu}^{-\frac{1}{2}})$ or larger, where $\uu$ is the unit roundoff.
In this work we extend the applicability of CholeskyQR-based algorithms to $\kappa_2(X)=\mathcal{O}(\uu^{-1})$. 

The idea is to execute a preconditioning step so that the conditioning is improved to a point where CholeskyQR2 is applicable. 
How do we find an effective preconditioner? 
An inspiration to answer this is the fact that Cholesky QR and
CholeskyQR2 belong to the category of \emph{triangular orthogonalization} type algorithms~\cite[Lecture 10]{trefbau}, in contrast to Householder type algorithms, which follows the principle of \emph{orthogonal triangularization}. 
We summarize the classification of algorithms in terms of their principle, communication cost, and stability in Figure~\ref{fig:diagram}, which also clarifies where our contribution (shifted CholeskyQR3) stands. 

\begin{figure}[htbp]
  \centering
        \includegraphics[width=90mm]{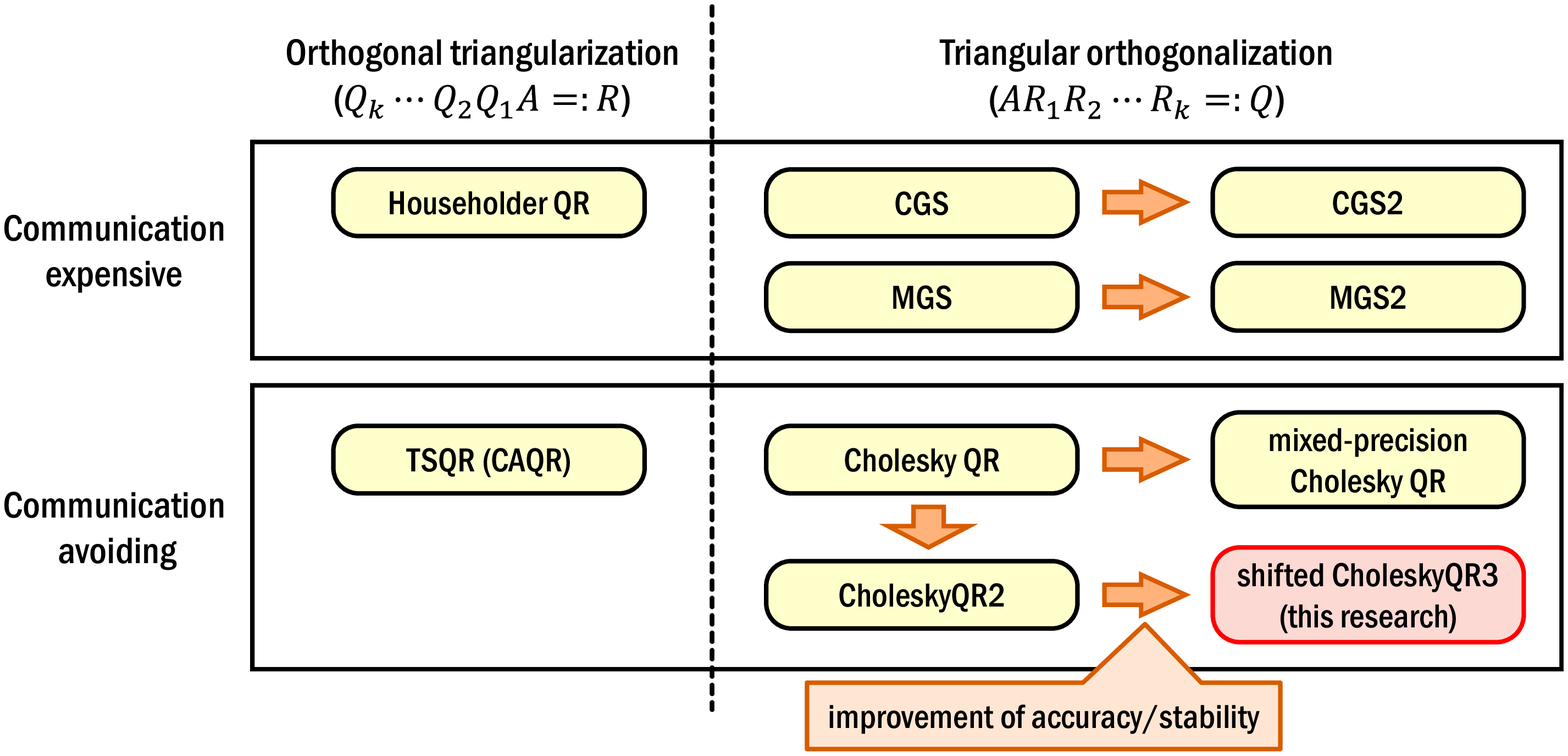}
  \caption{Classification of QR factorization algorithms. 
}
  \label{fig:class}
\end{figure}

In triangular orthogonalization, one right-multiplies an appropriate upper triangular matrix $R$ so that $\kappa_2(AR)=1$. 
Now, what if our goal is merely $\kappa_2(AR)=O(\uu^{-1/2})$? 
Once we have this, we can safely compute the QR factorization $AR=QR_1$ using CholeskyQR2, to arrive at the overall QR factorization $A=Q(R_1R^{-1})$.

Clearly, such $R$ is not unique, and we propose one way of finding such $R$. 
Namely, as in Cholesky QR we compute the Gram matrix, but add a small shift $sI$ so as to guarantee the Cholesky factorization $R^TR= A^TA+sI$ does not break down numerically. We show 
that under the mild assumption $\kappa_2(A)\leq \uu^{-1}$,
the resulting $AR^{-1}$ (note that $R^{-1}$ is also triangular) has reduced condition number $\leq {\uu}^{-\frac{1}{2}}$, for which CholeskyQR2 safely computes the QR factorization, yielding computed $Q,R$ with excellent orthogonality $\|Q^TQ-I\|=\mathcal{O}({\uu})$ and 
residual $\|A-QR\|_F/\|A\|_F=\mathcal{O}({\uu})$, overall a backward stable QR factorization. 
The algorithm is deceptively simple (essentially the only new ingredient being the introduction of a shift); the analysis is however not trivial. We give detailed analysis that gives the constants hidden in the $\mathcal{O}(\uu)$ notation. 

The main message of this paper is that for any matrix with condition number well above ${\uu^{-\frac{1}{2}}}$ (but bounded by ${\uu^{-1}}$), 
the QR factorization can be computed in a backward stable manner by essentially running Cholesky QR thrice. 
We refer to this overall algorithm as \shiftcholqr3; Figure~\ref{fig:diagram} shows its diagram, and how $\kappa_2(A)$ is reduced eventually to 1 through repeated multiplication by triangular matrices. 

\begin{figure}[htbp]
  \centering
        \includegraphics[width=80mm]{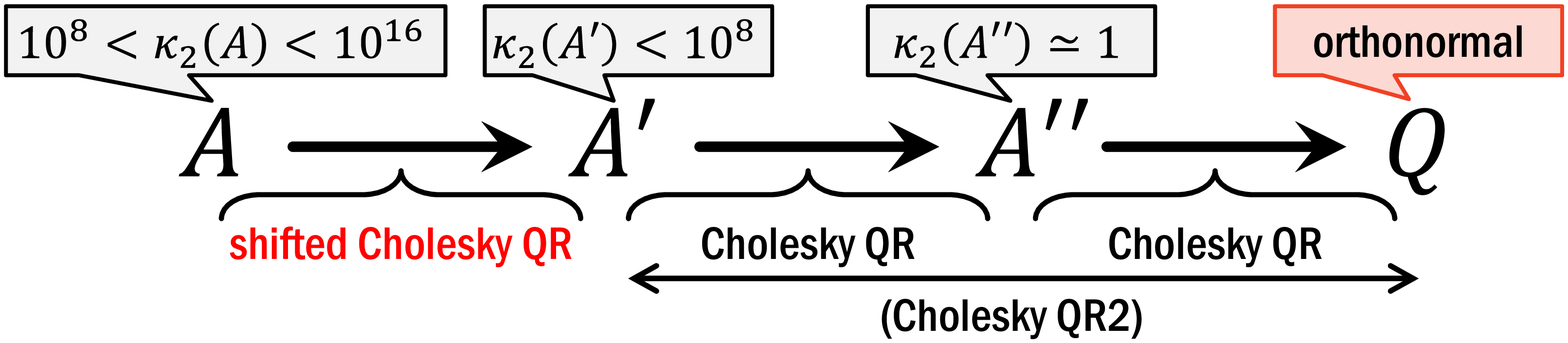}
  \caption{Diagram illustrating \shiftcholqr3.}
  \label{fig:diagram}
\end{figure}

Let us comment on related studies in the literature. 
A recent work of Yamazaki et al.~\cite{yamazaki2015mixed} uses doubled precision arithmetic (e.g., quadruple precision when using double precision) for the first two steps~\eqref{eq:gram}, \eqref{cholA}, and shows that the condition number gets reduced by about $O({\uu}^{-1})$, and thus the QR factorization will be obtained by repeating the process. This results in about 8.5 times  as many arithmetic operations (per iteration) as does the standard CholeskyQR without doubled precision. 
Moreover, this approach clearly requires that higher-precision arithmetic is available, which may not always be the case (and even when it is, it often comes with a significant price in speed). 
\shiftcholqr3 developed in this paper does not require change of arithmetic precision  (in our experiments we use only IEEE double precision in which ${\uu}\approx 1.1\times 10^{-16}$), and requires just one more CholeskyQR iteration than \cite{yamazaki2015mixed}, thus is usually much faster. 

As a bonus, our development is straightforward to apply to computing the QR factorization in a non-standard inner product space induced by a positive definite matrix $B\succ 0$, in which  $(x,y)_B = x^TBy$. 
Available algorithms for this task include 
 (modified) Gram-Schmidt \cite{rozlovznik2012numerical}, and Cholesky QR \cite{lowery2014stability,stathopoulos2002block,kannanthesis}. 
A recent work by Lowery and Langou~\cite{lowery2014stability} studies the numerical stability, with no method apparently being (near) optimal in both orthogonality and backward error. We analyze the stability of \shiftcholqr3 in this case to show it has favorable stability properties.
Moreover, \shiftcholqr3 is significantly faster than Gram-Schmidt type algorithms, achieving up to 40-fold speedup in our experiments with sparse $B$.

This paper is organized as follows. In Section~\ref{sec:alg} we describe the \shiftcholqr\ algorithm. 
Section~\ref{sec:stab} analyzes \shiftcholqr\ in detail and shows that it can be used to reduce $\kappa_2(X)$. In Section~\ref{sec:shiftandchol2} we combine \shiftcholqr\ and CholeskyQR2 to derive \shiftcholqr3 for ill-conditioned matrices with $\kappa_2(X)=\mathcal{O}(\uu^{-1})$ and prove its backward stability. Section~\ref{sec:oblique} discusses the extension to the non-standard inner product space. Numerical experiments are shown to illustrate the results in Section~\ref{sec:ex} and Section~\ref{sec:perf} summarizes the performance of our software implementations. 

We primarily focus on real matrices $A\in\mathbb{R}^{m\times n}$, but everything carries over to complex matrices 
$A\in\mathbb{C}^{m\times n}$. We assume $m\geq  n$, and the algorithms developed here are particularly useful in the tall-skinny case  $m\gg n$.  

\section{Shifted Cholesky QR}\label{sec:alg}
To overcome the numerical breakdown in the Cholesky factorization~\eqref{cholA}, we propose a simple remedy: introduce a small shift $A^TA+sI$ to 
force the computed Gram matrix $A^TA$ to be 
numerically positive definite, so that its Cholesky factorization runs without breakdown. The rest of the algorithm is the 
same as Cholesky QR, and the algorithm \shiftcholqr\ can be summarized in pseudocode in Algorithm \ref{alg1}. 

 \begin{algorithm}[h]
 \caption{\shiftcholqr\ for $X=QR$}
 \label{alg1}
 \begin{algorithmic}[1]
\STATE $ A  = X^{\top}X,$ 
\STATE choose $s>0$
\STATE $R = \mbox{chol}(A + sI)$
\STATE $Q = XR^{-1}$
    \end{algorithmic}
 \end{algorithm}%

Introducing shifts in Cholesky QR was briefly mentioned in \cite{stathopoulos2002block}, also as a remedy 
for the breakdown. However, the focus there was on the algorithm called SVQB, which computes the 
SVD of the Gram matrix instead of the Cholesky factorization. SVQB computes a factorization 
of the form $A = QB$, where $B$ is a full $n\times n$ matrix. 
While $Q$ is still a basis for the column space of $A$, an additional QR factorization of $B$ is needed 
to obtain a complete QR factorization of $A$. Furthermore, 
experiments suggest that there are benefits in working with a triangular matrix as in Cholesky QR as opposed to full matrices as in SVQB, wirh respect to the 
row-wise stability of the computed decomposition. 

We shall discuss an appropriate choice of the shift $s$, 
which will be $O(\uu)$, and prove that applying \shiftcholqr\
to a matrix $X$ with $\uu^{-\frac{1}{2}}< \kappa_2(X) < \uu^{-1}$ results in a computed $\tilde Q$ with much reduced condition number 
$\kappa_2(\tilde Q)<{\uu}^{-\frac{1}{2}}$, for which CholeskyQR2 suffices to compute the QR factorization. 

\section{Convergence and stability analysis of \shiftcholqr}\label{sec:stab}
In this section we present the main technical analysis of \shiftcholqr. The goal is to show that the algorithm improves the conditioning significantly, that is, $\kappa_2(\tilde Q)\ll \kappa_2(X)$. 

A few assumptions need to be made on the matrix size and condition number. The constants below are not of significant importance but chosen so that the forthcoming analysis runs smoothly. We shall assume the following 
hold:
\begin{eqnarray}
&& 6n^2{\uu}\kappa_2(X) < 1, \label{eq:assumption_kappa} \\
&& mn{\uu} \le \frac{1}{64}, \label{eq:assumption_m} \\
&& n(n+1){\uu} \le \frac{1}{64}, \label{eq:assumption_n}\\
&& 11\{mn+n(n+1)\}{\uu}\|X\|_2^2 \le s\le \frac{1}{100}\|X\|_2^2. \label{eq:assumption_s} 
\end{eqnarray}
Roughly speaking, 
the first three assumptions require
that 
the condition number $\kappa_2(X)$ is safely bounded above by ${\uu^{-1}}$, and 
the matrix dimensions $m,n$ are small compared with the precision $\uu^{-1}$. 
We reiterate that $\kappa_2(X)>\uu^{-\frac{1}{2}}$ is allowed, a crucial difference from CholeskyQR2 treated in~\cite{yamamotoetna2015}. 
The assumption \eqref{eq:assumption_s} imposes that $s$ is large enough for Cholesky to work, but small compared with $\|X\|_2^2$. 

Note that by~\eqref{eq:assumption_m}, \eqref{eq:assumption_n} we have 
\begin{equation}
\gamma_m: = \frac{m{\uu}}{1-m{\uu}} \le 1.02m{\uu}, \quad \gamma_{n+1}: = \frac{(n+1){\uu}}{1-(n+1){\uu}} \le 1.02(n+1){\uu}.
\label{eq:gammabound}
\end{equation}

\subsection{Preparations}
We denote the computed quantities in 
 \shiftcholqr, accounting for the numerical errors, by 
\begin{eqnarray}
\hat{A} &=& X^{\top}X + E_1, \label{eq:forward1} \\
\hat{R}^{\top}\hat{R} &=& \hat{A} + sI + E_2 \;=\; X^{\top}X + sI + E_1 + E_2,  \label{eq:shiftchol} \\
\hat q_{i}^{\top}&=& x_{i}^{\top}(\hat{R}+\Delta \hat R_{i})^{-1}\quad (i=1,2,\dots m). \label{eq:backward2}
\end{eqnarray}
$\hat{q}_{i}^{\top}$, $x_{i}^{\top}$ are the $i$th rows of $X$ and $\hat{Q}$ respectively. 
$E_{1}$ is the matrix-matrix multiplication error in the computation of the Gram matrix 
$X^{\top}X$, and $E_{2}$ is the $\hat{A}$ backward error incurred when computing the Cholesky factorization. 
$\Delta \hat R_{i}$ is the backward error involved in the solution of the linear system 
$q_{i}^{\top}\hat{R}=x_{i}^{\top}$. 

We shall take $s$ so that 
\begin{equation}
\norm{E_{1}}_{2},\norm{E_{2}}_{2}=o(s), 
\label{s_shift}
\end{equation}
which means $s=c\max(\norm{E_{1}}_{2},\norm{E_{2}}_{2})$ for some $c>1$.
In fact we shall see that $\|E_1\|_2,\|E_2\|_2=O(\uu)\|X\|^2_2$, so~\eqref{s_shift} simply 
means $s$ is chosen to be safely larger than $\uu\|X\|^2_2$ (qualitatively this is also assumed by~\eqref{eq:backward2}). 
In particular, we write 
\begin{equation}  \label{eq:sdef}
s:=\alpha \norm{X}^{2}_{2}, \quad 0<\alpha<1.  
\end{equation}
\eqref{eq:backward2} gives 
$$
\hat{q}_{i}^{\top}=x_{i}^{\top}(\hat{R}+\Delta \hat{R}_{i})^{-1}=x_{i}^{\top}(I+\hat{R}^{-1}\Delta \hat{R}_{i})^{-1}\hat{R}^{-1}. 
$$
Hence $(I+\hat{R}^{-1}\Delta \hat R_{i})^{-1}=I+ \breve{R}_i$, $\breve{R}_i:=\sum_{k=1}^{\infty}(-\hat{R}^{-1}\Delta\hat{R}_i)^k$, so defining 
\begin{equation}
\Delta  x_i^{\top} := x_i^{\top} \breve{R}_i
\label{eq:breveRi}
\end{equation}
we have 
\begin{equation}\label{eq:backward2x}
\hat{q}_{i}^{\top}= (x_{i}^{\top}+\Delta x_{i}^{\top})\hat{R}^{-1}\quad (i=1,2,\dots m). 
\end{equation}
Let $\Delta X=
\bigg[\begin{smallmatrix}
\Delta x_{1}^{\top}  \\
\vdots\\
\Delta x_{m}^{\top}    
\end{smallmatrix}\bigg]
$ be the matrix obtained by stacking up the row vectors $\Delta x_{i}^{\top}  $.
Then 
\begin{equation}
\hat{Q}=(X+\Delta X)\hat{R}^{-1}. 
\label{eq:backward2X}
\end{equation}

\subsection*{$\Delta x_i^{\top}$ and $\Delta\hat{R}_i$}
For later use, here we examine the relation between 
$\Delta\hat{R}_i$ and $\Delta x_i^{\top}$. 
From~\eqref{eq:backward2} we have 
$
\hat{q}_i^{\top}\hat{R}+\hat{q}_i^{\top}\Delta\hat{R}_i = x_i^{\top}.
$
We also have from~\eqref{eq:backward2x}
$
\hat q_i^{\top}\hat{R} = x_i^{\top}+\Delta x_i^{\top}.
$
Combining these we obtain 
\begin{equation}
\Delta x_i^{\top} = -\hat{q}_i^{\top}\Delta\hat{R}_i.
\label{eq:deltaxdeltaR}
\end{equation}

\subsubsection{Error in computing $X^{\top}X$}
For general matrices $A\in\mathbb R^{m\times n}$, $B\in\mathbb R^{n\times m}$, 
the error in computing the matrix product $C=AB$  can be bounded by~\cite[Ch.~3]{Higham:2002:ASNA}
\begin{equation}
\abs{AB-fl(AB)}\le \gamma_{n}\abs{A}\abs{B}
\label{err:multiplication}
\end{equation}
Here $\abs{A}$ is the matrix whose $(i,j)$ element is $\abs{a_{ij}}$. 
In \cite{yamamotoetna2015} it is shown that $E_{1}$ is bounded by 
\begin{equation}
\norm{E_{1}}_{2}\le \norm{\abs{E_{1}}}_{F}\le \gamma_{m}n\norm{X}^{2}_{2}
\label{err_1}
\end{equation}
Simplifying \eqref{err_1} using  \eqref{eq:gammabound} yields 
$\norm{E_{1}}_{2}\le 1.1mnu\norm{X}^{2}_{2}$.

\paragraph{Backward error in the Cholesky factorization}
Suppose that $A\in\mathbb R^{n\times n}$ is a positive definite matrix 
and its Cholesky factorization computed in floating-point arithmetic runs to completion 
and outputs $\hat{R}$. Then there exists $\Delta A\in\mathbb R^{n\times n}$ such that~\cite[Thm. 10.3]{Higham:2002:ASNA}
\begin{equation}
\hat{R}^{\top}\hat{R}=A+\Delta A, \quad \abs{\Delta A}\le \gamma_{n+1}\abs{\hat{R}^{\top}}\abs{\hat{R}}
\label{err:cholesky}
\end{equation}
Applying this to our situation gives 
$$\norm{E_{2}}_{2}\le\norm{\abs{E_{2}}}_{F}\le \gamma_{n+1}n(\norm{\hat{A}}_{2}+\norm{E_{2}}_{2}).
$$ 
Using \cite[eqn. (3.16)]{yamamotoetna2015} in the right-hand side gives 
\begin{eqnarray}
\norm{E_{2}}_{2}&\le&  \frac{\gamma_{n+1}n((1+\gamma_{m}n+\alpha)}{1-\gamma_{n+1}n}\norm{X}^{2}_{2}\nonumber\\
&\le& \frac{1.02(n+1){\uu}\cdot n(1+1.02m{\uu}\cdot n+0.01)}{1-1.02(n+1){\uu}\cdot n}\|X\|_2^2 \nonumber \\
&\le& \frac{1.02\cdot n(n+1){\uu}\cdot(1+1.02\cdot\frac{1}{64}+0.01)}{1-\frac{1.02}{64}}\|X\|_2^2 \le 1.1n(n+1){\uu}\|X\|_2^2, 
\label{err_2}
\end{eqnarray}
where we have used $\alpha=s/\|X\|_2^2\le 1/100$ and \eqref{eq:assumption_m}, \eqref{eq:assumption_n}.

\paragraph{Bounding $\|\hat{R}^{-1}\|_2$}
Using Weyl's theorem~\cite[Sec.~8.6.2]{golubbook4th} in \eqref{eq:shiftchol} gives 
\begin{equation}\label{sigr}
 \sigma_{n}(\hat{R})^{2}\ge(\sigma_{n}(X))^{2}+s-(\norm{E_{1}}_{2}+\norm{E_{2}}_{2}).
\end{equation}
By $\|E_1\|_2\le 1.1mn{\uu}\|X\|_2^2$ and \eqref{eq:assumption_s} and \eqref{err_2} 
 we obtain 
\begin{equation}
\|E_1\|_2+\|E_2\|_2 \le 1.1(mn+n(n+1)){\uu}\|X\|_2^2 \le 0.1s.
\label{eq:E1E2bound}
\end{equation}
Substituting this into~\eqref{sigr} we obtain 
\begin{equation}\label{eq:sigman}
(\sigma_n(\hat{R}))^2 \ge (\sigma_n(X))^2 + 0.9s, 
\end{equation}
therefore 
\begin{equation}
\|\hat{R}^{-1}\|_2 = \sigma_n(\hat{R})^{-1} \le \frac{1}{\sqrt{(\sigma_n(X))^2+0.9s}}.
\label{eq:Rinvbound}
\end{equation}

\subsubsection{Bounding $\|X\hat{R}^{-1}\|_2$}
We next bound $\|X\hat{R}^{-1}\|_2$. This can be done using 
\eqref{eq:E1E2bound} and \eqref{eq:Rinvbound} as 
\begin{eqnarray}
\|X\hat{R}^{-1}\|_2 &\le& \sqrt{1+\|\hat{R}^{-1}\|_2^2(s+\|E_1\|_2+\|E_2\|_2)} \nonumber \\
&\le& \sqrt{1+\frac{1.1s}{(\sigma_n(X))^2+0.9s}} 
\le \sqrt{1+\frac{1.1}{0.9}} \le 1.5.
\label{eq:XRinvbound}
\end{eqnarray}

\subsubsection{Bounding $\|\Delta\hat{R}_i\|_2$}
Let $R\in\mathbb R^{n\times n}$ be a nonsingular upper triangular matrix. 
Generally, the computed solution $\hat{x}$ obtained by solving an upper triangular linear system 
$Rx=b$ by back substitution in floating-point arithmetic satisfies~\cite[Thm.~8.5]{Higham:2002:ASNA}
\begin{equation}
(R+\Delta R)\hat{x}=b,  \quad \abs{\Delta R}\le \gamma_{n}\abs{R}.
\end{equation}
We have for $1\le i\le m$
\begin{equation}
\|\Delta\hat{R}_i\|_2 \le \|\abs{\Delta\hat{R}_i}\|_F \le \gamma_n \sqrt{n}\|\hat{R}\|_2.
\label{eq:DeltaRbound}
\end{equation}

From~\eqref{eq:shiftchol}, \eqref{eq:assumption_s},  and \eqref{eq:E1E2bound} we obtain 
\begin{equation}
\|\hat{R}\|_2^2 \le \|X\|_2^2+s+\|E_1\|_2+\|E_2\|_2 \le \|X\|_2^2 + 1.1s \le 1.1\|X\|_2^2.
\label{eq:hatR2bound}
\end{equation}
Substituting this into \eqref{eq:DeltaRbound} gives 
\begin{equation}\label{eq:DeltaRibound}
\|\Delta\hat{R}_i\|_2 \le 1.02n{\uu}\cdot\sqrt{n}\cdot\sqrt{1.1}\|X\|_2 \le 1.1n\sqrt{n}{\uu}\|X\|_2.
\end{equation}

\subsubsection{Bounding $\|\Delta X\|_2$ roughly}
Here we give a rough bound for $\|\Delta X\|_F$ and prove that $\|\Delta X\|_F=O({\uu}\|X\|_2^2/\sqrt{s})$.
This will be insufficient for proving that 
$\kappa_2(\hat{Q})=O(1/\sqrt{\uu})$, 
for which we will need $\|\Delta X\|_F=O({\uu}\|X\|_2)$, which we will prove later 
after having obtained a bound for 
$\hat{Q}$. We shall proceed as follows. 
\begin{enumerate}
\item Derive the ``rough'' bound $\|\Delta X\|_F=O({\uu}\|X\|_2^2/\sqrt{s})$. 
\item Use above to show $\|\hat{Q}\|_2=O(1)$. 
\item Use above and \eqref{eq:deltaxdeltaR} to prove the ``tight'' bound $\|\Delta X\|_F=O({\uu}\|X\|_2)$. 
\item Use above to prove 
$\kappa_2(\hat{Q})=O(\uu^{-\frac{1}{2}})$. 
\end{enumerate}

To establish the first statement we recall~\eqref{eq:breveRi}, and  bound $\|\breve{R}_i\|_2$ as 
\begin{eqnarray}
\|\breve{R}_i\|_2 &\le& \sum_{k=1}^{\infty}(\|\hat{R}^{-1}\|_2\|\Delta\hat{R}_i\|_2)^k 
= \frac{\|\hat{R}^{-1}\|_2\|\Delta\hat{R}_i\|_2}{1-\|\hat{R}^{-1}\|_2\|\Delta\hat{R}_i\|_2} .
\label{eq:breveRibound}
\end{eqnarray}
We bound the denominator from below as 
\begin{eqnarray}
1-\|\hat{R}^{-1}\|_2\|\Delta\hat{R}_i\|_2 &\ge& 1-\frac{1.1n\sqrt{n}{\uu}\|X\|_2}{\sqrt{(\sigma_n(X))^2+0.9s}} \nonumber \\
&\ge& 1-\frac{1.1n\sqrt{n}{\uu}\|X\|_2}{\sqrt{0.9\cdot11mn{\uu}\|X\|_2^2}} 
\ge 1-\sqrt{\frac{1.21}{9.9}\cdot\frac{n^2{\uu}}{m}} \ge 0.95.
\end{eqnarray}
Substituting this into~\eqref{eq:breveRibound} 
yields 
\begin{equation}
\|\breve{R}_i\|_2 \le \frac{1}{0.95}\cdot\frac{1.1n\sqrt{n}{\uu}\|X\|_2}{\sqrt{(\sigma_n(X))^2+0.9s}} \le \frac{1.2n\sqrt{n}{\uu}\|X\|_2}{\sqrt{(\sigma_n(X))^2+0.9s}}.
\label{eq:Rbrevebound}
\end{equation}
Together with the fact $\|\Delta x_i^{\top}\|\le\|x_i^{\top}\|\,\|\breve{R}_i\|_2$, 
we can bound 
$\|\Delta X\|_F$ as 
\begin{equation}
\|\Delta X\|_F = \sqrt{\sum_{i=1}^m\|\Delta x_i^{\top}\|^2} \le \sqrt{\sum_{i=1}^m\|x_i^{\top}\|^2}\cdot\max_{1\le i\le m}\|\breve{R}_i\|_2 \le \frac{1.2n^2{\uu}\|X\|_2^2}{\sqrt{(\sigma_n(X))^2+0.9s}}, 
\label{eq:DeltaXbound}
\end{equation}
where we used $\sqrt{\sum_{i=1}^m\|x_i^{\top}\|^2} = \|X\|_F\le \sqrt{n}\|X\|_2$ for the last inequality. 

\subsubsection{Bounding $\|\hat{Q}\|_2$}
We now proceed to bound $\|\hat{Q}\|_2$. 
\begin{lemma}
\label{lemma1}
Suppose that $X\in\mathbb{R}^{m\times n}$
with $m\ge n$ satisfies 
\eqref{eq:assumption_m} and~\eqref{eq:assumption_n}. 
Then, the matrix $\hat{Q}$ obtained by applying the \shiftcholqr\ algorithm in floating-point arithmetic to $X$ satisfies
$$
\|\hat{Q}^{\top}\hat{Q}-I\|_2 < 2, 
$$
and hence 
\begin{equation}
\|\hat{Q}\|_2<\sqrt{3}.
\label{eq:hatQbound}
\end{equation}
\end{lemma}

\begin{proof} 
We have 
\begin{eqnarray}
 \hat{Q}^{\top}\hat{Q} \nonumber
&&= \hat{R}^{-\top}(X+\Delta X)^{\top}(X+\Delta X)\hat{R}^{-1} \nonumber \\
&&= I - \hat{R}^{-\top}(s+E_1+E_2)\hat{R}^{-1} + (X\hat{R}^{-1})^{\top}\Delta X\hat{R}^{-1} + \hat{R}^{-\top}\Delta X^{\top}(X\hat{R}^{-1}) + \hat{R}^{-\top}\Delta X^{\top}\Delta X\hat{R}^{-1}. \nonumber
\end{eqnarray}
Thus we can bound $\|\hat{Q}^{\top}\hat{Q}-I\|_2$ as 
\begin{align}
  \label{eq:eq:Th1-2}
\|\hat{Q}^{\top}\hat{Q}-I\|_2 
\le& \|\hat{R}^{-1}\|_2^2(s+\|E_1\|_2+\|E_2\|_2) + 2\|\hat{R}^{-1}\|_2\|X\hat{R}^{-1}\|_2\|\Delta X\|_F\\
\nonumber& + \|\hat{R}^{-1}\|_2^2\|\Delta X\|_F^2 .    
\end{align}

The first term of~\eqref{eq:eq:Th1-2} can be bounded as
\begin{equation}
\|\hat{R}^{-1}\|_2^2(s+\|E_1\|_2+\|E_2\|_2) \le \frac{1.1s}{(\sigma_n(X))^2+0.9s} \le \frac{1.1}{0.9}.
\end{equation}
and for the second term in \eqref{eq:eq:Th1-2}, using \eqref{eq:Rinvbound}, \eqref{eq:XRinvbound} and \eqref{eq:DeltaXbound} we obtain 
\begin{eqnarray}
2\|\hat{R}^{-1}\|_2\|X\hat{R}^{-1}\|_2\|\Delta X\|_F
&\le& 2\cdot\frac{1}{\sqrt{(\sigma_n(X))^2+0.9s}}\cdot 1.5\cdot\frac{1.2n^2{\uu}\|X\|_2^2}{\sqrt{(\sigma_n(X))^2+0.9s}} \nonumber \\
&\le& \frac{2\cdot 1.5\cdot 1.2\cdot\frac{1}{11}s}{0.9s}=\frac{4}{11}.
\end{eqnarray}
For the third term in \eqref{eq:eq:Th1-2}, from 
\eqref{eq:Rinvbound} and \eqref{eq:DeltaXbound}
\begin{eqnarray}
\|\hat{R}^{-1}\|_2^2\|\Delta X\|_F^2
&\le& \frac{1}{(\sigma_n(X))^2+0.9s}\cdot\frac{(1.2n^2{\uu}\|X\|_2^2)^2}{(\sigma_n(X))^2+0.9s} \nonumber \\
&\le& \frac{(1.2\cdot\frac{1}{11}s)^2}{(0.9s)^2} = \frac{16}{1089}.
\end{eqnarray}
Summarizing, we can bound the right-hand side of \eqref{eq:eq:Th1-2} as 
$\|\hat{Q}^{\top}\hat{Q}-I\|_2 < 2, $
as required. 
\end{proof}

\subsubsection{Bounding the residual in \shiftcholqr}
We now bound the residual. 
\begin{lemma}\label{lemma3}
Under the assumptions in Lemma~\ref{lemma1}, 
$\hat{Q}\hat{R}$ computed by \shiftcholqr\ satisfies 
\[\frac{\norm{\hat{Q}\hat{R}-X}_{F}}{\norm{X}_{2}}
\leq 
2n^{2}{\uu}. 
\]
\end{lemma}
\begin{proof}
First note that 
$$
\norm{\hat q_{i}^{\top}\hat{R} -x_{i}^{\top}}=\norm{\hat q_{i}^{\top}\hat{R} -\hat q_{i}^{\top}(\hat{R} +\Delta \hat{R} _{i})}\le \norm{\hat q_{i}^{\top}\Delta \hat{R} _{i}}\le \norm{\hat q_{i}^{\top}}\norm{\Delta \hat{R} _{i}}.
$$
Substituting \eqref{eq:DeltaRibound} into this gives 
\begin{equation}
\|\hat{q}_i^{\top}\hat{R}-x_i^{\top}\| \le \|\hat{q}_i^{\top}\|\cdot 1.1n\sqrt{n}{\uu}\|X\|_2.
\end{equation}
On the other hand, from~\eqref{eq:hatQbound} we have 
\begin{equation}
\|\hat{Q}\|_F < \sqrt{3n}.
\end{equation}
Hence it follows that 
\begin{eqnarray}
\|\hat{Q}\hat{R}-X\|_F &=& \sqrt{\sum_{i=1}^m\|\hat{q}_i^{\top}\hat{R}-x_i^{\top}\|^2} \le \sqrt{\sum_{i=1}^m\|\hat{q}_i^{\top}\|^2}\cdot 1.1n\sqrt{n}{\uu}\|X\|_2 \nonumber \\
&=& \|\hat{Q}\|_F\cdot 1.1n\sqrt{n}{\uu}\|X\|_2 \le 2n^2{\uu}\|X\|_2.
\label{eq:residual}
\end{eqnarray}
\end{proof}

Lemma~\ref{lemma3}
 shows that \shiftcholqr\ gives optimal residual up to a factor involving a low-degree polynomial of $m,n$ (recall that $\|X\|_F\leq \sqrt{n}\|X\|_2$).

\subsection*{Tighter bound for $\|\Delta X\|_F$}
By \eqref{eq:backward2X}
 we have $\Delta X=\hat{Q}\hat{R}-X$, so \eqref{eq:residual} in fact 
provides a bound for $\|\Delta X\|_F$ that is tighter than the
previous bound \eqref{eq:DeltaXbound}:
\begin{equation}  \label{eq:DeltaXbound2}
  \|\Delta X\|_F \le 2n^2{\uu}\|X\|_2.
\end{equation}

\subsection{Main result}
\label{mainresult}
We are now ready to state the main result of the section, which bounds the condition number of $\hat Q$. 

\begin{theorem} \label{eq:kap2}
With one step of \shiftcholqr\ 
in double precision arithmetic 
  applied to $X$ satisfying \eqref{eq:assumption_kappa}--\eqref{eq:assumption_n}
with
  shift $s$ satisfying~\eqref{eq:assumption_s}, 
 we obtain $\hat Q$ with 
\begin{equation} \label{eq:condimprove}
\kappa_2(\hat  Q)\leq 
2\sqrt{1+\alpha(\kappa_2(X))^2}\cdot\sqrt{3}, 
\qquad \mbox{where}\quad \sqrt{\alpha} = \frac{\sqrt{s}}{\|X\|_2}. 
 \end{equation} \end{theorem}

\begin{proof}
Recall from~\eqref{eq:hatQbound} that $\sigma_{1}(\hat{Q})<\sqrt{3}$. The remaining task is to bound $\sigma_n(\hat{Q})$ from below. 
Using
Weyl's theorem in \eqref{eq:backward2X} gives
\begin{equation}  \label{eq:weylg}
\sigma_{n}(\hat{Q}) \ge \sigma_{n}(X\hat{R}^{-1})-\norm{\Delta X \hat{R}^{-1}}_{2}.    
\end{equation}
Using \eqref{eq:Rinvbound} and \eqref{eq:DeltaXbound2} we obtain
\begin{equation}
  \|\Delta X\hat{R}^{-1}\|_2 \le \|\Delta X\|_F\|\hat{R}^{-1}\|_2 \le \frac{2n^2{\uu}\|X\|_2}{\sqrt{(\sigma_n(X))^2+0.9s}}.
  \label{eq:DeltaXRinvbound}
\end{equation}
Note that this is $O(\uu^{\frac{1}{2}})$ when 
we take $s=O(u)$ and regard low-degree polynomials in $m,n$ as constants.

We next bound $\sigma_{n}(X\hat{R}^{-1})$ from below.
We proceed by examining the equation
\begin{equation} \hat{R}^{-\top}(X^{\top}X+sI)\hat{R}^{-1} = I -
  \hat{R}^{-\top}(E_1+E_2)\hat{R}^{-1}.  \end{equation} Let
$X=U\Sigma V^{\top}$ by the SVD. Then for a diagonal matrix $G$, we
can write
\begin{equation} X^{\top}X + sI =
  (U(\Sigma+G)V^{\top})^{\top}(U(\Sigma+G)V^{\top}). \end{equation}
Indeed the left-hand side is $V(\Sigma^2+sI)V^{\top}$ and the
right-hand side $V(\Sigma+G)^2V^{\top}$, so we can take
\begin{equation} G=(\Sigma^2+sI)^{\frac{1}{2}}-\Sigma = {\rm
    diag}(\sqrt{(\sigma_{i}(X))^2+s}-\sigma_{i}(X)) \end{equation} Now
setting $T=U(\Sigma+G)V^{\top}\hat{R}^{-1}$ we have
\begin{equation} T^{\top}T = I -
  \hat{R}^{-\top}(E_1+E_2)\hat{R}. \end{equation} We next bound the
singular values of $T$. By \eqref{eq:E1E2bound} and
\eqref{eq:Rinvbound} we have
\begin{equation}
  \|\hat{R}^{-\top}(E_1+E_2)\hat{R}^{-1}\|_2 \le \|\hat{R}^{-1}\|_2^2(\|E_1\|_2+\|E_2\|_2) \le \frac{0.1s}{(\sigma_n(X))^2+0.9s} \le \frac{1}{9},
\end{equation}
so
$\sigma_i(T)\in[\sqrt{1-\frac{1}{9}},\sqrt{1+\frac{1}{9}}]\subseteq
[0.9,1.1]$.  Therefore, letting $T = U'(I+E')V'$ be the SVD where $E'$ is diagonal, we have
\begin{equation} T =
  U^{\prime}(I+E^{\prime})V^{\prime\top}=U^{\prime}V^{\prime\top}(I+V^{\prime}E^{\prime}V^{\prime\top})
  = Q(I+E), \label{eq:IplusE} \end{equation}
where $Q=U^{\prime}V^{\prime\top}$ has orthonormal columns, and 
$\|E\|_2=\|V^{\prime}E^{\prime}V^{\prime\top}\|_2=\|E^{\prime}\|_2\le 0.1$.

Now plugging into~\eqref{eq:IplusE} the definition of $T$ gives 
\begin{eqnarray} 
Q(I+E) &=&U(\Sigma+G)V^{\top}\hat{R}^{-1}.
\end{eqnarray}
Recalling that $X=U\Sigma V^\top$, 
we have $\sigma_i(X\hat{R}^{-1})=\sigma_i(\Sigma V^\top\hat{R}^{-1})$
and 
\begin{eqnarray}
\Sigma V^\top \hat{R}^{-1} &=&
  \Sigma(\Sigma+G)^{-1} U^{\top}Q(I+E) \nonumber \\ &=& {\rm
    diag}\left(\frac{\sigma_i(X)}{\sqrt{(\sigma_i(X))^2+s}}\right)U^{\top}Q(I+E). \end{eqnarray}
Using the general inequality for singular values of matrix products~$\sigma_{\min}(AB)\ge\sigma_{\min}(A)\sigma_{\min}(B)$ (which holds when $A$ or $B$ is square) along with~\eqref{eq:sigman}, we obtain
\begin{equation}
\sigma_n(X\hat{R}^{-1}) \ge \frac{\sigma_n(X)}{\sqrt{(\sigma_n(X))^2+s}}\cdot 0.9.
\label{eq:XRinvbound2}
\end{equation}

Using this and~\eqref{eq:DeltaXRinvbound}, from~\eqref{eq:weylg} we obtain
\begin{eqnarray}
\sigma_n(\hat{Q}) &\ge& \frac{0.9\sigma_n(X)}{\sqrt{(\sigma_n(X))^2+s}} - \frac{2n^2{\uu}\|X\|_2}{\sqrt{(\sigma_n(X))^2+0.9s}} \nonumber \\
&\ge& \frac{0.9\sigma_n(X)}{\sqrt{(\sigma_n(X))^2+s}} - \frac{2n^2{\uu}\|X\|_2}{\sqrt{0.9}\sqrt{(\sigma_n(X))^2+s}} \nonumber \\
&\ge& \frac{0.9}{\sqrt{(\sigma_n(X))^2+s}}\left(\sigma_n(X)-\frac{2}{0.9\sqrt{0.9}}\cdot n^2{\uu}\|X\|_2\right) \nonumber \\
&\ge& \frac{0.9}{\sqrt{(\sigma_n(X))^2+s}}(\sigma_n(X)-0.4\sigma_n(X)) \nonumber \\
&\ge& \frac{\sigma_n(X)}{2\sqrt{(\sigma_n(X))^2+s}}
= \frac{1}{2\sqrt{1+\alpha(\kappa_2(X))^2}}.
\end{eqnarray}
We have used the assumption~\eqref{eq:assumption_kappa} for the fourth inequality.
Together with \eqref{eq:hatQbound} we obtain 
\begin{equation}
\kappa_2(\hat{Q})=\frac{\|\hat{Q}\|_2}{\sigma_n(\hat{Q})} \le 2\sqrt{1+\alpha(\kappa_2(X))^2}\cdot\sqrt{3}.
\label{eq:CNreduction}
\end{equation}
\end{proof}
Theorem~\ref{eq:kap2} implies that, provided that $\alpha(\kappa_2(X))^2\gg 1$, 
\begin{equation}\label{eq:redcond}
\kappa_2(\hat{Q}) \lesssim 2\sqrt{3}\cdot\sqrt{\alpha}\kappa_2(X). 
\end{equation}
Thus applying one step of \shiftcholqr\ results in 
the condition number being reduced by about a factor $\sqrt{\alpha}=\frac{\sqrt{s}}{\|X\|_2}$.

\section{\shiftcholqr3: \shiftcholqr\ + CholeskyQR2}\label{sec:shiftandchol2}
We now discuss an algorithm for the QR factorization of an ill-conditioned  matrix 
that first uses \shiftcholqr, then runs CholeskyQR2. We refer to this algorithm as \shiftcholqr3, since it runs Cholesky QR (or its shifted variant) three times. 
The initial \shiftcholqr\ can thus be regarded as a preconditioning step that reduces the condition number so that CholeskyQR2 becomes applicable. 
We continue to assume~\eqref{eq:assumption_kappa}--\eqref{eq:assumption_n}; a particular case of interest is $\uu^{-\frac{1}{2}}<\kappa_2(X)<\uu^{-1}$. 

\subsection{Choice of shift $s$}\label{sec:choices}
We first discuss the choice of the shift $s$ for \shiftcholqr, which balances two requirements: 
\begin{itemize} 
\item $s$ should be as small as possible to maximize the condition number improvement~\eqref{eq:condimprove}.  
\item $s$ should be large enough so that the Cholesky factorization $\mbox{chol}(A + sI)$ runs to completion without numerically breaking down. 
\end{itemize} 
In addition to these, $s$ must satisfy Eq.~(\ref{eq:assumption_s}) for the error analysis in the previous section to be valid. 
 
To address the issue of breakdown we review  
Rump and Ogita's~\cite{rump2007super} error analysis for Cholesky factorization, which builds upon Demmel's early work~\cite{demmellawn14}.  
 
Let $A$ be a symmetric positive definite matrix. It is known~\cite[Thm.~2.3]{rump2007super} that  
$\mbox{chol}(A)$ succeeds numerically if the following holds:  
\begin{eqnarray}\label{cholruncomplete} 
\lambda_{n}(A) \ge \sum^{n}_{i=1}\frac{\gamma_{i+1}}{1-\gamma_{i+1}}\|A\|_2. 
\label{eq:Rump} 
\end{eqnarray} 
It has been shown in \cite{yanagisawa2014modified} that a sufficient condition for (\ref{eq:Rump}) to hold is 
\begin{eqnarray}\label{eq:delta1} 
\tilde{s}\ge c_{n+2}\uu\mathrm{tr}(A), \quad c_{n+2} &:=& \frac{(n + 2)}{1 - (n + 1)(n + 3)\uu} < 2.2(n+1).
\label{eq:tildes} 
\end{eqnarray}

In our context of the Cholesky factorization~\eqref{eq:shiftchol}, 
we need to take into account the error term $E_{1}$ in computing the matrix multiplication $A= X^{\top}X$.  
By  \eqref{eq:forward1} and Weyl's theorem we obtain a lower bound 
\begin{equation} 
\lambda_{n}(\hat{A})\ge 
\lambda_{n}(X^{\top}X) -  \norm{E_1}_{2}. 
 \label{spd_safe}  
\end{equation} 
This means that $\hat{A}$  may not be positive definite if  
$\lambda_{n}(X^{\top}X)\le \gamma_{m}n\norm{X}^{2}_{2}$.  
Accordingly, to apply formula (\ref{eq:tildes}), we must first shift $\hat{A}$ by $\gamma_{m}n\norm{X}^{2}_{2}$. Thus we conclude that a safe choice of $s$ to 
avoid numerical breakdown is 
\begin{equation} 
\tilde{s}:=\gamma_{m+1}n\norm{X}^{2}_{2}+c_{n+2}{\bf u}\mbox{tr}(\hat{A}+\gamma_{m+1} n\norm{X}^{2}_{2} I).
\label{eq:tildes2}
\end{equation}
By further taking into account (\ref{eq:assumption_s}), we have 
\begin{equation} 
s:=\max\left(11\{mn+n(n+1)\}{\bf u}\|X\|_2^2,\tilde{s}\right). 
\label{eq:finds2max} 
\end{equation} 
This expression can be simplified by evaluating $\tilde{s}$ using the results in the previous section. First, we note that $\hat{A}=X^{\top}X + E_1$ from~\eqref{eq:forward1}. Denoting the $j$th column vector of $X$ by $\tilde{\bf x}_j$, we can evaluate $\mbox{tr}(E_1)$ as
\begin{equation}\label{eq:trE1}
\mbox{tr}(E_1) \le \sum_{i=1}^n|E_1|_{ii} \le \sum_{i=1}^n\gamma_m|\tilde{\bf x}_i|^{\top}|\tilde{\bf x}_i| = \gamma_m\sum_{i=1}^n\|\tilde{\bf x}_i\|^2 = \gamma_m\|X\|_F^2 \le \gamma_m n\|X\|_2^2.
\end{equation}
On the other hand,
\begin{equation}\label{eq:trXTX}
\mbox{tr}(X^{\top}X) = \|X\|_F^2 \le n\|X\|_2^2.
\end{equation}
Plugging these into the right-hand side of~\eqref{eq:tildes2}, we can bound $\tilde{s}$ as
\begin{eqnarray}
\tilde{s} &\le& \gamma_{m+1}n\norm{X}_2^2 + 2.2(n+1){\bf u}(n\|X\|_2^2 + \gamma_m n\|X\|_2^2 + \gamma_{m+1}n\|X\|_2^2) \nonumber \\
&\le& 2.4\{mn+n(n+1)\}{\bf u}\|X\|_2^2,
\end{eqnarray}
which shows that the maximum in~\eqref{eq:finds2max} is attained by the first argument. Hence, in what follows we shall assume
\begin{equation}
s:=11\{mn+n(n+1)\}{\bf u}\|X\|_2^2.
\label{eq:finds2}
\end{equation}
In practice, since $\norm{X}_{2}$ is expensive to compute we can estimate it reliably using a norm estimator e.g. MATLAB's function {\tt normest}, or alternatively just replace it with $\norm{X}_{F}$, which results in a larger (more conservative) shift.

We now summarize our algorithm in pseudocode. 
Algorithm~\ref{alg12} 
blends \shiftcholqr\ and CholeskyQR2 in an adaptive manner, initially attempting to reduce the condition number using \shiftcholqr\ 
as in~\eqref{eq:redcond} so that CholeskyQR(2) becomes applicable. 
The shifts are introduced only when necessary, judged by whether or not the Cholesky factorization 
$\mbox{chol}(A^{(k)})$ breaks down. 
Our experiments suggest that Algorithm~\ref{alg12} may well be applicable even to extremely ill-conditioned matrices with possibly $\kappa_2(X)>\uu^{-1}$. 


 \begin{algorithm}[h]
 \caption{Iterated CholeskyQR for $X=QR$ with shifts when necessary.}
 \label{alg12}
 \begin{algorithmic}[1]
  \STATE Let $Q:=X,R:=I$
  \REPEAT
  \STATE $A:=Q^TQ$
  \STATE $\tilde{R}:=\mbox{chol}(A)$\quad //chol: Cholesky factorization
  \IF{$\mbox{chol}(A)$ breaks down}
  \STATE $s:=11\{mn+n(n+1)\}{\bf u}\|X\|_2^2$
\quad //introduce shift
  \STATE $\tilde R:=\mbox{chol}(A+sI)$
  \ENDIF
  \STATE $Q:=Q\tilde R^{-1}$, $R:=\tilde RR$
 \UNTIL{$\norm{Q^{T}Q-I}_{F}\le \sqrt{n}\uu$}
    \end{algorithmic}
 \end{algorithm}

In the analysis below, we focus on the case 
$\kappa_2(X)<\uu^{-1}$; to be precise, when 
$\uu^{-\frac{1}{2}}<\kappa_2(X)$ (so that CholeskyQR2 is inapplicable) and $\kappa_2(X)$ is bounded from above by~\eqref{eq:guarantee} given below. 
For such matrices, the algorithm provenly runs one \shiftcholqr, then CholeskyQR2 (thus executing three Cholesky factorizations). 
This computes a stable QR factorization, and we refer to this algorithm as \shiftcholqr3. For completeness we present its pseudocode in Algorithm~\ref{alg:shiftchol3}.
The last two lines represent CholeskyQR2 for the $Q$ obtained by the first \shiftcholqr. 

 \begin{algorithm}[h]
 \caption{\shiftcholqr3 for $X=QR$. 
}
 \label{alg:shiftchol3}
 \begin{algorithmic}[1]
  \STATE Let $Q:=X$
  \STATE $A:=Q^TQ$
  \STATE $s:=11\{mn+n(n+1)\}{\bf u}\|X\|_2^2$\quad //introduce shift
  \STATE $R:=\mbox{chol}(A+sI)$ \quad // \shiftcholqr
  \STATE $Q:=Q\tilde R^{-1}$
  \STATE $\tilde R:=\mbox{chol}(Q^TQ)$,  $Q:=Q\tilde R^{-1}$, $R:=\tilde RR$ // Cholesky QR
  \STATE $\tilde R:=\mbox{chol}(Q^TQ)$,  $Q:=Q\tilde R^{-1}$, $R:=\tilde RR$ // CholeskyQR2
    \end{algorithmic}
 \end{algorithm}

\subsection{When is thrice enough?}
Here we derive a condition on $\kappa_2(X)$ that guarantees that 
\shiftcholqr3 gives a numerically stable QR factorization of $X$. 

Recall that \eqref{eq:CNreduction} gives a bound for the condition number of $\hat{Q}$ obtained by \shiftcholqr: $\kappa_2(\hat Q) \le 2\sqrt{1+\alpha(\kappa_2(X))^2}\cdot\sqrt{3}$. As in~\eqref{eq:finds2}, to guarantee avoidance of breakdown we take $\alpha=11\{mn+n(n+1)\}{\bf u}$, so the condition number of $\hat Q$ is bounded as 
\begin{equation}
  \label{eq:hatqcond}
\kappa_2(\hat Q) \le 2\sqrt{3}\sqrt{1+11\{mn+n(n+1)\}{\bf u}(\kappa_2(X))^2}
\end{equation}
On the other hand, 
as shown in~\cite{yamamotoetna2015}, 
a sufficient condition for CholeskyQR2 to compute  a stable QR factorization of $\hat{Q}$ is 
\begin{equation}\label{eq:cholqr2gurantee}
\kappa_2(\hat{Q})\le\frac{1}{8\sqrt{\{mn+n(n+1)\}{\uu}}}. 
\end{equation}
Combining these facts, we obtain the following condition under which \shiftcholqr3 is guaranteed to compute a numerically stable QR factorization:
\[2\sqrt{3}\sqrt{1+11\{mn+n(n+1)\}{\bf u}(\kappa_2(X))^2}
\leq\frac{1}{8\sqrt{\{mn+n(n+1)\}{\uu}}}.\]
If $\kappa_2(X)>\uu^{-\frac{1}{2}}$, we have $1+11\{mn+n(n+1)\}{\bf u}(\kappa_2(X))^2 \simeq 11\{mn+n(n+1)\}{\bf u}(\kappa_2(X))^2$ and the condition can be simplified as
\begin{equation}\label{eq:guarantee}
\kappa(X) \leq \frac{{\uu^{-1}}}{96\{mn+n(n+1)\}}.
\end{equation}
Note that this ensures that the condition~\eqref{eq:assumption_kappa} for the error analysis in
Section~\ref{sec:stab} is automatically satisfied.

In practice, it often happens that \shiftcholqr3 (or more often the iterated Algorithm~\ref{alg12}) 
 computes the QR factorization for matrices with even larger condition numbers than indicated by~\eqref{eq:guarantee}. One explanation is that in the absence of roundoff errors, one iteration of \shiftcholqr\ reduces the condition number by a factor $\approx \uu^{\frac{1}{2}}$. However, we think that a rigorous convergence analysis in finite precision arithmetic would be possible only under some assumption on $\kappa_2(X)$, and that~\eqref{eq:guarantee} provides a sharp bound up to (at most) a low-degree polynomial in $m,n$.

We now examine the numerical stability of \shiftcholqr3 and show that it enjoys excellent stability both in orthogonality and backward error. 
Roughly, the result follows by combining the facts that (i) \shiftcholqr\ gives a $\hat Q$ with $\kappa_2(\hat Q)<\uu^{-\frac{1}{2}}$ 
with small backward error, and (ii) for matrices with condition number $<\uu^{-\frac{1}{2}}$, CholeskyQR2 computes a stable QR factorization of $\hat Q$ as shown in~\cite{yamamotoetna2015}. Below we make this statement precise. 

\subsection{Numerical stability of \shiftcholqr3}

\begin{theorem}\label{thm:maincholqr3}
  Let $X\in\mathbb{R}^{m\times n}$ be a matrix satisfying~\eqref{eq:assumption_kappa}--\eqref{eq:assumption_n} and 
~\eqref{eq:guarantee}. 
Then \shiftcholqr3 computes a QR factorization $X\approx \hat Q\hat R$ satisfying the orthogonality measure 
\begin{eqnarray}
\norm{\hat Q^{\top}\hat Q-I}_{F}\le6\{mn+n(n+1)\}\uu, \label{boundQ:Chol2}
\end{eqnarray}
and backward error 
  \begin{equation}    \label{eq:berr}
\frac{\|\hat{Q}\hat{R}-X\|_F}{\|X\|_2}
\le 15n^2{\uu}.     
  \end{equation}
\end{theorem}
\begin{proof}
The orthogonality measure of the output $\hat Q$ of \shiftcholqr3 is essentially exactly the same as that of CholeskyQR2, which is 
analyzed in detail in~\cite{yamamotoetna2015}. This is because the bound there applies to any matrix with condition number $\lesssim \uu^{-\frac{1}{2}}$.

We next establish~\eqref{eq:berr}. 
By~\eqref{eq:backward2X}, with the first \shiftcholqr\ executed in finite precision arithmetic we have 
\begin{equation}
X+\Delta X = \hat{Q}\hat{R}, 
\end{equation}
where $\|\Delta X\|_F$ is bounded as in (\ref{eq:DeltaXbound2}). 
We then apply CholeskyQR2 to 
$\hat{Q}$ to obtain the QR factorization $\hat{Q}=ZU$. 
In finite precision we have 
\begin{equation}
\hat{Q}+\Delta\hat{Q} = \hat{Z}\hat{U}, 
\label{eq:CholeskyQR2resid}
\end{equation}
where (see Appendix; this bound slightly improves~\cite{yamamotoetna2015})
\begin{equation}
\|\Delta\hat{Q}\|_F \le 5n^2{\uu}\|\hat{Q}\|_2.
\end{equation}
The upper triangular factor $S$ in the QR factorization of the 
original matrix $X$ is computed as
\begin{equation}
\hat{S}={\it fl}(\hat{U}\hat{R}) = \hat{U}\hat{R}+\Delta S.
\end{equation}
Here $\Delta S$ represents the forward error incurred in the matrix multiplication.

Summarizing, we can bound the overall backward error as 
\begin{eqnarray}
\|\hat{Z}\hat{S}-X\|_F
&=& \|\hat{Z}(\hat{U}\hat{R}+\Delta S)-\hat{Q}\hat{R}+\Delta X\|_F \nonumber \\
&=& \|\Delta\hat{Q}\hat{R}+\hat{Z}\Delta S+\Delta X\|_F \nonumber \\
&\le& \|\Delta\hat{Q}\|_F\|\hat{R}\|_2 + \|\hat{Z}\|_2\|\Delta S\|_F + \|\Delta X\|_F.
\label{eq:resid_shiftedCholesky2}
\end{eqnarray}
We now bound the terms in the right-hand side. 
For the first term, 
using \cite[Thm~3.5]{yamamotoetna2015} and (\ref{eq:hatQbound})
we can bound $\|\Delta\hat{Q}\|_F$ as 
\begin{equation}
\|\Delta\hat{Q}\|_F \le 5n^2{\uu}\|\hat{Q}\|_2 \le 5\sqrt{3}n^2{\uu}.
\end{equation}
To bound $\|\hat{R}\|_2$, we use (\ref{eq:hatR2bound}) to obtain 
\begin{equation}
\|\hat{R}\|_2 \le \sqrt{1.1}\|X\|_2.
\label{eq:hatRbound}
\end{equation}
We next bound the second term in~\eqref{eq:resid_shiftedCholesky2}. 
By \cite[Thm.~3.3]{yamamotoetna2015}, 
$\|\hat{Z}\|_2$ can be bounded as
\begin{equation}
\|\hat{Z}\|_2 \le \sqrt{1+6(mn{\uu}+n(n+1){\uu})} \le \sqrt{1+6\left(\frac{1}{64}+\frac{1}{64}\right)} = \frac{\sqrt{76}}{8}.
\label{eq:hatZnorm}
\end{equation}
Regarding $\|\Delta S\|_F$, 
using the general error bound for matrix multiplications 
$|\Delta S|\le\gamma_n|\hat{U}|\,|\hat{R}|$ we obtain 
\begin{eqnarray}
\|\Delta S\|_F \le \gamma_n\|\,|\hat{U}|\,|\hat{R}|\,\| \le \gamma_n\|\hat{U}\|_F\|\hat{R}\|_F \le n\gamma_n\|\hat{U}\|_2\|\hat{R}\|_2.
\end{eqnarray}
Here $\|\hat{R}\|_2$ can be bounded as in (\ref{eq:hatRbound}).
To bound $\|\hat{U}\|_2$, we recall (\ref{eq:CholeskyQR2resid})
and left-multiply $\hat{Z}^{\top}$ to obtain
\begin{equation}
\hat{Z}^{\top}(\hat{Q}+\Delta\hat{Q})=\hat{Z}^{\top}\hat{Z}\hat{U}. 
\end{equation}
Now, by \cite[Thm.~3.3]{yamamotoetna2015}, 
the eigenvalues of $\hat{Z}^{\top}\hat{Z}$ lie in the interval $[1-6(mn{\uu}+n(n+1){\uu}),1+6(mn{\uu}+n(n+1){\uu})]$, so it follows that 
\begin{eqnarray}
\|\hat{U}\|_2 &\le& \|(\hat{Z}^{\top}\hat{Z})^{-1}\|_2\|\hat{Z}^{\top}\|_2(\|\hat{Q}\|_2+\|\Delta\hat{Q}\|_2) \nonumber \\
&\le& \frac{1}{1-6(mn{\uu}+n(n+1){\uu})}\cdot\frac{\sqrt{76}}{8}(\sqrt{3}+5\sqrt{3}n^2{\uu}) \nonumber \\
&\le& \frac{1}{1-6\left(\frac{1}{64}+\frac{1}{64}\right)}\cdot\frac{\sqrt{76}}{8}\left(\sqrt{3}+5\sqrt{3}\cdot\frac{1}{64}\right) \le 2.6.
\end{eqnarray}
Finally, we can bound $\|\Delta X\|_F$ as in (\ref{eq:DeltaXbound2}). 

Combining the above bounds and substituting into 
(\ref{eq:resid_shiftedCholesky2}) yields 
\begin{eqnarray}
\|\hat{Z}\hat{S}-X\|_F
&\le& 5\sqrt{3}n^2{\uu}\cdot\sqrt{1.1}\|X\|_2 + \frac{\sqrt{76}}{8}\cdot n\cdot 1.02n{\uu}\cdot 2.6\cdot\sqrt{1.1}\|X\|_2 + 2n^2{\uu}\|X\|_2 \nonumber \\
&\le& 15n^2{\uu}\|X\|_2, 
\end{eqnarray}
as required.   
\end{proof}

\subsection*{Comparison with CGS2}\label{com_cgs2}
As we saw above, \eqref{eq:guarantee} is a sufficient condition for \shiftcholqr3 to work in finite precision arithmetic. 
This condition roughly requires that $\kappa_2(X)(mn+n^2){\uu}=O(1)$. 
Let us compare this with the analysis in \cite{giraud2005rounding} for the CGS2 algorithm, which shows that 
\begin{equation}
\kappa_2(X)m^2 n^3{\uu}=O(1)
\label{eq:CGS2assumption}
\end{equation}
is a sufficient condition for CGS2 to compute the QR factorization in a stable manner. 

Observe that 
(\ref{eq:CGS2assumption}) is much more stringent than \eqref{eq:guarantee}; indeed in large-scale computing in which $m,n\gg 1000$, with double precision  \eqref{eq:CGS2assumption} is unlikely to be satisfied even with well-conditioned $X$. 

This difference might appear to suggest \shiftcholqr3 is superior to CGS2 in terms of robustness, but we have not observed this in practice. We suspect that the difference is an artifact of the analysis, and the practical robustness 
of CGS2 and \shiftcholqr3 seem comparable. 
An advantage of \shiftcholqr3  is that it 
is rich in BLAS-3 operations, and offers ample opportunity for parallelization.

\section{Oblique inner product}\label{sec:oblique}
The (shifted) Cholesky QR algorithm is readily applicable to the QR decomposition in a non-standard inner product space 
$(x,y)_B = x^TBy$
defined via a symmetric positive definite matrix $B\in\mathbb{R}^{m\times m}$. The resulting Algorithm~\ref{alg1B} is almost identical to Algorithm~\ref{alg1}, except that $A$ is computed as $A=X^{\top}BX$ and the shift $s$ is chosen in a manner to be described below. 
 \begin{algorithm}[h]
 \caption{\shiftcholqr\ for $X=QR$,  $Q^TBQ=I_n$}
 \label{alg1B}
 \begin{algorithmic}[1]
\STATE $ A  = X^{\top}BX$ 
\STATE choose $s>0$
\STATE $R  =  \mbox{chol}(A + sI)$
\STATE $Q  =  XR^{-1}$
    \end{algorithmic}
 \end{algorithm}%

In this section, we examine the stability of Algorithm~\ref{alg1B}. The argument closely parallels that in Sections~\ref{sec:stab} and \ref{sec:shiftandchol2}, but new features arise that affect the bounds, in particular involving  $\sqrt{\kappa_2(B)}$. 

\subsection{Assumptions}\label{sec:stabB}
We make the following assumptions on $m$, $n$, $X$ and $B$. As in the case of standard inner product, the constants below are not of significant importance but chosen so that the analysis goes through. 
\begin{eqnarray}
&& 6n^2{\uu}\cdot\frac{\|X\|_2\sqrt{\|B\|_2}}{\sqrt{\sigma_n(X^{\top}BX)}}\cdot\sqrt{\kappa_2(B)} < 1, \label{eq:assumption_kappa_B} \\
&& m\sqrt{mn}{\uu} \le \frac{1}{64}, \label{eq:assumption_m_B} \\
&& n(n+1){\uu} \le \frac{1}{64}, \label{eq:assumption_n_B}\\
&& 11\{2m\sqrt{mn}+n(n+1)\}{\uu}\|X\|_2^2\|B\|_2 \le s\le \frac{1}{100}\|X\|_2^2\|B\|_2. \label{eq:assumption_s_B} 
\end{eqnarray}
The assumption~\eqref{eq:assumption_kappa_B} roughly demands that $X$ is not too ill-conditioned relative to $\uu$. 
As before, \eqref{eq:assumption_m_B} and \eqref{eq:assumption_n_B} require that the matrix dimensions $m,n$ are small compared with the precision $\uu^{-1}$.
In \eqref{eq:assumption_kappa_B}, we can use a simpler assumption
\begin{equation}
6n^2{\bf u}\kappa_2(X)\kappa_2(B)<1,
\end{equation}
but \eqref{eq:assumption_kappa_B} is less stringent. 

\subsection{Preparations}
We denote the computed results of
 \shiftcholqr\ in an oblique inner product, accounting for the numerical errors, as
\begin{eqnarray}
\hat{A} &=& X^{\top}BX + E_1, \label{eq:forward1_B} \\
\hat{R}^{\top}\hat{R} &=& \hat{A} + sI + E_2 \;=\; X^{\top}BX + sI + E_1 + E_2,  \label{eq:shiftchol_B} \\
\hat q_{i}^{\top}&=& x_{i}^{\top}(\hat{R}+\Delta \hat R_{i})^{-1}\quad (i=1,2,\dots m). \label{eq:backward2_B}
\end{eqnarray}
$\hat{q}_{i}^{\top}$, $x_{i}^{\top}$ are the $i$th rows of $Q$ and $\hat{X}$, respectively. 
$E_{1}$ is the matrix-matrix multiplication error in the computation of the Gram matrix $X^{\top}BX$, and $E_{2}$ is the $\hat{A}$ backward error incurred when computing the Cholesky factorization. $\Delta \hat R_{i}$ is the backward error involved in the solution of the linear system $q_{i}^{\top}\hat{R}=x_{i}^{\top}$.
Equation \eqref{eq:backward2_B} can be rewritten as
\begin{equation}
\hat q_i^{\top}\hat{R} = x_i^{\top}+\Delta x_i^{\top},
\end{equation}
where
\begin{equation}
\Delta x_i^{\top} = -\hat{q}_i^{\top}\Delta\hat{R}_i.
\label{eq:deltaxdeltaR_B}
\end{equation}
Let $\Delta X=
\bigg[  \begin{smallmatrix}
\Delta x_{1}^{\top}  \\
\vdots\\
\Delta x_{m}^{\top}      
  \end{smallmatrix}\bigg]
$ be the matrix obtained by stacking up the row vectors $\Delta x_{i}^{\top}$. Then 
\begin{equation}
X+\Delta X=\hat{Q}\hat{R},
\label{eq:backward2X_B}
\end{equation}
showing that $\Delta X $ is the residual.

Now we give bounds on $\|E_1\|_2$, $\|E_2\|_2$, $\|\hat{R}^{-1}\|_2$, $\|B^{\frac{1}{2}}X\hat{R}^{-1}\|_2$, $\|\Delta\hat{R}_i\|_2$, $\|\hat{Q}\|_2$ and $\|\Delta X\|_F$ as in the case of standard inner product. 

\subsubsection{Bounding $\|E_1\|_2$, $\|E_2\|_2$, $\|\hat{R}^{-1}\|_2$, $\|B^{\frac{1}{2}}X\hat{R}^{-1}\|_2$ and $\|\Delta\hat{R}_i\|_2$}
Using the standard error analysis of matrix-matrix multiplication and Cholesky factorization \cite{Higham:2002:ASNA} and the assumptions \eqref{eq:assumption_m_B} and  \eqref{eq:assumption_n_B}, we can bound $\|E_1\|_2$ and $\|E_2\|_2$ as
\begin{eqnarray}
\norm{E_{1}}_{2} &\le& 2.2m\sqrt{mn}{\bf u}\norm{X}^{2}_{2}\|B\|_2, 
\label{err_1_B} \\
\norm{E_{2}}_{2} &\le& 1.1n(n+1){\bf u}\|X\|_2^2\|B\|_2.
\label{err_2_B}
\end{eqnarray}
See \cite{yamamotojsiam2015} for details. From the assumption \eqref{eq:assumption_s_B} on $s$, these bounds ensure that
\begin{equation}
\|E_1\|_2+\|E_2\|_2 \le (2.2m\sqrt{mn}+1.1n(n+1)){\uu}\|X\|_2^2\|B\|_2 \le 0.1s.
\label{eq:E1E2bound_B}
\end{equation}
Combining this with \eqref{eq:shiftchol_B} and using Weyl's theorem~\cite[Sec.~8.6.2]{golubbook4th}, we obtain a bound on $\|\hat{R}^{-1}\|_2$:
\begin{equation}
\|\hat{R}^{-1}\|_2 \le \frac{1}{\sqrt{\sigma_n(X^{\top}BX)+0.9s}}.
\label{eq:Rinvbound_B}
\end{equation}
The bound on $\|B^{\frac{1}{2}}X\hat{R}^{-1}\|_2$ can be derived as follows. First, note that from \eqref{eq:shiftchol_B},
\begin{equation}
(B^{\frac{1}{2}}X\hat{R}^{-1})^{\top}(B^{\frac{1}{2}}X\hat{R}^{-1}) = I-\hat{R}^{-\top}(sI+E_1+E_2)\hat{R}^{-1}.
\end{equation}
Using \eqref{eq:E1E2bound_B} and \eqref{eq:Rinvbound_B}, we have 
\begin{eqnarray}
\|B^{\frac{1}{2}}X\hat{R}^{-1}\|_2 &\le& \sqrt{1+\|\hat{R}^{-1}\|_2^2(s+\|E_1\|_2+\|E_2\|_2)} \nonumber \\
&\le& \sqrt{1+\frac{1.1s}{\sigma_n(X^{\top}BX)+0.9s}} \nonumber \\
&\le& \sqrt{1+\frac{1.1}{0.9}} \le 1.5.
\label{eq:XRinvbound_B}
\end{eqnarray}
The bound on $\|\Delta\hat{R}_i\|_2$ can be obtained from the standard error analysis of backward substitution \cite[Thm.~8.5] {Higham:2002:ASNA} as
\begin{equation}
\|\Delta\hat{R}_i\|_2 \le \|\abs{\Delta\hat{R}_i}\|_F \le \gamma_n \sqrt{n}\|\hat{R}\|_2.
\label{eq:DeltaRbound_B}
\end{equation}
From~\eqref{eq:shiftchol_B}, \eqref{eq:assumption_s_B},  and \eqref{eq:E1E2bound_B} we obtain 
\begin{equation}
\|\hat{R}\|_2^2 \le \|X\|_2^2\|B\|_2+s+\|E_1\|_2+\|E_2\|_2 \le \|X\|_2^2\|B\|_2 + 1.1s \le 1.1\|X\|_2^2\|B\|_2.
\label{eq:hatR2bound_B}
\end{equation}
Substituting this into \eqref{eq:DeltaRbound_B} gives 
\begin{equation}
\|\Delta\hat{R}_i\|_2 \le 1.02n{\uu}\cdot\sqrt{n}\cdot\sqrt{1.1}\|X\|_2\sqrt{\|B\|_2} \le 1.1n\sqrt{n}{\uu}\|X\|_2\sqrt{\|B\|_2}.
\label{eq:DeltaRibound_B}
\end{equation}

\subsubsection{Bounding $\|\Delta X\|_2$ roughly}
Here we give a rough bound for $\|\Delta X\|_F$ and prove that $\|\Delta X\|_F=O({\uu}\|X\|_2^2\sqrt{\|B\|_2}/\sqrt{s})$.
This will be insufficient for proving our main result, Theorem \ref{eq:kap2_B}, 
for which we will need $\|\Delta X\|_F=O({\uu}\|X\|_2\sqrt{\kappa_2(B)})$, which we will prove later 
after having obtained a bound for 
$\|\hat{Q}\|_2$. We shall proceed in the following steps. 
\begin{enumerate}
\item Derive the ``rough'' bound $\|\Delta X\|_F=O({\uu}\|X\|_2^2\sqrt{\|B\|_2}/\sqrt{s})$. 
\item Use above to show $\|B^{\frac{1}{2}}\hat{Q}\|_2=O(1)$ and $\|\hat{Q}\|_2=O(1/\sqrt{\sigma_n(B)})$. 
\item Use above and \eqref{eq:deltaxdeltaR_B} to prove the ``tight'' bound $\|\Delta X\|_F=O({\uu}\|X\|_2\sqrt{\kappa_2(B)})$. 
\item Use above to prove Theorem \ref{eq:kap2_B}. 
\end{enumerate}

To establish the first statement, we express $\Delta x_i^{\top}$ in terms of $x_i^{\top}$, $\hat{R}$ and $\Delta\hat{R}_i$. Substituting (\ref{eq:backward2_B}) into (\ref{eq:deltaxdeltaR_B}), we have
\begin{equation}
\Delta x_i^{\top} = x_i^{\top}(\hat{R}+\Delta\hat{R}_i)^{-1}\Delta\hat{R}_i = x_i^{\top}(I+\hat{R}^{-1}\Delta\hat{R}_i)^{-1}\hat{R}^{-1}\Delta\hat{R}_i = x_i^{\top}\Delta\breve{R}_i,
\end{equation}
where
\begin{equation}
\Delta\breve{R}_i = (I+\hat{R}^{-1}\Delta\hat{R}_i)^{-1}\hat{R}^{-1}\Delta\hat{R}_i.
\end{equation}
Here, $\|\hat{R}^{-1}\Delta\hat{R}_i\|_2$ can be bounded using (\ref{eq:Rinvbound_B}) and (\ref{eq:DeltaRibound_B}) as
\begin{eqnarray}
\|\hat{R}^{-1}\Delta\hat{R}_i\|_2 \le \|\hat{R}^{-1}\|_2\|\Delta\hat{R}_i\|_2
&\le& \frac{1.1n\sqrt{n}{\uu}\|X\|_2\sqrt{\|B\|_2}}{\sqrt{\sigma_n(X^{\top}BX)+0.9s}} \nonumber \\
&\le& \frac{1.1n\sqrt{n}{\uu}\|X\|_2\sqrt{\|B\|_2}}{\sqrt{0.9\cdot11n(n+1){\uu}\|X\|_2^2\|B\|_2}} \nonumber \\
&\le& \sqrt{\frac{1.21}{9.9}\cdot n{\bf u}} \le 0.05.
\end{eqnarray}
Thus, we can rewrite $\Delta\breve{R}_i$ using the Neumann expansion as $\Delta\breve{R}_i=\sum_{k=1}^{\infty}(\hat{R}^{-1}\Delta\hat{R}_i)^k$. Hence,
\begin{eqnarray}
\|\Delta\breve{R}_i\|_2 &\le& \sum_{k=1}^{\infty}(\|\hat{R}^{-1}\|_2\|\Delta\hat{R}_i\|_2)^k \nonumber \\
&=& \frac{\|\hat{R}^{-1}\|_2\|\Delta\hat{R}_i\|_2}{1-\|\hat{R}^{-1}\|_2\|\Delta\hat{R}_i\|_2} \nonumber \\
&\le& \frac{1}{1-0.05}\cdot\frac{1.1n\sqrt{n}{\uu}\|X\|_2\sqrt{\|B\|_2}}{\sqrt{\sigma_n(X^{\top}BX)+0.9s}} \le \frac{1.2n\sqrt{n}{\uu}\|X\|_2\sqrt{\|B\|_2}}{\sqrt{\sigma_n(X^{\top}BX)+0.9s}}.
\label{eq:Rbrevebound_B}
\end{eqnarray}
Together with the fact $\|\Delta x_i^{\top}\|\le\|x_i^{\top}\|\,\|\Delta\breve{R}_i\|_2$, 
we can bound 
$\|\Delta X\|_F$ as 
\begin{equation}
\|\Delta X\|_F = \sqrt{\sum_{i=1}^m\|\Delta x_i^{\top}\|^2} \le \sqrt{\sum_{i=1}^m\|x_i^{\top}\|^2}\cdot\max_{1\le i\le m}\|\Delta\breve{R}_i\|_2 \le \frac{1.2n^2{\uu}\|X\|_2^2\sqrt{\|B\|_2}}{\sqrt{\sigma_n(X^{\top}BX)+0.9s}}, 
\label{eq:DeltaXbound_B}
\end{equation}
where we used $\sqrt{\sum_{i=1}^m\|x_i^{\top}\|^2} = \|X\|_F\le \sqrt{n}\|X\|_2$ for the last inequality.

\subsubsection{Bounding $\|\hat{Q}\|_2$}
We now proceed to bound $\|\hat{Q}\|_2$. 
\begin{lemma}
\label{lemma51}
Suppose that $X\in\mathbb{R}^{m\times n}$
with $m\ge n$ satisfies~\eqref{eq:assumption_m_B} and~\eqref{eq:assumption_n_B}. Then, the matrix $\hat{Q}$ obtained by applying the \shiftcholqr\ algorithm in floating-point arithmetic to $X$ satisfies
$$
\|\hat{Q}^{\top}B\hat{Q}-I\|_2 < 2, 
$$
hence 
\begin{equation}
\|B^{\frac{1}{2}}\hat{Q}\|_2<\sqrt{3}.
\label{eq:hatQbound_B}
\end{equation}
Moreover,
\begin{equation}
\|\hat{Q}\|_2 \le \frac{\sqrt{3}}{\sqrt{\sigma_n(B)}}.
\label{eq:hatQbound2_B}
\end{equation}

\end{lemma}

\begin{proof}
We have 
\begin{eqnarray}
\hat{Q}^{\top}B\hat{Q} &=& \hat{R}^{-\top}(X+\Delta X)^{\top}B(X+\Delta X)\hat{R}^{-1} \nonumber \\
&=& I - \hat{R}^{-\top}(sI+E_1+E_2)\hat{R}^{-1} + (X\hat{R}^{-1})^{\top}B\Delta X\hat{R}^{-1} \nonumber \\
&& \quad\quad  + \hat{R}^{-\top}\Delta X^{\top}B(X\hat{R}^{-1}) + \hat{R}^{-\top}\Delta X^{\top}B\Delta X\hat{R}^{-1}. \nonumber
\end{eqnarray}
Thus we can bound $\|\hat{Q}^{\top}B\hat{Q}-I\|_2$ as 
\begin{eqnarray}
\|\hat{Q}^{\top}B\hat{Q}-I\|_2 
&\le& \|\hat{R}^{-1}\|_2^2(s+\|E_1\|_2+\|E_2\|_2) \nonumber \\
&& + 2\|\hat{R}^{-1}\|_2\sqrt{\|B\|_2}\,\|B^{\frac{1}{2}}X\hat{R}^{-1}\|_2\|\Delta X\|_F \nonumber \\
&& + \|\hat{R}^{-1}\|_2^2\|\Delta X\|_F^2\|B\|_2.
\label{eq:Th1-2_B}
\end{eqnarray}

The first term of~\eqref{eq:Th1-2_B} can be bounded as
\begin{equation}
\|\hat{R}^{-1}\|_2^2(sI+\|E_1\|_2+\|E_2\|_2) \le \frac{1.1s}{\sigma_n(X^{\top}BX)+0.9s} \le \frac{1.1}{0.9}.
\end{equation}
and for the second term in \eqref{eq:Th1-2_B}, using \eqref{eq:Rinvbound_B}, \eqref{eq:XRinvbound_B} and \eqref{eq:DeltaXbound_B} we obtain 
\begin{eqnarray}
&& 2\|\hat{R}^{-1}\|_2\sqrt{\|B\|_2}\,\|B^{\frac{1}{2}}X\hat{R}^{-1}\|_2\|\Delta X\|_F \nonumber \\
&&\le 2\cdot\frac{1}{\sqrt{\sigma_n(X^{\top}BX)+0.9s}}\cdot\sqrt{\|B\|_2}\cdot 1.5\cdot\frac{1.2n^2{\uu}\|X\|_2^2\sqrt{\|B\|_2}}{\sqrt{\sigma_n(X^{\top}BX)+0.9s}} \nonumber \\
&&\le \frac{2\cdot 1.5\cdot 1.2\cdot\frac{1}{11}s}{0.9s}=\frac{4}{11}.
\end{eqnarray}
For the third term in \eqref{eq:Th1-2_B}, from 
\eqref{eq:Rinvbound_B} and \eqref{eq:DeltaXbound_B}
\begin{eqnarray}
\|\hat{R}^{-1}\|_2^2\|\Delta X\|_F^2\|B\|_2
&\le& \frac{1}{\sigma_n(X^{\top}BX)+0.9s}\cdot\frac{(1.2n^2{\uu}\|X\|_2^2\sqrt{\|B\|_2})^2}{\sigma_n(X^{\top}BX)+0.9s}\cdot\|B\|_2 \nonumber \\
&\le& \frac{(1.2\cdot\frac{1}{11}s)^2}{(0.9s)^2} = \frac{16}{1089}.
\end{eqnarray}
Summarizing, we can bound the right-hand side of \eqref{eq:Th1-2_B} as 
$\|\hat{Q}^{\top}B\hat{Q}-I\|_2 < 2, $
as required. 

To derive \eqref{eq:hatQbound2_B}, let $B=VDV^{\top}$ and $\hat{Q}^{\top}B\hat{Q}=U\Lambda U^{\top}$ be the symmetric eigenvalue decompositions of $B$ and $\hat{Q}^{\top}B\hat{Q}$, respectively. Then, from $\hat{Q}^{\top}VDV^{\top}\hat{Q}=U\Lambda U^{\top}$, we have
\begin{equation}
\Lambda^{-\frac{1}{2}}U^{\top}\hat{Q}^{\top}VDV^{\top}\hat{Q}U\Lambda^{-\frac{1}{2}}=I.
\end{equation}
Hence there exists an orthogonal matrix $W$ such that
\begin{equation}
D^{\frac{1}{2}}V^{\top}\hat{Q}U\Lambda^{-\frac{1}{2}}=W.
\end{equation}
Noting that $\|\Lambda\|_2<3$ and $\|D^{-1}\|_2=(\sigma_n(B))^{-1}$, we can bound $\|\hat{Q}\|_2$ as
\begin{equation}
\|\hat{Q}\|_2 \le \|D^{-\frac{1}{2}}\|_2\|\Lambda^{\frac{1}{2}}\|_2 \le \sqrt{3}/\sqrt{\sigma_n(B)}.
\end{equation}
\end{proof}

\subsection{Bounding the residual}
We now bound the residual.
\begin{lemma}\label{lemma3_B}
Suppose that $X\in\mathbb{R}^{m\times n}$
with $m\ge n$ satisfies~\eqref{eq:assumption_m_B} and~\eqref{eq:assumption_n_B}. 
Then, the matrices $\hat{Q},\hat{R}$ obtained by Algorithm~\ref{alg1B} 
in floating-point arithmetic to $X$ satisfies
  \begin{equation}
    \label{eq:Bres1_B}
\frac{\norm{\hat{Q}\hat{R}-X}_{F}}{\norm{X}_{2}}
\leq 
2n^{2}{\uu}\sqrt{\kappa_2(B)},
  \end{equation}
and 
  \begin{equation}
    \label{eq:Bres2_B}
\norm{\hat{Q}\hat{R}-X}_{F}\leq \gamma_n \sqrt{n}\|\hat Q\|_F\|\hat R\|_2.    
  \end{equation}
\end{lemma}
\begin{proof}
First note that
\begin{equation}  \label{eq:qrx_B}
\norm{\hat q_{i}^{\top}\hat{R} -x_{i}^{\top}}=\norm{\hat q_{i}^{\top}\hat{R} -\hat q_{i}^{\top}(\hat{R} +\Delta \hat{R} _{i})}\le \norm{\hat q_{i}^{\top}\Delta \hat{R} _{i}}\le \norm{\hat q_{i}^{\top}}\norm{\Delta \hat{R} _{i}}_2.
\end{equation}
Substituting \eqref{eq:DeltaRibound_B} into this gives
\begin{equation}
\|\hat{q}_i^{\top}\hat{R}-x_i^{\top}\| \le \|\hat{q}_i^{\top}\|\cdot 1.1n\sqrt{n}{\uu}\|X\|_2\sqrt{\|B\|_2}.
\end{equation}
On the other hand, from~\eqref{eq:hatQbound2_B} we have 
\begin{equation}
\|\hat{Q}\|_F < \sqrt{3n}/\sqrt{\sigma_n(B)},
\end{equation}
so it follows that 
\begin{eqnarray}
\|\Delta X\|_F &=& \|\hat{Q}\hat{R}-X\|_F = \sqrt{\sum_{i=1}^m\|\hat{q}_i^{\top}\hat{R}-x_i^{\top}\|^2} \nonumber \\
&\le& \sqrt{\sum_{i=1}^m\|\hat{q}_i^{\top}\|^2}\cdot 1.1n\sqrt{n}{\uu}\|X\|_2\sqrt{\|B\|_2} = \|\hat{Q}\|_F\cdot 1.1n\sqrt{n}{\uu}\|X\|_2\sqrt{\|B\|_2} \nonumber \\
&\le& 2n^2{\uu}\|X\|_2\sqrt{\kappa_2(B)}.
\label{eq:residual_B}
\end{eqnarray}
To obtain the second bound in the statement, we use~\eqref{eq:DeltaRbound_B} in~\eqref{eq:qrx_B} to obtain
$\norm{\hat q_{i}^{\top}\hat{R} -x_{i}^{\top}}\leq 
\gamma_n \sqrt{n}\norm{\hat q_{i}^{\top}}\norm{\hat{R}}_2$, hence 
\[\norm{\hat{Q}\hat{R}-X}_{F}^2
\leq \gamma_n^2 n\sum_{i=1}^m\|\hat q_i^\top \|^2\|\hat R\|_2^2
=\gamma_n^2 n\|\hat Q\|_F^2\|\hat R\|_2^2
,\]
as required.
\end{proof}

Lemma~\ref{lemma3_B}, in particular~\eqref{eq:Bres2_B}, shows that \shiftcholqr\ gives optimal residual up to a factor bounded by a low-degree polynomial of $m,n$.

\subsection{Main result}
We are now ready to state the main result of the section, which bounds the quantity $\|\hat{Q}\|_2\sqrt{\|B\|_2}/\sqrt{\sigma_n(\hat{Q}^{\top}B\hat{Q})}$. Note that this is a $B$-orthonormality measure for $\hat{Q}$, reducing to $\kappa_2(\hat{Q})$ when $B=I$. 

\begin{theorem} \label{eq:kap2_B}
With one step of \shiftcholqr\ 
  applied in double precision arithmetic 
to $X$ satisfying~\eqref{eq:assumption_kappa_B}--\eqref{eq:assumption_n_B} 
with
  shift $s$ satisfying~\eqref{eq:assumption_s_B}, 
 we obtain $\hat Q$ with 
(defining $\alpha = \frac{s}{\|X\|_2^2\|B\|_2}$)
\begin{equation}
\label{eq:condimprove_B}
\frac{\|\hat{Q}\|_2\sqrt{\|B\|_2}}{\sqrt{\sigma_n(\hat{Q}^{\top}B\hat{Q})}} \leq 
2\sqrt{3}\cdot\sqrt{1+\alpha\cdot\frac{\|X\|_2^2\|B\|_2}{\sigma_n(X^{\top}BX)}}\cdot\sqrt{\kappa_2(B)}
\end{equation}
\end{theorem}

\begin{proof}  
Recall from~\eqref{eq:hatQbound2_B} that $\|\hat{Q}\|_2<\sqrt{3}/\sqrt{\sigma_n(B)}$. The remaining task is to bound $\sigma_n(B^{\frac{1}{2}}\hat{Q})$ from below. Note that $B^{\frac{1}{2}}\hat{Q}= B^{\frac{1}{2}}X\hat{R}^{-1}+B^{\frac{1}{2}}\Delta X\hat{R}^{-1}$ from \eqref{eq:backward2X_B}. Using Weyl's theorem gives
\begin{equation}  \label{eq:weylg_B}
\sigma_{n}(B^{\frac{1}{2}}\hat{Q}) \ge \sigma_{n}(B^{\frac{1}{2}}X\hat{R}^{-1})-\norm{B^{\frac{1}{2}}\Delta X \hat{R}^{-1}}_{2}.    
\end{equation}
Using \eqref{eq:Rinvbound_B} and \eqref{eq:residual_B} we obtain
\begin{equation}
  \|B^{\frac{1}{2}}\Delta X\hat{R}^{-1}\|_2 \le \sqrt{\|B\|_2}\|\Delta X\|_F\|\hat{R}^{-1}\|_2 \le \frac{2n^2{\uu}\|X\|_2\sqrt{\|B\|_2}\sqrt{\kappa_2(B)}}{\sqrt{\sigma_n(X^{\top}BX)+0.9s}}.
  \label{eq:DeltaXRinvbound_B}
\end{equation}
Note that this is $O(\uu^{\frac{1}{2}}\sqrt{\kappa_2(B)})$ when we regard low-degree polynomials in $m,n$ as constants.

We next bound $\sigma_n(B^{\frac{1}{2}}X\hat{R}^{-1})$ from below.
We start with the equation
\begin{equation}
\hat{R}^{-\top}(X^{\top}BX+sI)\hat{R}^{-1}
= I - \hat{R}^{-\top}(E_1+E_2)\hat{R}^{-1}.
\end{equation}
Let
$B^{\frac{1}{2}}X=U\Sigma V^{\top}$ by the SVD. Then for a diagonal matrix $G$, we
can write
\begin{equation} X^{\top}BX + sI =
  (U(\Sigma+G)V^{\top})^{\top}(U(\Sigma+G)V^{\top}).
\end{equation}
Indeed the left-hand side is $V(\Sigma^2+sI)V^{\top}$ and the
right-hand side $V(\Sigma+G)^2V^{\top}$, so we can take
\begin{equation} G=(\Sigma^2+sI)^{\frac{1}{2}}-\Sigma = {\rm
    diag}\left(\sqrt{\sigma_{i}(X^{\top}BX)+s}-\sqrt{\sigma_{i}(X^{\top}BX)}\right) \end{equation} Now
setting $T=U(\Sigma+G)V^{\top}\hat{R}^{-1}$ we have
\begin{equation}
T^{\top}T = \hat{R}^{-\top}V(\Sigma+G)^2V^{\top}\hat{R}^{-1} = \hat{R}^{-\top}(X^{\top}BX+sI)\hat{R}^{-1}= I -
  \hat{R}^{-\top}(E_1+E_2)\hat{R}.
\end{equation}
We next bound the singular values of $T$. By \eqref{eq:E1E2bound_B} and
\eqref{eq:Rinvbound_B} we have
\begin{equation}
  \|\hat{R}^{-\top}(E_1+E_2)\hat{R}^{-1}\|_2 \le \|\hat{R}^{-1}\|_2^2(\|E_1\|_2+\|E_2\|_2) \le \frac{0.1s}{\sigma_n(X^{\top}BX)+0.9s} \le \frac{1}{9},
\end{equation}
so
$\sigma_i(T)\in[\sqrt{1-\frac{1}{9}},\sqrt{1+\frac{1}{9}}]\subseteq
[0.9,1.1]$. Therefore, letting $T = U'(I+E')V'$ be the SVD where $E'$ is diagonal, we have
\begin{equation} T =
  U^{\prime}(I+E^{\prime})V^{\prime\top}=U^{\prime}V^{\prime\top}(I+V^{\prime}E^{\prime}V^{\prime\top})
  = W(I+E), \label{eq:IplusE_B} \end{equation}
where $W=U^{\prime}V^{\prime\top}$ has orthonormal columns, and 
$\|E\|_2=\|V^{\prime}E^{\prime}V^{\prime\top}\|_2=\|E^{\prime}\|_2\le 0.1$.

Now plugging into~\eqref{eq:IplusE_B} the definition of $T$ gives 
\begin{eqnarray} 
W(I+E) &=&U(\Sigma+G)V^{\top}\hat{R}^{-1}.
\end{eqnarray}
Recalling that $B^{\frac{1}{2}}X=U\Sigma V^\top$, 
we have $\sigma_i(B^{\frac{1}{2}}X\hat{R}^{-1})=\sigma_i(\Sigma V^\top\hat{R}^{-1})$
and 
\begin{eqnarray}
\Sigma V^\top \hat{R}^{-1} &=&
  \Sigma(\Sigma+G)^{-1} U^{\top}W(I+E) \nonumber \\ &=& {\rm
    diag}\left(\frac{\sqrt{\sigma_i(X^{\top}BX)}}{\sqrt{\sigma_i(X^{\top}BX)+s}}\right)U^{\top}W(I+E).
\end{eqnarray}
Using the general inequality for singular values of matrix products~$\sigma_{\min}(AB)\ge\sigma_{\min}(A)\sigma_{\min}(B)$ (applicable if $A$ or $B$ is square) we obtain
\begin{equation}
\sigma_n(B^{\frac{1}{2}}X\hat{R}^{-1}) \ge \frac{\sqrt{\sigma_n(X^{\top}BX)}}{\sqrt{\sigma_n(X^{\top}BX)+s}}\cdot 0.9.
\label{eq:XRinvbound2_B}
\end{equation}
By inserting this and~\eqref{eq:DeltaXRinvbound_B} into \eqref{eq:weylg_B}, we obtain
\begin{eqnarray}
\sigma_n(B^{\frac{1}{2}}\hat{Q})
&\ge& \frac{0.9\sqrt{\sigma_n(X^{\top}BX)}}{\sqrt{\sigma_n(X^{\top}BX)+s}} - \frac{2n^2{\uu}\|X\|_2\sqrt{\|B\|_2}\sqrt{\kappa_2(B)}}{\sqrt{\sigma_n(X^{\top}BX)+0.9s}} \nonumber \\
&\ge& \frac{0.9\sqrt{\sigma_n(X^{\top}BX)}}{\sqrt{\sigma_n(X^{\top}BX)+s}} - \frac{2n^2{\uu}\|X\|_2\sqrt{\|B\|_2}\sqrt{\kappa_2(B)}}{\sqrt{0.9}\sqrt{\sigma_n(X^{\top}BX)+s}} \nonumber \\
&=& \frac{0.9}{\sqrt{\sigma_n(X^{\top}BX)+s}}\left(\sqrt{\sigma_n(X^{\top}BX)}-\frac{2}{0.9\sqrt{0.9}}\cdot n^2{\uu}\|X\|_2\sqrt{\|B\|_2}\sqrt{\kappa_2(B)}\right) \nonumber \\
&\ge& \frac{0.9}{\sqrt{\sigma_n(X^{\top}BX)+s}}\left(\sqrt{\sigma_n(X^{\top}BX)}-0.4\sqrt{\sigma_n(X^{\top}BX)}\right) \nonumber \\
&\ge& \frac{\sqrt{\sigma_n(X^{\top}BX)}}{2\sqrt{\sigma_n(X^{\top}BX)+s}}
= \frac{1}{2\sqrt{1+\alpha\cdot\frac{\|X\|_2^2\|B\|_2}{\sigma_n(X^{\top}BX)}}},
\end{eqnarray}
where we used \eqref{eq:assumption_kappa_B} in the fourth inequality and the definition of $\alpha$ in the last equality.

Together with \eqref{eq:hatQbound2_B} we obtain 
\begin{eqnarray}
\frac{\|\hat{Q}\|_2\sqrt{\|B\|_2}}{\sqrt{\sigma_n(\hat{Q}^{\top}B\hat{Q})}}
&\le& 2\sqrt{1+\alpha\cdot\frac{\|X\|_2^2\|B\|_2}{\sigma_n(X^{\top}BX)}}\cdot\frac{\sqrt{3}}{\sqrt{\sigma_n(B)}}\cdot\sqrt{\|B\|_2} \nonumber \\
&=& 2\sqrt{3}\cdot\sqrt{1+\alpha\cdot\frac{\|X\|_2^2\|B\|_2}{\sigma_n(X^{\top}BX)}}\cdot\sqrt{\kappa_2(B)}.
\label{eq:CNreduction_B}
\end{eqnarray}
\end{proof}

Thus we conclude that, provided that $\alpha\|X\|_2^2\|B\|_2/\sigma_n(X^{\top}BX) \gg 1$, 
\begin{equation}
\frac{\|\hat{Q}\|_2\sqrt{\|B\|_2}}{\sqrt{\sigma_n(\hat{Q}^{\top}B\hat{Q})}} \lesssim 2\sqrt{3\alpha\kappa_2(B)}\cdot\frac{\|X\|_2\sqrt{\|B\|_2}}{\sqrt{\sigma_n(X^{\top}BX)}}.
\end{equation}
Thus applying one step of \shiftcholqr\ results in 
the condition number being reduced by about a factor $\sqrt{\alpha\kappa_2(B)}=\frac{\sqrt{s\kappa_2(B)}}{\|X\|_2\sqrt{\|B\|_2}}$.

\subsection{\shiftcholqr3\ in oblique inner product}
We now describe the analogue of \shiftcholqr3\ for an oblique inner product space. 
We continue to assume \eqref{eq:assumption_kappa_B}--\eqref{eq:assumption_s_B}; a particular case of interest is $\uu^{-\frac{1}{2}}<{\|\hat{Q}\|_2\sqrt{\|B\|_2}}/{\sqrt{\sigma_n(\hat{Q}^{\top}B\hat{Q})}} <\uu^{-1}$. 

\paragraph{Choice of shift $s$}
As in Section~\ref{sec:choices}, we need to take into account the error term $E_{1}$ in \eqref{eq:forward1_B}, incurred in the computation of $A=X^{T}BX$. 
From \eqref{eq:forward1_B}, we obtain a lower bound
\begin{equation}
 \lambda_{n}(\hat{A})\ge  \lambda_{n}(X^{\top}BX) - \|E_{1}\|_2. 
\label{ineq: lower bound}
\end{equation}
Accordingly to apply \eqref{eq:delta1}, we first shift $\hat{A}$ by the upper bound in \eqref{err_1_B}. 
Thus a choice of $s$ that avoids numerical breakdown is 
\begin{eqnarray}
\tilde{s}&:=&2.2m\sqrt{mn}\uu\|X\|^{2}_{2}\|B\|_{2}
+c_{n+2}{\bf u}\mbox{tr}(\hat{A}+2.2m\sqrt{mn}\uu\|X\|^{2}_{2}\|B\|_{2} I).
\label{eq:tildesB2}
\end{eqnarray}
Together with the assumption \eqref{eq:assumption_s_B}, we obtain
\begin{equation}
s:=\max(11\{2m\sqrt{mn}+n(n+1)\}{\uu}\|X\|_2^2\|B\|_2, \tilde{s}).
\label{eq:max_sB}
\end{equation}
First we note that $\hat{A}=X^{T}BX+E_{1}$ from \eqref{eq:forward1_B}. 
As in \eqref{eq:trE1} and \eqref{eq:trXTX}, we have
\begin{eqnarray*}
\mbox{tr}(E_{1}) \le 2.2mn\sqrt{mn} \uu \|X\|^{2}_{2}\|B\|_{2}, 
\mbox{tr}(X^{\top}BX) \le n\|X^{\top}BX\|_{2} \le n \|X\|^{2}_{2}\|B\|_{2}. 
\end{eqnarray*}
Thus, we can bound $\tilde{s}$ as
\begin{eqnarray*}
\tilde{s}&:=&2.2m\sqrt{mn}\uu\|X\|^{2}_{2}\|B\|_{2} \\
&&+2.2(n+1){\bf u}\left(n \|X\|^{2}_{2}\|B\|_{2} +  2.2mn\sqrt{mn} \uu \|X\|^{2}_{2}\|B\|_{2} + 2.2mn\sqrt{mn}\uu\|X\|^{2}_{2}\|B\|_{2}\right)\\
&\le& 2.4\{m\sqrt{mn} + n(n+1)\}\uu \|X\|^{2}_{2}\|B\|_{2},
\label{eq:tildesB21}
\end{eqnarray*}
which shows that the maximum in (\ref{eq:max_sB}) is 
\begin{equation}
s:=11\{2m\sqrt{mn}+n(n+1)\}{\uu}\|X\|_2^2\|B\|_2. 
\label{eq:finds2B}
\end{equation}

\ignore{
Algorithm~\ref{algB} blends \shiftcholqr\ and CholeskyQR2 in an adaptive manner: it attempts
 be applicable to matrices with possibly ${\|\hat{Q}\|_2\sqrt{\|B\|_2}}/{\sqrt{\sigma_n(\hat{Q}^{\top}B\hat{Q})}} >\uu^{-1}$. 
The shifts are introduced only when necessary, judged by whether or not the Cholesky factorization 
$\mbox{chol}(A^{(k)})$ breaks down. 

 \begin{algorithm}[h]
 \caption{Iterated CholeskyQR in an oblique inner product for $X=QR$ with shifts when necessary.}
 \label{algB}
 \begin{algorithmic}[1]
  \STATE Let $Q:=X,R:=I$
  \REPEAT
  \STATE $A:=Q^TBQ$
  \STATE $\tilde{R}:=\mbox{chol}(A)$\quad //chol: Cholesky factorization
  \IF{$\mbox{chol}(A)$ breaks down}
  \STATE $s:=11\{2m\sqrt{mn}+n(n+1)\}{\uu}\|X\|_2^2\|B\|_2$
\quad //introduce shift
  \STATE $\tilde R:=\mbox{chol}(A+sI)$
  \ENDIF
  \STATE $Q:=Q\tilde R^{-1}$, $R:=\tilde RR$
 \UNTIL{$\norm{Q^{T}BQ-I}_{F}\le \sqrt{n}\uu\|B\|_{F}$}
    \end{algorithmic}
 \end{algorithm}
}

We now summarize the algorithm in pseudocode 
in Algorithm~\ref{alg:shiftchol2B}, which is the $B$-orthogonal analogue of Algorithm~\ref{alg:shiftchol3}. We still call the algorithm \shiftcholqr3 (supressing $B$ as the two algorithms are essentially equivalent when $B=I$, aside from a slight difference in the shift strategy). 
The iterated Algorithm~\ref{alg12} can also be extended to $B\neq I$ similarly; we omit its pseudocode for brevity. 
 In the analysis below, we focus on the case $\uu^{-\frac{1}{2}}<\|X\|_2\sqrt{\|B\|_2}/{\sqrt{\sigma_n(X^{\top}BX)}}<\uu^{-1}$. 

 \begin{algorithm}[h]
 \caption{\shiftcholqr3 for $X=QR$,  $Q^TBQ=I_n$.
}
 \label{alg:shiftchol2B}
 \begin{algorithmic}[1]
  \STATE Let $Q:=X$
  \STATE $A:=Q^TBQ$
  \STATE $s:=11\{2m\sqrt{mn}+n(n+1)\}{\uu}\|X\|_2^2\|B\|_2$
  \STATE $R:=\mbox{chol}(A+sI)$
  \STATE $Q:=QR^{-1}$
  \STATE $\tilde R:=\mbox{chol}(Q^TQ)$,  $Q:=Q\tilde R^{-1}$, $R:=\tilde RR$
  \STATE $\tilde R:=\mbox{chol}(Q^TQ)$,  $Q:=Q\tilde R^{-1}$, $R:=\tilde RR$
    \end{algorithmic}
 \end{algorithm}

\paragraph{When is thrice enough?}
Here we derive a condition on $\|\hat{Q}\|_2\sqrt{\|B\|_2}/{\sqrt{\sigma_n(\hat{Q}^{\top}B\hat{Q})}}$ that guarantees that 
\shiftcholqr3 gives a numerically stable QR factorization of $X$. 

Recall that \eqref{eq:CNreduction_B} gives a bound by \shiftcholqr: 
\begin{eqnarray}
\frac{\|\hat{Q}\|_2\sqrt{\|B\|_2}}{\sqrt{\sigma_n(\hat{Q}^{\top}B\hat{Q})}}
&\le&  2\sqrt{3}\cdot\sqrt{1+\alpha\cdot\frac{\|X\|_2^2\|B\|_2}{\sigma_n(X^{\top}BX)}}\cdot\sqrt{\kappa_2(B)}.
\end{eqnarray}
As in~\eqref{eq:finds2B}, to guarantee avoidance of breakdown we take $\alpha=11\{2m\sqrt{mn}+n(n+1)\}{\uu}$, so we have
\begin{equation}
  \label{eq:hatqcondB}
\frac{\|\hat{Q}\|_2\sqrt{\|B\|_2}}{\sqrt{\sigma_n(\hat{Q}^{\top}B\hat{Q})}}
\le  2\sqrt{3}\cdot\sqrt{1+11\{2m\sqrt{mn}+n(n+1)\}{\uu}\frac{\|X\|_2^2\|B\|_2}{\sigma_n(X^{\top}BX)}}\cdot\sqrt{\kappa_2(B)}.
\end{equation}
On the other hand, 
as shown in~\cite{yamamotojsiam2015}, 
a sufficient condition for CholeskyQR2 in an oblique inner product to compute  a stable QR factorization of $\hat{Q}$ is 
\begin{equation}\label{eq:cholqr2Bgurantee}
\frac{\|\hat{Q}\|_2\sqrt{\|B\|_2}}{\sqrt{\sigma_n(\hat{Q}^{\top}B\hat{Q})}}
\le \frac{1}{8\sqrt{(m\sqrt{mn}+n(n+1))\uu}}. 
\end{equation}
Combining these facts, we obtain the following condition under which \shiftcholqr3 is guaranteed to compute a numerically stable QR factoriziation:
\[
2\sqrt{3}\cdot\sqrt{1+11\{2m\sqrt{mn}+n(n+1)\}{\uu}\frac{\|X\|_2^2\|B\|_2}{\sigma_n(X^{\top}BX)}}\cdot\sqrt{\kappa_2(B)}
\leq\frac{1}{8\sqrt{(m\sqrt{mn}+n(n+1))\uu}}.
\]
If $\frac{\| X\|_2\sqrt{\|B\|_2}}{\sqrt{\sigma_n(X^{\top}BX)}}>\uu^{-\frac{1}{2}}$, we have 
\[
1+11\{2m\sqrt{mn}+n(n+1)\}{\bf u} \frac{\|X\|_2\sqrt{\|B\|_2}}{\sqrt{\sigma_n(X^{\top}BX)}}\simeq 11\{2m\sqrt{mn}+n(n+1)\}{\bf u} \frac{\|X\|_2\sqrt{\|B\|_2}}{\sqrt{\sigma_n(X^{\top}BX)}}
\]
 and the condition can be simplified as

\begin{equation}\label{eq:guaranteeB}
\frac{\| X\|_2\sqrt{\|B\|_2}}{\sqrt{\sigma_n(X^{\top}BX)}} \leq \frac{{\uu^{-1}}}{96\{2m\sqrt{mn}+n(n+1)\}\sqrt{\kappa_{2}(B)}}.
\end{equation}
Note that this ensures that the condition \eqref{eq:assumption_kappa_B} for the error analysis in
Subsection~\ref{sec:stabB} is automatically satisfied.

\paragraph{Numerical stability}

We now examine the numerical stability of \shiftcholqr3 and show that it enjoys excellent stability both in orthogonality and backward error. 

\begin{theorem}\label{thm:Bsta}
  Let $B\in\mathbb{R}^{m\times m}\succ 0$ and
$X\in\mathbb{R}^{m\times n}$ be a matrix satisfying~\eqref{eq:guaranteeB}, 
and $80\kappa_2(B)(m\sqrt{mn}\uu + n(n+1)\uu)\leq 1$. Then \shiftcholqr\ followed by CholeskyQR2 computes a QR factorization $X=QR$ satisfying the $B$-orthogonality measure 
\begin{eqnarray}
\norm{\hat Q^{\top}B\hat Q-I}_{F}\le8[m\sqrt{mn}\uu+n(n+1)\uu]\kappa_{2}(B), \label{boundQ:Chol2B}
\end{eqnarray}
and backward error 
  \begin{equation}    \label{eq:berrB}
\frac{\|\hat{Q}\hat{R}-X\|_F}{\|\hat{X}\|_2}
\le 16n^2{\uu}{(\kappa_{2}(B))^{3/2}}. 
  \end{equation}
\end{theorem}

\begin{proof}
When both (\ref{boundQ:Chol2B}) and $80\kappa_2(B)(m\sqrt{mn}\uu + n(n+1)\uu)\leq 1$ hold, the assumptions (\ref{eq:assumption_kappa_B}), (\ref{eq:assumption_m_B}) and (\ref{eq:assumption_n_B}) for the application of the shifted Cholesky QR algorithm are automatically satisfied and the computed orthogonal factor $\hat{Q}$ satisfies (\ref{eq:cholqr2Bgurantee}). Then the $B$-orthogonal version of CholeskyQR2 can be safely applied to $\hat{Q}$ and the resulting orthogonal factor $\hat{Z}$ satisfies (\ref{boundQ:Chol2B}), as shown in Theorem 2 of \cite{yamamotojsiam2015}.

We next establish (\ref{eq:berrB}). 
By (\ref{eq:backward2X_B}), after the first \shiftcholqr\ we have
\begin{equation}
X + \Delta X = \hat{Q}\hat{R},
\end{equation}
where $\|\Delta X\|_{F}$ is bounded as in (\ref{eq:Bres1_B}).
We apply CholsekyQR2 to $\hat{Q}$ to obtain the QR factorization $\hat{Q}=ZU$. 
We have 
\begin{equation}
\hat{Q} + \Delta\hat{Q} = \hat{Z}\hat{U}
\label{eq:Q+deltaQB}
\end{equation}
where 
\begin{equation}
\|\Delta \hat{Q}\|_{F}\le 5 n^2 \uu\|\hat{Q}\|_2\kappa_2(B) \le\frac{ 5\sqrt{3}}{\sqrt{\sigma_n(B)}}\kappa_2(B)n^2 \uu
\label{eq:delta_QB}
\end{equation}
from (\ref{eq:hatQbound2_B}) and \cite [Thm.~3]{yamamotojsiam2015}. 
The upper triangular factor $S$ in the QR factorization of $X$ is obtained as
\begin{equation}
\hat{S}=\hat{U}\hat{R} + \Delta S
\end{equation}
where $\Delta S$ is the forward error in the matrix multiplication.
Summarizing, we can bound the overall backward error as 
\begin{eqnarray}
\|\hat{Z}\hat{S}-X\|_F
&=& \|\hat{Z}(\hat{U}\hat{R}+\Delta S)-\hat{Q}\hat{R}+\Delta X\|_F \nonumber \\
&=& \|\Delta\hat{Q}\hat{R}+\hat{Z}\Delta S+\Delta X\|_F \nonumber \\
&\le& \|\Delta\hat{Q}\|_F\|\hat{R}\|_2 + \|\hat{Z}\|_2\|\Delta S\|_F + \|\Delta X\|_F.
\label{eq:resid_shiftedCholesky2B}
\end{eqnarray}
We now bound the terms in the right-hand side. 
For the first term, we can bound $\|\Delta\hat{Q}\|_F$ as (\ref{eq:delta_QB}). 
To bound $\|\hat{R}\|_2$, we use (\ref{eq:assumption_s_B}), (\ref{eq:shiftchol_B}) and (\ref{eq:E1E2bound_B}) to obtain 
\begin{equation}
\|\hat{R}\|_2 \le \sqrt{1.1}\|X\|_2\sqrt{\|B\|_2}.
\label{eq:hatRboundB}
\end{equation}
We next bound the second term in~\eqref{eq:resid_shiftedCholesky2B}. 
By \cite{yamamotojsiam2015}, 
$\|\hat{Z}\|_2$ can be bounded as
\begin{equation}
\|\hat{Z}\|_2 \le \frac{\sqrt{2}}{\sqrt{\sigma_n(B)}}.
\label{eq:hatZnormB}
\end{equation}
Regarding $\|\Delta S\|_F$, 
using the general error bound for matrix multiplications 
$|\Delta S|\le\gamma_n|\hat{U}|\,|\hat{R}|$ we obtain 
\begin{eqnarray}
\|\Delta S\|_F \le \gamma_n\|\,|\hat{U}|\,|\hat{R}|\,\|_{F} \le \gamma_n\|\hat{U}\|_F\|\hat{R}\|_F \le n\gamma_n\|\hat{U}\|_2\|\hat{R}\|_2.
\end{eqnarray}
Here $\|\hat{R}\|_2$ can be bounded as in (\ref{eq:hatRboundB}).
To bound $\|\hat{U}\|_2$, we recall (\ref{eq:Q+deltaQB})
and left-multiply $\hat{Z}^{\top}B$ to obtain
\begin{equation}
\hat{Z}^{\top}B(\hat{Q}+\Delta\hat{Q})=\hat{Z}^{\top}B\hat{Z}\hat{U}. 
\end{equation}
Now, by \cite[Thm.~2]{yamamotojsiam2015}, 
the eigenvalues of $\hat{Z}^{\top}B\hat{Z}$ lie in the interval $[1-8\kappa_2(B)(m\sqrt{mn}{\uu}+n(n+1){\uu}),1+8\kappa_2(B)(m\sqrt{mn}{\uu}+n(n+1){\uu})]$, so it follows that 
\begin{eqnarray}
\|\hat{U}\|_2 &\le& \|(\hat{Z}^{\top}B\hat{Z})^{-1}\|_2\|\hat{Z}^{\top}\|_2\|B\|_2(\|\hat{Q}\|_2+\|\Delta\hat{Q}\|_2) \nonumber \\
&\le& \frac{1}{1-8\kappa_2(B)(m\sqrt{mn}+n(n+1)){\uu}}\frac{\sqrt{2}}{\sqrt{\sigma_n(B)}}\|B\|_2\left(\frac{\sqrt{3}}{\sqrt{\sigma_n(B)}}+\frac{ 5\sqrt{3}}{\sqrt{\sigma_n(B)}}\kappa_2(B)n^2 \uu\right) \nonumber 
\end{eqnarray}
Using the assumption $80\kappa_2(B)(m\sqrt{mn}\uu + n(n+1)\uu)\leq 1$ 
(which is \cite[eqn.~(1)]{yamamotojsiam2015}), we have
\begin{eqnarray}
\|\hat{U}\|_2 \le 2.9 \kappa_2(B). 
\end{eqnarray}
Finally, we can bound $\|\Delta X\|_F$ as in (\ref{eq:Bres1_B}). 

Combining the above bounds and substituting into 
(\ref{eq:resid_shiftedCholesky2B}) yields 
\begin{eqnarray}
\|\hat{Z}\hat{S}-X\|_F
&\le& \frac{ 5\sqrt{3}}{\sqrt{\sigma_n(B)}}\kappa_2(B)n^2 \uu\cdot \sqrt{1.1}\|X\|_2\sqrt{\|B\|_2}\nonumber\\
&&+\frac{\sqrt{2}}{\sqrt{\sigma_n(B)}}\cdot 1.02n^2\uu \cdot2.9\kappa_2(B) \sqrt{1.1}\|X\|_2\sqrt{\|B\|_2}
+ 2n^2\uu\sqrt{\kappa_2(B)}\|X\|_2\nonumber \\
&\le& 16n^2{\uu}\|X\|_2(\kappa_2(B))^{3/2}, \nonumber \
\end{eqnarray}
as required.   
\end{proof}

Experiments indicate that the bounds in Theorem~\ref{thm:Bsta} are overestimates, in particular the dependence on $\kappa_2(B)$ appears to be much weaker. 

\section{Numerical experiments}\label{sec:ex}
In this section we present some numerical experiments to illustrate our results. 
All computations were carried out on MATLAB 2017b and IEEE standard 754 binary64 (double precision) in Mac OS X version 10.13 with 2 GHz Intel Core i7 Duo processor, so that ${\bf u}=2^{-53}\approx1.11\times10^{-16}$. 

\subsection{Convergence with CholeskyQR iterates}
First, we take $B=I$ and examine how $\kappa_{2}(\hat{Q}^{(k)})$ 
and $\|\hat{Q}^{(k)\top}\hat{Q}^{(k)}-I\|_{2}$ 
are reduced after $k$ (shifted)Cholesky QR steps. 
We also 
compare \shiftcholqr3 with the mixed-precision Cholesky QR (mixedCholQR) \cite{yamazaki2015mixed}, which uses doubled precision for the first two steps~\eqref{eq:gram}, \eqref{cholA} and repeats the process in double precision. 
To run mixedCholQR we used the Multiprecision Computing Toolbox~\cite{adva-mct}, which enables computation in MATLAB with arbitrary precision.  
We generate test matrices with a specified condition number by forming 
\begin{equation}\label{testmat}
X := U \Sigma V^{T} \in {\mathbb R}^{m\times n}, 
\end{equation}
where $U$ is an $m \times n$ random orthogonal matrix obtained by taking the QR factorization of a random matrix, 
$V$ is an $n \times n$ random orthogonal matrix and 
$$
\Sigma = {\rm diag}(1, \sigma^{\frac{1}{n-1}}, \cdots, \sigma^{\frac{n-2}{n-1}}, \sigma). 
$$ 
Here, $0<\sigma<1$ is some constant. 
This is essentially MATLAB's {\tt randsvd} construction. 
Thus $\|X\|_2=1$ and the $2$-norm condition number of $X$ is $\kappa_2(X)=1/\sigma$. 
Let $k$ denote the number of iterations. 
In Table \ref{tab:itechol1}, $\kappa_{2}(X)=10^{12}, m=1000, n=30$. 
In Table \ref{tab:itechol2}, $\kappa_{2}(X)=10^{13}, m=100, n=100$. 

Tables \ref{tab:itechol1} and \ref{tab:itechol2} illustrate that the conditioning  $\kappa_{2}(\hat{Q}^{(1)})$ is improved by \shiftcholqr\ to approximately $\mathcal{O}(\sqrt{\alpha})\kappa_{2}(X)\lesssim \uu^{-\frac{1}{2}}$. Then $\|\hat{Q}^{(1)\top}\hat{Q}^{(1)}-I\|_{2}= \mathcal{O}(1)$ and CholeskyQR2 safely computes the QR factorization of of $\hat{Q}^{(1)}$; these are all consistent with Theorem \ref{eq:kap2} and Lemma \ref{lemma1}. 
We also see that $\kappa_{2}(\hat{Q}^{(1)})\approx \uu\kappa_2(X)=\mathcal{O}(1)$ in mixedCholQR, which is consistent with Theorem 3.4 in \cite{yamazaki2015mixed}. 
Here \shiftcholqr3 requires one more iteration than mixedCholQR, which is 
a typical behavior and reflects the theory for $\kappa_2(X)\in (\uu^{-1/2},\uu^{-1})$. 
\shiftcholqr3 has the advantage over mixedCholQR
that no high-precision arithmetic is needed, thereby being much faster in practice; indeed here it was faster by orders of magnitude.

We next turn to $B\neq I$. 
We set $B$ to be a SPD matrix as above, with $U=V$. 
We illustrate Theorem~\ref{eq:kap2_B} by examining how $\frac{\|X\|_2\sqrt{\|B\|_2}}{\sqrt{\sigma_n(X^{\top}BX)}}$ is reduced as compared with $\sqrt{\alpha\kappa_2(B)}$ where $\alpha=\frac{\sqrt{s}}{\|X\|_{2}\sqrt{\|B\|_2}}$. 
We set $X$ to be a random matrix with a specified condition number as in (\ref{testmat}) and form $B\succeq 0$ as in~\eqref{testmat}, now taking $V=U^T$. 
In Table \ref{tab:itecholB}, we took $\kappa_{2}(X)=10^{12}, \kappa_{2}(B)=10^{8}$, and $m=300, n=30$. 

Table \ref{tab:itecholB} also confirms that $\frac{\|\hat{Q}^{(1)}\|_2\sqrt{\|B\|_2}}{{\sqrt{\sigma_n(\hat{Q}^{(1)T}B\hat{Q}^{(1)})}}}$ is improved by \shiftcholqr3 by a factor $\mathcal{O}(\sqrt{\alpha/\kappa_2(B)})\lesssim \uu^{-\frac{1}{2}}$ (experiments suggest that the $\sqrt{\kappa_2(B)}$ dependence is often a significant overestimate). CholeskyQR2 then safely completes the QR factorization, as predicted by Theorem \ref{eq:kap2_B}. 

\begin{table}[H]
  \begin{center}
    \caption{Results for test matrices with $\kappa_2(X)=10^{12}$, $m=1000$, $n=30$, $B=I$. }
  \begin{tabular}{ccccccc} \hline
   $ $ &  \multicolumn{3}{c}{Algorithm~\ref{alg12}}& & \multicolumn{2}{c}{mixedCholQR}\\\cline{2-4}\cline{6-7}\mystrut{8pt}
    $k$ & $\kappa_2(\hat{Q}^{(k)})$ & $\|\hat{Q}^{(k)\top}\hat{Q}^{(k)}-I\|_{2}$ & $\sqrt{\alpha}$ & &$\kappa_2(\hat{Q}^{(k)})$ &$\|\hat{Q}^{(k)\top}\hat{Q}^{(k)}-I\|_{2}$\\ \hline \hline
    1 & $6.14\cdot 10^{6}$ & $1.00$ &  $ 1.80\cdot 10^{-5}$ & & $ 1.00$&$9.88\cdot 10^{-6}$\\
    2 & $1.01$ & $1.70\cdot 10^{-4}$ & - & & $ 1.00$&$9.04\cdot 10^{-16}$\\
    3 & $1.00$ & $5.66\cdot 10^{-16}$& - & &- &-\\ \hline 
  \end{tabular}
  \label{tab:itechol1}
    \end{center}
\end{table}

\begin{table}[H]
  \begin{center}
    \caption{Results for test matrices with $\kappa_2(X)=10^{13}$, $m=100$, $n=100$, $B=I$.}
  \begin{tabular}{ccccccc} \hline
   $ $ &  \multicolumn{3}{c}{Algorithm~\ref{alg12}}& & \multicolumn{2}{c}{mixedCholQR}\\\cline{2-4}\cline{6-7}\mystrut{8pt}
    $k$ & $\kappa_2(\hat{Q}^{(k)})$ & $\|\hat{Q}^{(k)\top}\hat{Q}^{(k)}-I\|_{2}$ & $\sqrt{\alpha}$ & &$\kappa_2(\hat{Q}^{(k)})$ &$\|\hat{Q}^{(k)\top}\hat{Q}^{(k)}-I\|_{2}$\\ \hline \hline
    1 & $4.95\cdot 10^{7}$ & $1.00$ &  $ 1.48\cdot 10^{-5}$ & & $ 1.00$&$3.62\cdot 10^{-4}$\\
    2 & $1.02$ & $1.93\cdot 10^{-2}$ & - & &$ 1.00$&$1.15\cdot 10^{-15}$\\
    3 & $1.00$ & $1.07\cdot 10^{-15}$& - & &- &-\\ \hline 
  \end{tabular}
  \label{tab:itechol2}
    \end{center}
\end{table}

\begin{table}[H]
  \begin{center}
    \caption{Results for test matrices with $B\neq I$, $\|X\|_2\sqrt{\|B\|_2}/{\sqrt{\sigma_n(X^{\top}BX)}}=1.38\times 10^{10}$, $m=300$, $n=30$. }
  \begin{tabular}{cccc} \hline
   $ $ &  \multicolumn{3}{c}{Algorithm~\ref{alg:shiftchol2B}}\\
   \cline{2-4}\mystrut{8pt}
    $k$ & $\frac{\|\hat{Q}^{(k)}\|_2\sqrt{\|B\|_2}}{{\sqrt{\sigma_n(\hat{Q}^{(k)T}B\hat{Q}^{(k)})}}}$ & $\|\hat{Q}^{(k)\top}B\hat{Q}^{(k)}-I\|_{2}$ & $\sqrt{\alpha \kappa_2(B)}$ \\ 
    \hline \hline
    1 & $4.11\cdot 10^{8}$ & $1.00$ &  $ 8.41\cdot 10^{-2}$ \\
    2 & $13.50$ & $8.31\cdot 10^{-3}$ & - \\
    3 & $13.50$ & $3.49\cdot 10^{-15}$& - \\ 
    \hline 
  \end{tabular}
  \\
   \smallskip
   \begin{tabular}{ccc} \hline
   $ $ &  \multicolumn{2}{c}{mixedCholQR}\\
   \cline{2-3}\mystrut{8pt}
    $k$ & $\frac{\|\hat{Q}^{(k)}\|_2\sqrt{\|B\|_2}}{{\sqrt{\sigma_n(\hat{Q}^{(k)T}B\hat{Q}^{(k)})}}}$ & $\|\hat{Q}^{(k)\top}B\hat{Q}^{(k)}-I\|_{2}$\\ 
    \hline \hline
    1 &$ 13.50$&$4.82\cdot 10^{-5}$\\
    2 & $ 13.50$&$3.34\cdot 10^{-15}$\\
    \hline 
  \end{tabular}
  \label{tab:itecholB}
    \end{center}
\end{table}

\begin{remark}
The choice of $s$ in~\eqref{eq:finds2B} tends to be a conservative overestimate, and in most cases, a successful Cholesky factorization can be computed with a smaller shift, such as  $A + (\uu \|X\|^2_2\|B\|_{2})I$. 
It can be seen (if Cholesky still does not break down) that the reduction factor of $\kappa_2(\hat{Q}^{(k)})$ improves to about $(\sqrt{\uu})^{(k)}$ after $k$ \shiftcholqr\ steps. 
To illustrate this, 
in Figure \ref{shift_test} we show 
the values of $\kappa_2(\hat{Q}^{(1)})$ as we vary the shift in  \shiftcholqr\, 
taking $B=I, m=1000$, $n=50$, $\kappa_2(X)=10^{15}$.
Our ``safe'' choice $s:=11\{mn+n(n+1)\}\uu\|X\|_{2}^{2}\approx 6.1\cdot 10^{-11}$ is shown in Figure \ref{shift_test} by a blue asterisk. 
In this case, $\kappa_2(\hat{Q}^{(1)})$ is larger than required by \eqref{eq:cholqr2gurantee}, and more \shiftcholqr\ iterations would be needed. 
On the other hand, if we set $s:=\uu \|X\|^2_2$, then 
$\hat{Q}^{(1)}$ will satisfy the sufficient condition 
$\kappa_2(\hat{Q}^{(1)})\leq \uu^{-\frac{1}{2}}$
for CholeskyQR2 to work. 
However, there is no guarantee that the initial Cholesky factorization $\mbox{chol}(A+\uu \|X\|^2_2 I )$ does not break down. 

\begin{figure}[H]
   \begin{center} 
      \includegraphics[width=90mm]{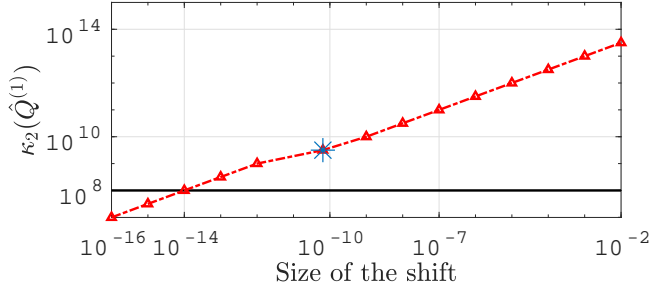}
            \caption{$\kappa_2(\hat{Q}^{(1)})$ for test 
matrix with $\kappa_2(X)=10^{15}$,
$m=1000$, $n=50$, varying the shift of \shiftcholqr .}
      \label{shift_test}
   \end{center}
\end{figure}  

\end{remark}

\subsection{Orthogonality and residual}
Next we examine the numerical stability of \shiftcholqr3 (shown in the figures as sCholQR3) and compare it with other popular QR decomposition algorithms, namely, Householder QR, classical and modified Gram-Schmidt (CGS and MGS; we also run them twice, CGS2 and MGS2).  We first take $B=I$ and vary $\kappa_2(X)$, $m$ and $n$ and investigate the dependence of the orthogonality and residual on them. 
We set $X$ 
as in (\ref{testmat}). 
We examine  the orthogonality and residual measured by the Frobenius norm under various conditions in Figures \ref{fig:varying_cond} through \ref{fig:varying_n}. 
Figure \ref{fig:varying_cond} shows the orthogonality $\|\hat{Q}^{T}\hat{Q}-I \|_F$ and residual $\|\hat{Q}\hat{R}-X \|_F$, where we take  $m=300$, $n=10$ and $\kappa_2(X)$ was varied from $10^8$ to $10^{15}$.
In Figure \ref{fig:varying_m}, $\kappa_2(X)=10^{12}$, $n=50$ and $m$ was varied from 1000 to 10000. 
In Figure \ref{fig:varying_n}, $\kappa_2(X)=10^{12}$, $m=1000$ and $n$ was varied from 10 to 500. 
 
We see in Figure~\ref{fig:varying_cond} that with \shiftcholqr, the orthogonality and the residual are independent of $\kappa_2(X)$ and are of $O({\bf u})$, as long as $\kappa_2(X)$ is at most $O({\bf u}^{-{1}})$. This is in good agreement with Theorem~\ref{thm:maincholqr3}. 
Figures \ref{fig:varying_m} and \ref{fig:varying_n} indicate that the orthogonality and residual increase only mildly with $m$ and $n$. Although they are inevitably overestimates, these also reflect our results (\ref{boundQ:Chol2}) and (\ref{eq:berr}). 
Compared with Householder QR, we observe that \shiftcholqr3 usually produces slightly better orthogonality and residual. 
With MGS, the deviation from orthogonality increases proportionally to $\kappa_2(X)$. 
As is well known, Gram-Schmidt type algorithms perform well when repeated twice, and we can verify this here. 
As mentioned in Section \ref{com_cgs2}, an advantage of \shiftcholqr3  is that it is rich in BLAS-3 operations and easily parallelized (it ran more than ten times faster than Gram-Schmidt algorithms in the experiments here). 
Overall, we see that \shiftcholqr3 is an efficient and reliable method for matrices with condition number at most $O({\bf u}^{-1})$.

Next, we again take $B\neq I$ and test 
Algorithm~\ref{alg:shiftchol2B}, comparing it with the stability of other popular QR decomposition algorithms, namely, MGS, CGS2 and MGS2. 
We varied $\kappa_2(B)$, $m$ and $n$ and investigated the orthogonality $\|\hat{Q}^{T}B\hat{Q}-I \|_F$ and residual $\|\hat{Q}\hat{R}-X \|_F$. 
Figure \ref{fig:varying_condX_B} shows the results 
for the case  $m=500$, $n=20$, $\kappa_2(B)=10^{10}$ and $\kappa_2(B)$ was varied from $10^1$ to $10^{10}$.
Figure \ref{fig:varying_condB_B} 
takes  $m=300$, $n=50$, $\kappa_2(X)=10^{10}$ and $\kappa_2(B)$ was varied from $10^8$ to $10^{15}$.
In Figure \ref{fig:varying_m_B} we took $\kappa_2(X)=10^{8}$, $\kappa_2(B)=10^{10}$, $n=50$ and $m$ was varied from 500 to 2000. 
Figure \ref{fig:varying_n_B} takes $\kappa_2(X)=10^{8}$, $\kappa_2(B)=10^{10}$, $m=1000$ and $n$ was varied from 50 to 500. 

From Figure \ref{fig:varying_condX_B} we see that the orthogonality and residual of \shiftcholqr3 are independent of $\kappa_2(X)$ and are of $O({\bf u})$, as long as $\frac{\|X\|_2\sqrt{\|B\|_2}}{\sqrt{\sigma_n(X^{\top}BX)}}\lesssim O({\bf u}^{-{1}})$, reflecting Theorem~\ref{thm:Bsta}. 
Figure \ref{fig:varying_condB_B} shows that the orthogonality increase rather mildly with $\kappa_2(B)$, indicating the dependence on $\kappa_2(B)$ suggested by~\ref{boundQ:Chol2B} is perhaps improvable; we leave this for future work.
 Figures \ref{fig:varying_m_B} and \ref{fig:varying_n_B} illustrate the orthogonality and residual increase with $m$ and $n$, again mildly. 
 These are also in agreement with  (\ref{boundQ:Chol2B}) and (\ref{eq:berrB}) in Theorem~\ref{thm:Bsta}; again, its $m,n$-dependence may be improvable. 
Compared with CGS2 and MGS2, the orthogonality and residual of CGS2 and MGS2 and that of \shiftcholqr3 are of the same magnitude. Again, \shiftcholqr3  has the advantage of being parallelization-friendly. 
All our experiments corroborate that \shiftcholqr3 is a reliable method whether $B=I$ or $B\neq I$, for matrices with $\frac{\|X\|_2\sqrt{\|B\|_2}}{\sqrt{\sigma_n(X^{\top}BX)}}<O({\bf u}^{-{1}})$.

\begin{figure}[htbp]
   \begin{center} 
      \includegraphics[width=60mm]{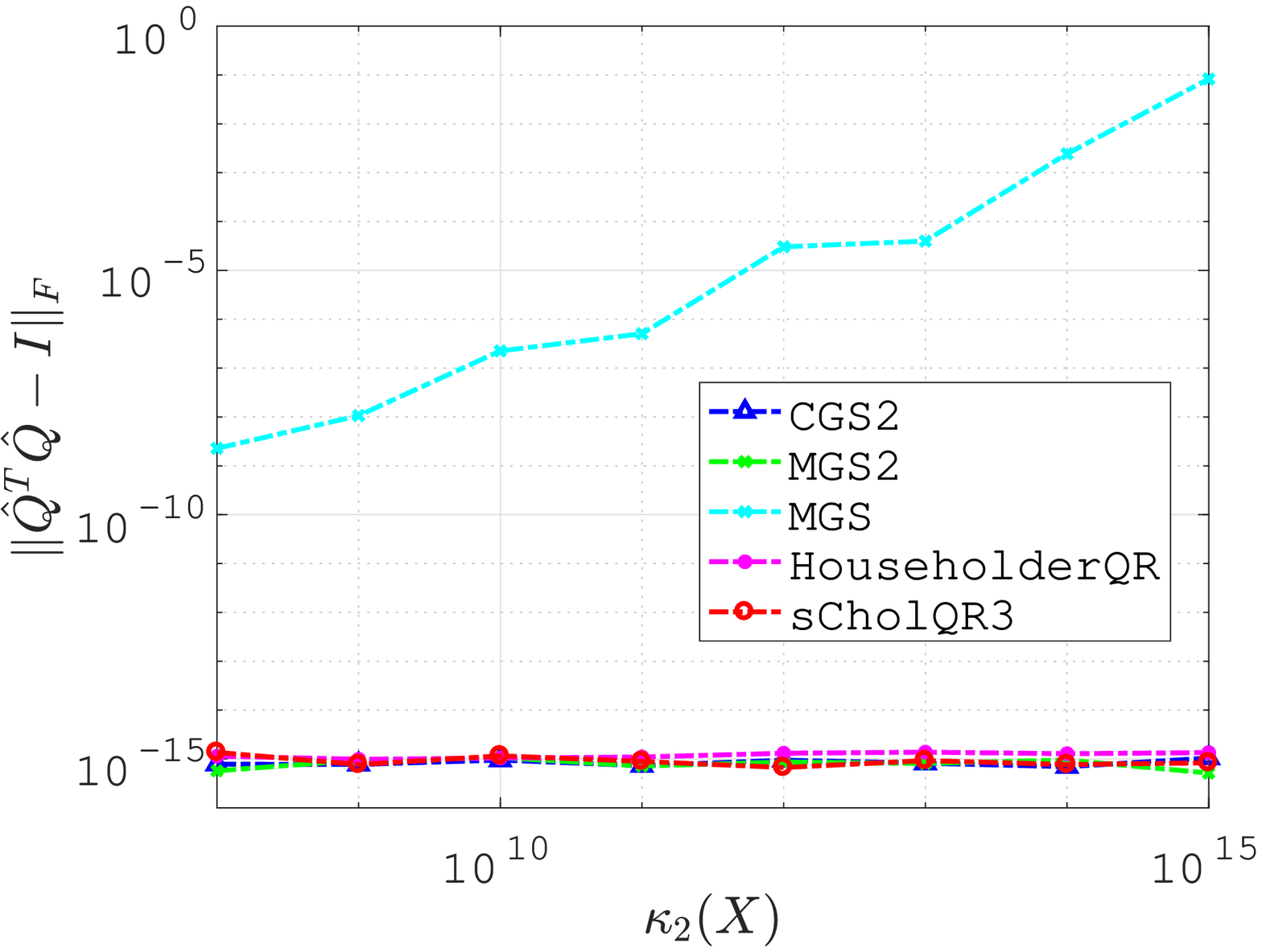}
      \includegraphics[width=60mm]{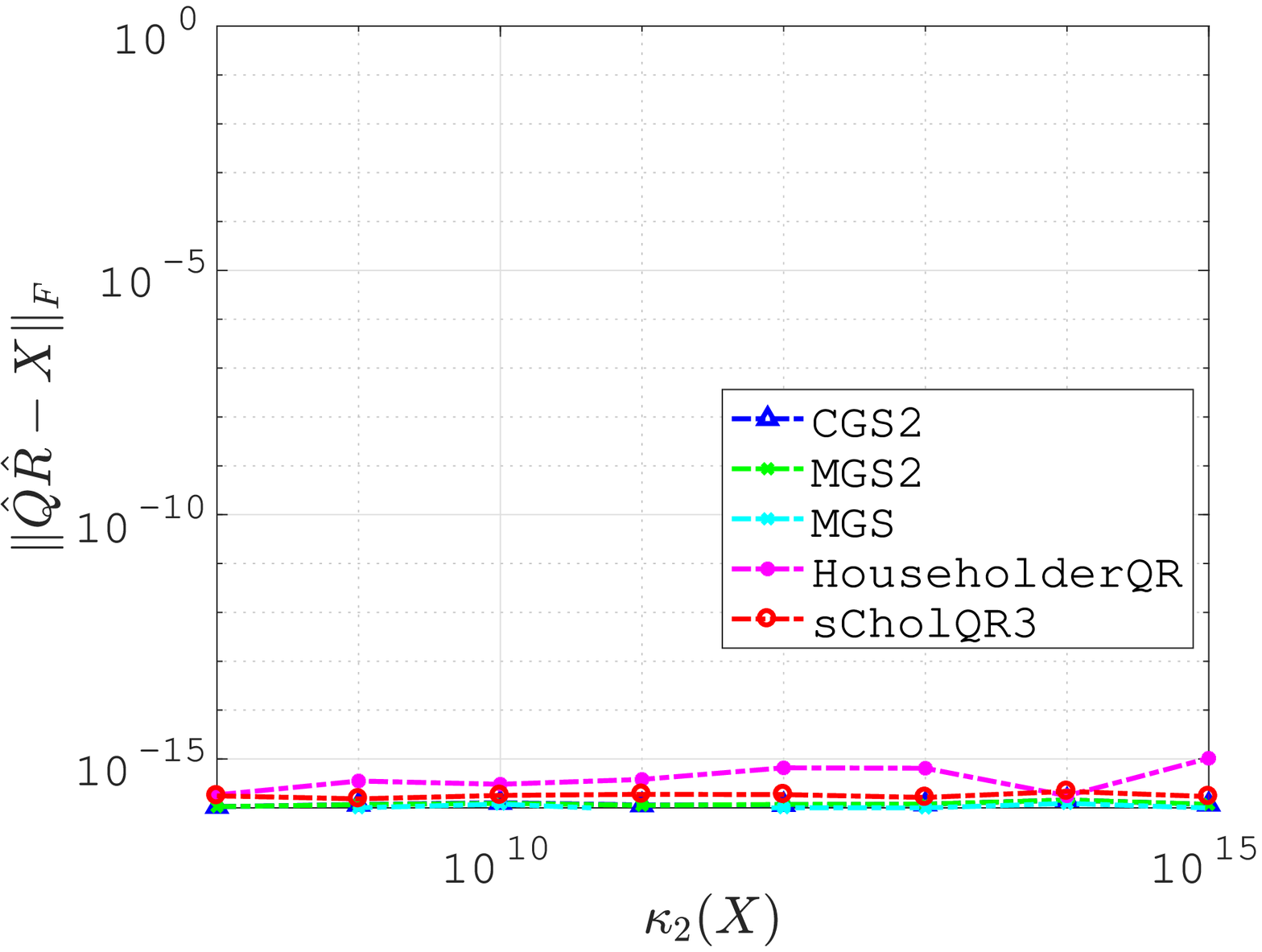}
            \caption{Orthogonality $\|\hat{Q}^{\top}\hat{Q}-I \|_F$ and residual $\|\hat{Q}\hat{R}-X \|_F$ for test matrices with $m=300$, $n=10$, varying $\kappa_2(X)$.}
      \label{fig:varying_cond}
   \end{center}
\end{figure}  

\begin{figure}[htbp]
   \begin{center} 
      \includegraphics[width=60mm]{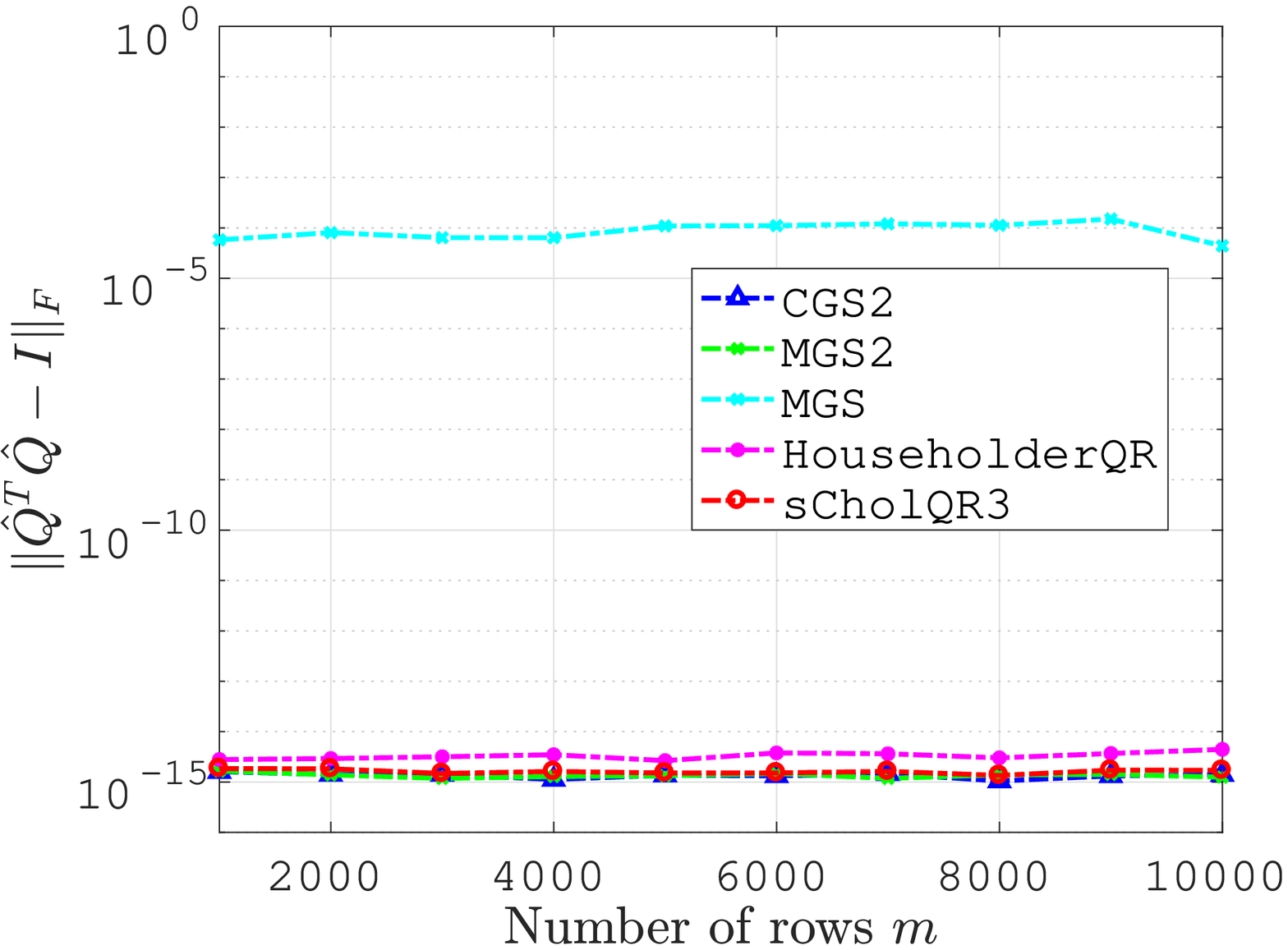}
      \includegraphics[width=60mm]{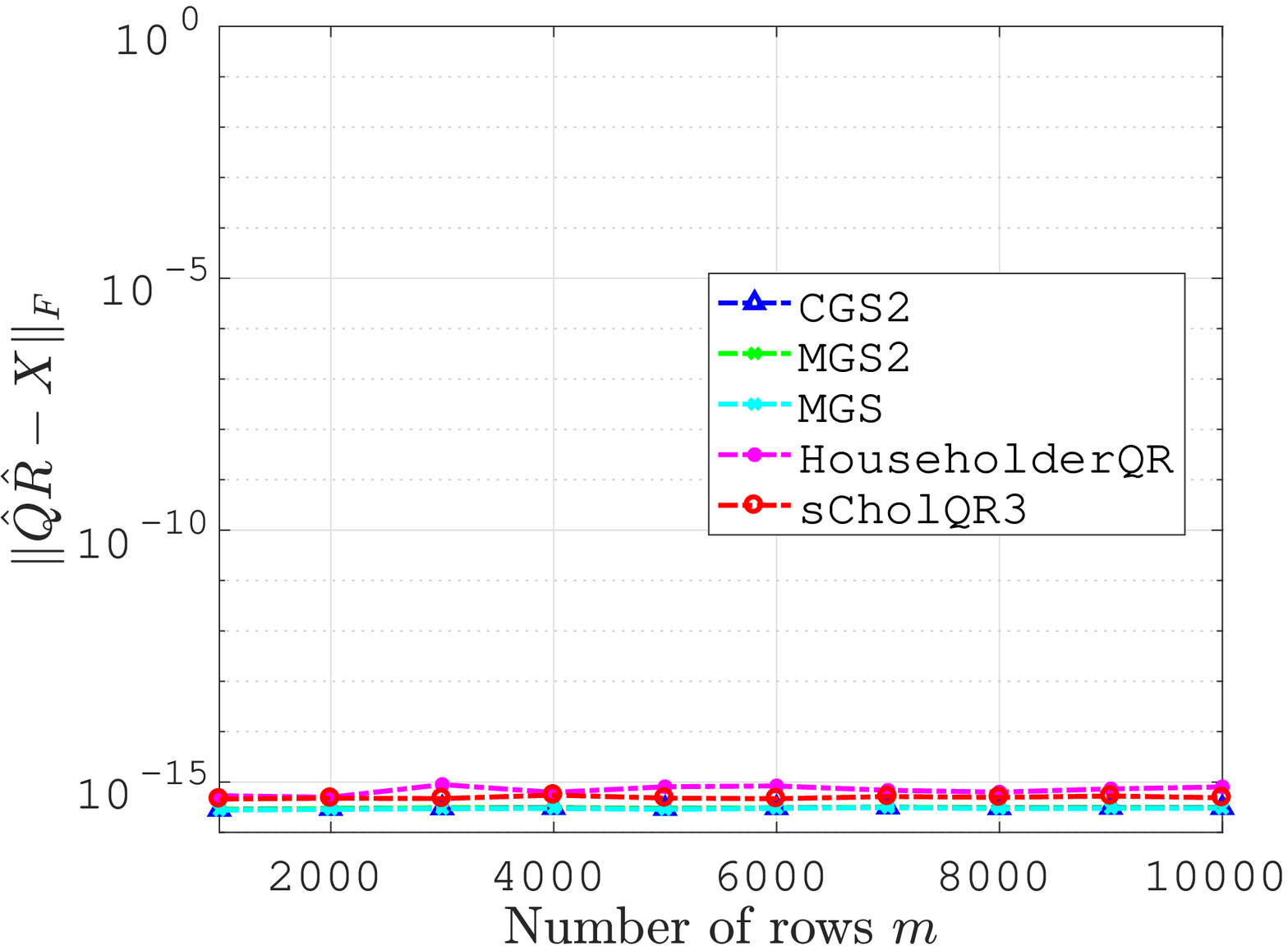}
      \caption{Orthogonality $\|\hat{Q}^{\top}\hat{Q}-I \|_F$ and residual $\|\hat{Q}\hat{R}-X \|_F$ for test matrices with $\kappa_2(X)=10^{12}$, $n=50$, varying $m$.}
      \label{fig:varying_m}
   \end{center}
\end{figure}  

\begin{figure}[htbp]
   \begin{center} 
      \includegraphics[width=60mm]{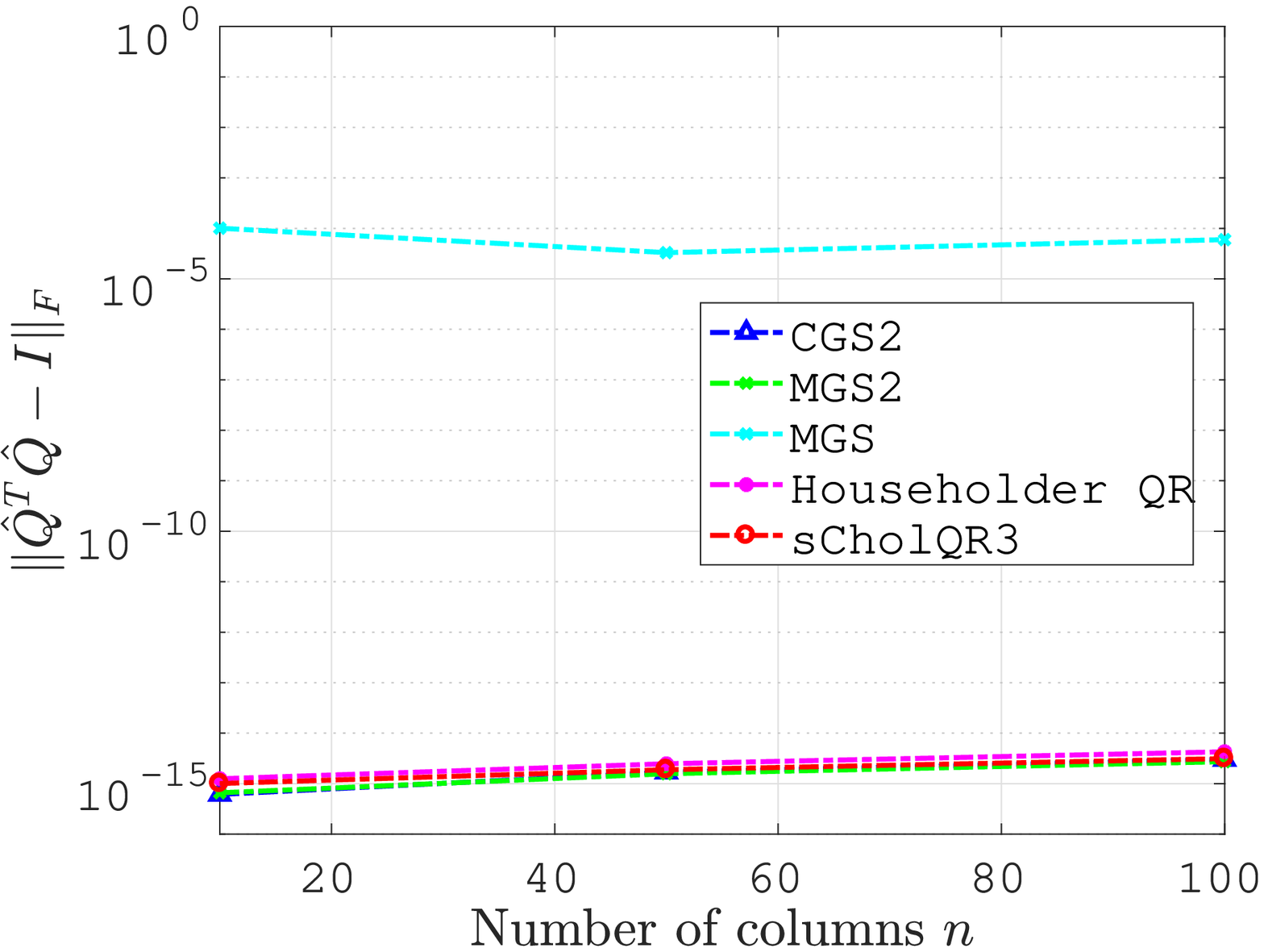}
      \includegraphics[width=60mm]{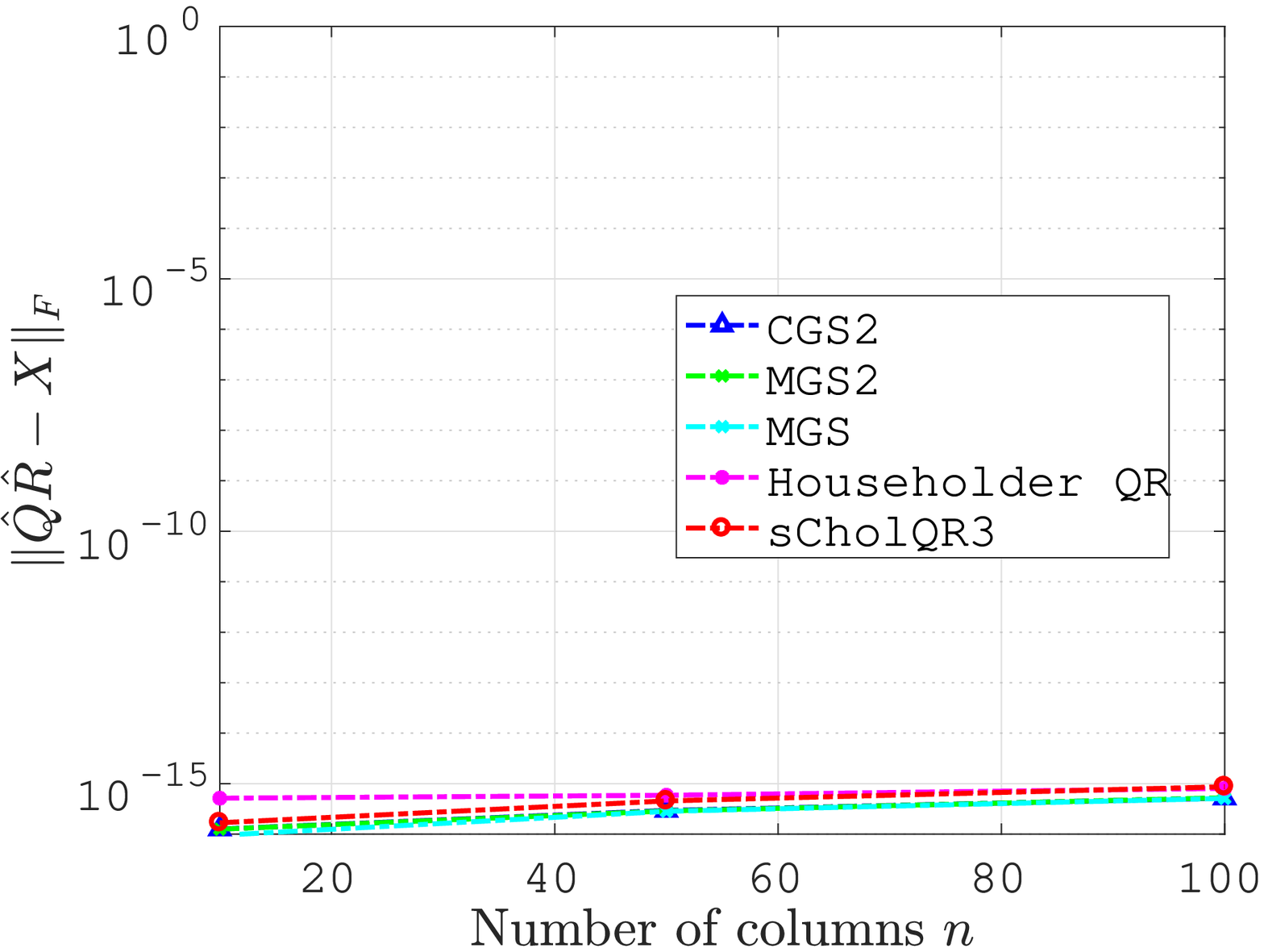}
      \caption{Orthogonality $\|\hat{Q}^{\top}\hat{Q}-I \|_F$ and residual $\|\hat{Q}\hat{R}-X \|_F$ for test matrices with $\kappa_2(X)=10^{12}$, $m=1000$, varying $n$.}
      \label{fig:varying_n}
   \end{center}
\end{figure}

\begin{figure}[htbp]
   \begin{center} 
      \includegraphics[width=60mm]{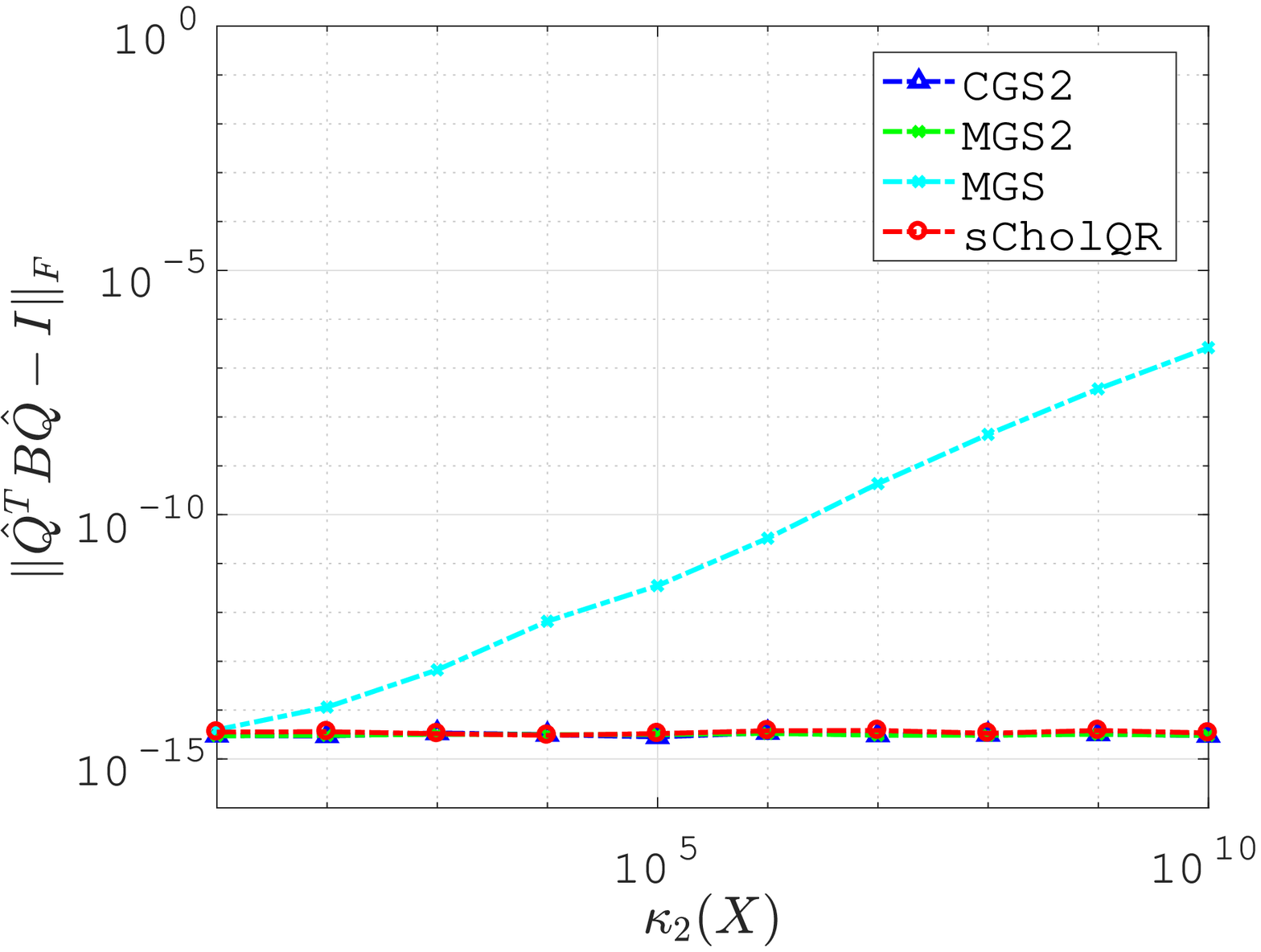}
      \includegraphics[width=60mm]{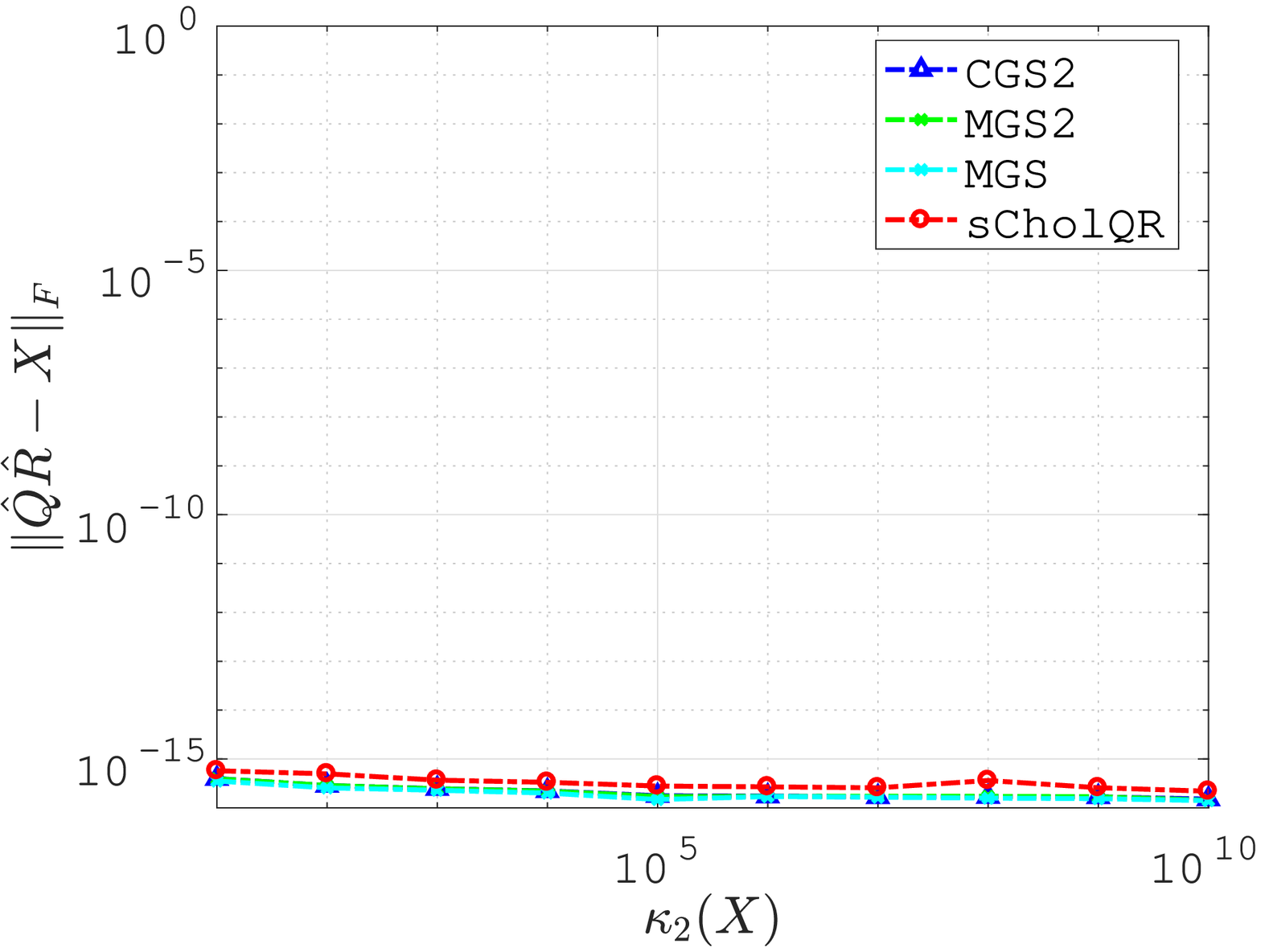}
      \caption{Orthogonality $\|\hat{Q}^{\top}B\hat{Q}-I \|_F$ and Residual $\|\hat{Q}\hat{R}-X \|_F$ for test matrices with $m=500$, $n=20$, $\kappa_2(B)=10^{10}$, varying $\kappa_2(X)$.}
      \label{fig:varying_condX_B}
   \end{center}
\end{figure}  

\begin{figure}[htbp]
   \begin{center} 
      \includegraphics[width=60mm]{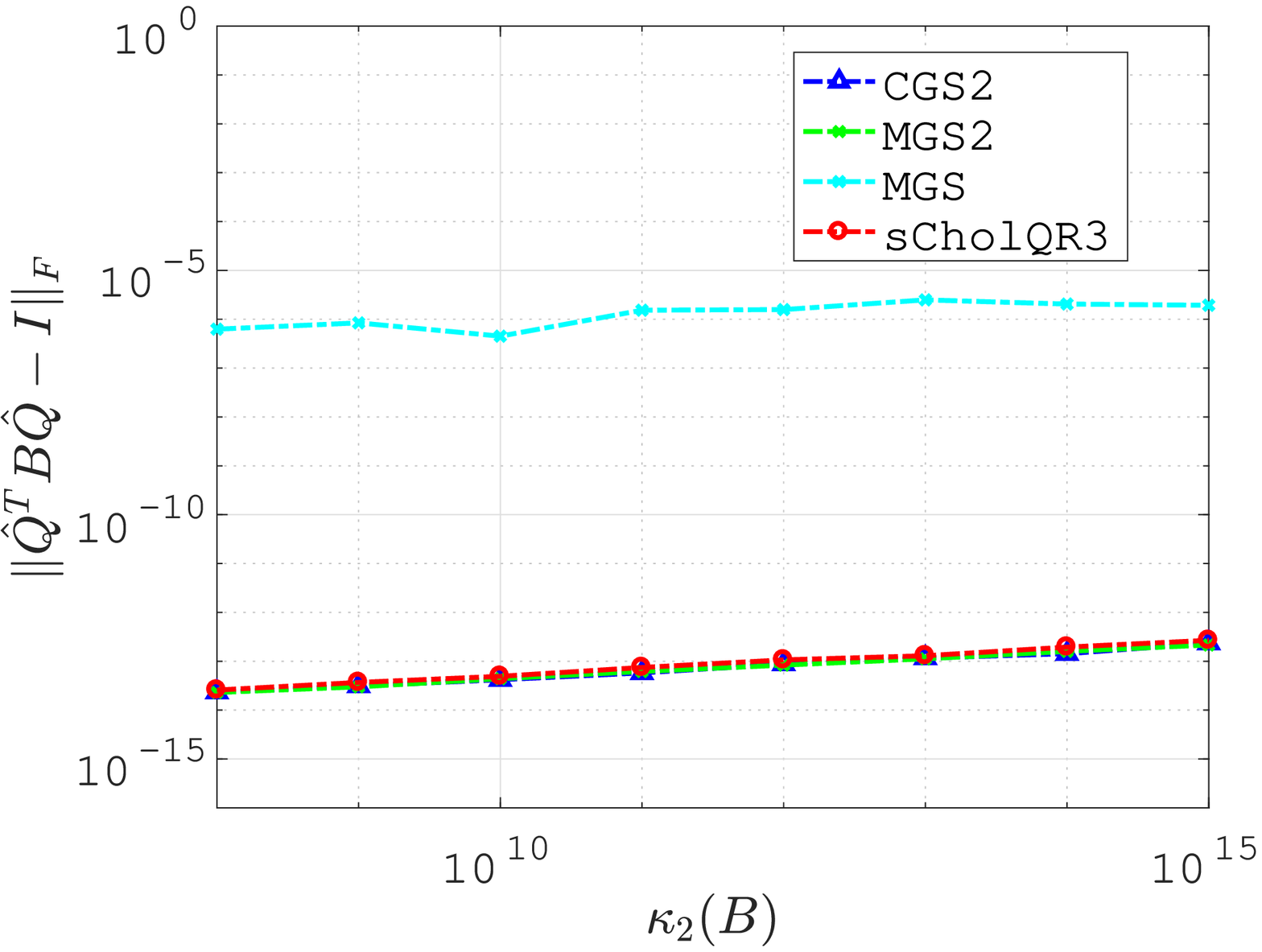}
      \includegraphics[width=60mm]{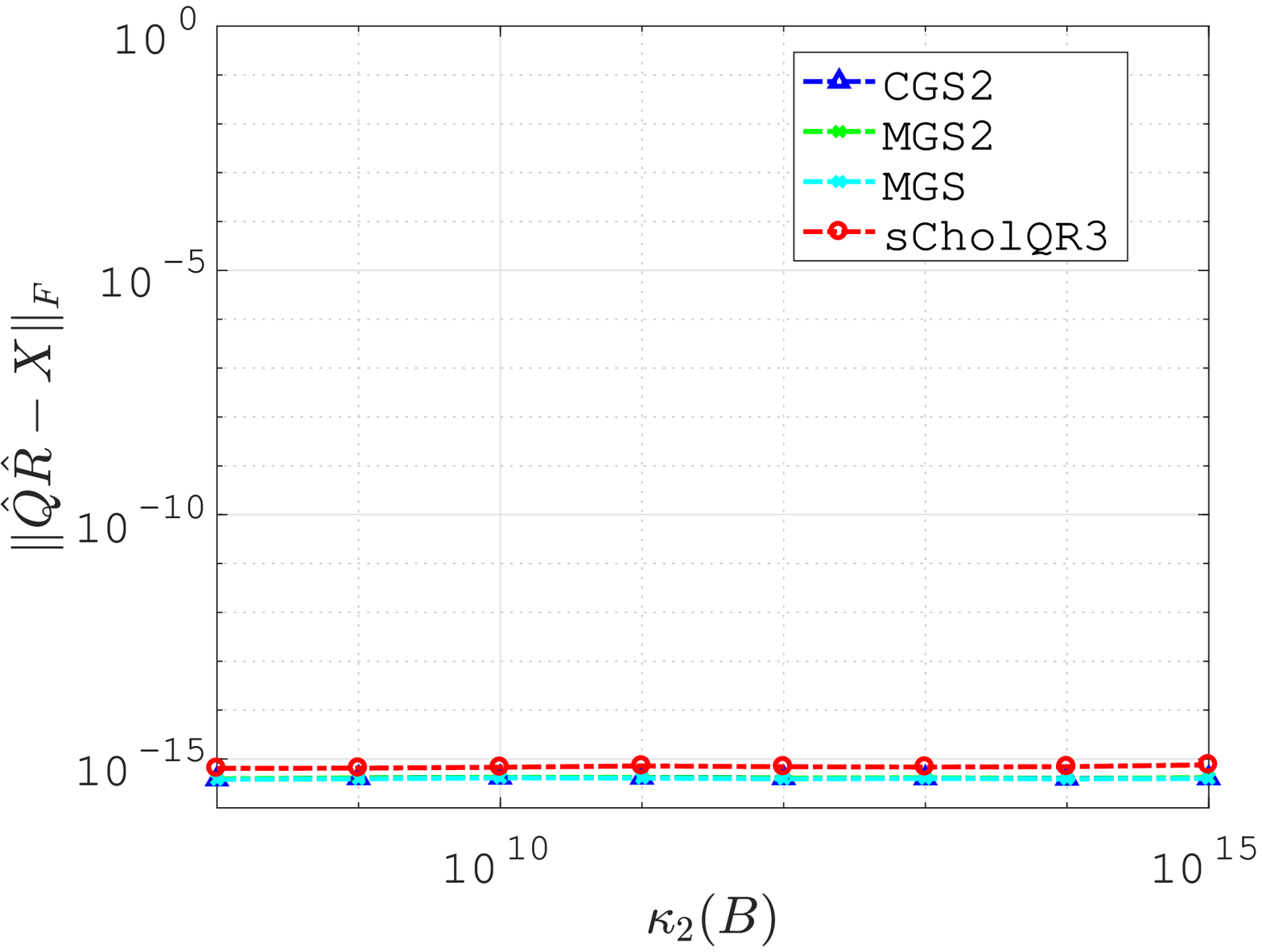}
      \caption{Orthogonality $\|\hat{Q}^{\top}B\hat{Q}-I \|_F$ and Residual $\|\hat{Q}\hat{R}-X \|_F$ for test matrices with $m=300$, $n=50$, $\kappa_2(X)=10^{10}$, varying $\kappa_2(B)$.}
      \label{fig:varying_condB_B}
   \end{center}
\end{figure}  

\begin{figure}[htbp]
   \begin{center} 
      \includegraphics[width=60mm]{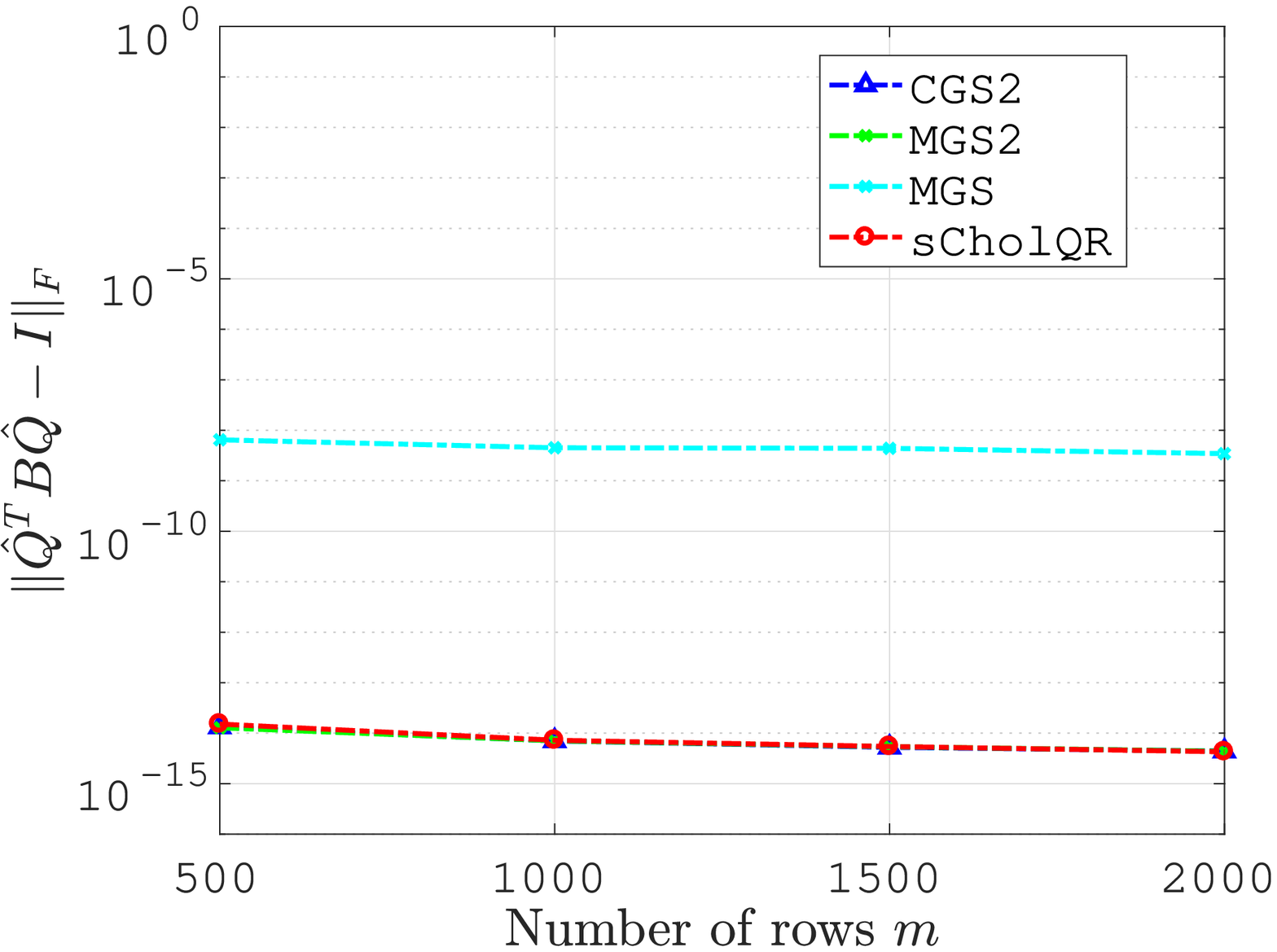}
      \includegraphics[width=60mm]{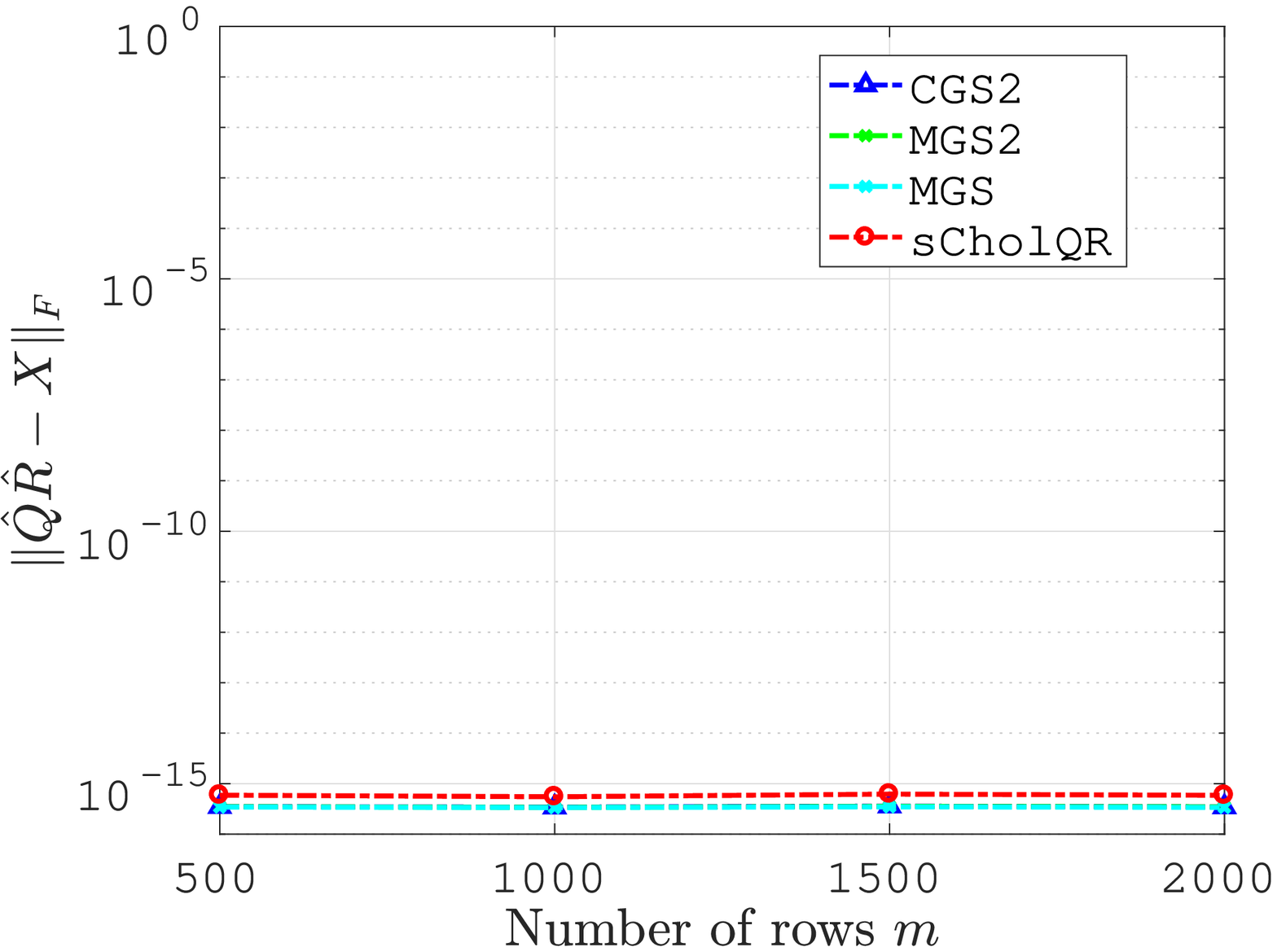}
      \caption{Orthogonality $\|\hat{Q}^{\top}B\hat{Q}-I \|_F$ and Residual $\|\hat{Q}\hat{R}-X \|_F$ for test matrices with $\kappa(X)=10^{8}$, $\kappa_2(B)=10^{10}$, $n=50$, varying $m$.}
      \label{fig:varying_m_B}
   \end{center}
\end{figure}  

\begin{figure}[htbp]
   \begin{center} 
      \includegraphics[width=60mm]{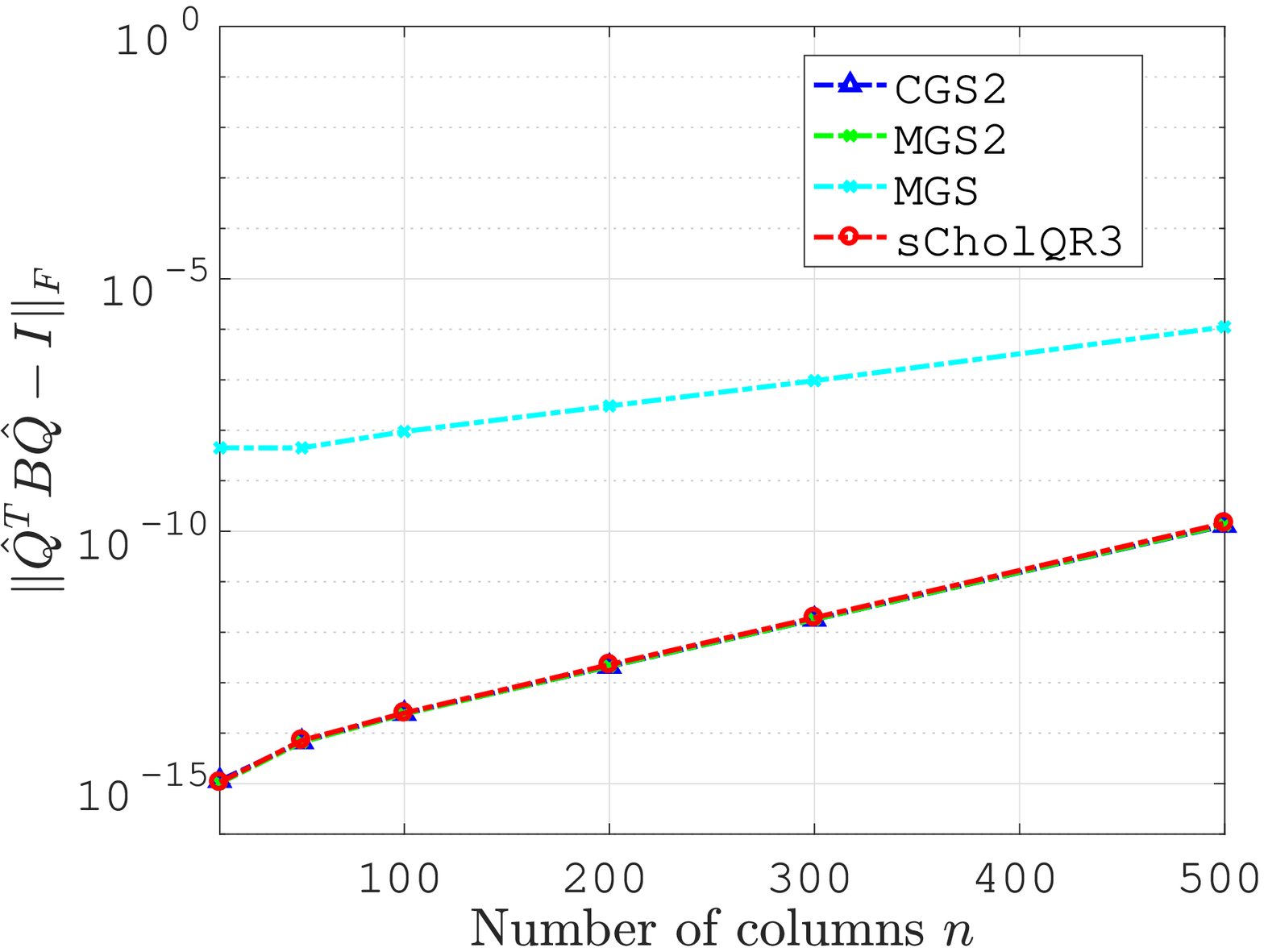}
      \includegraphics[width=60mm]{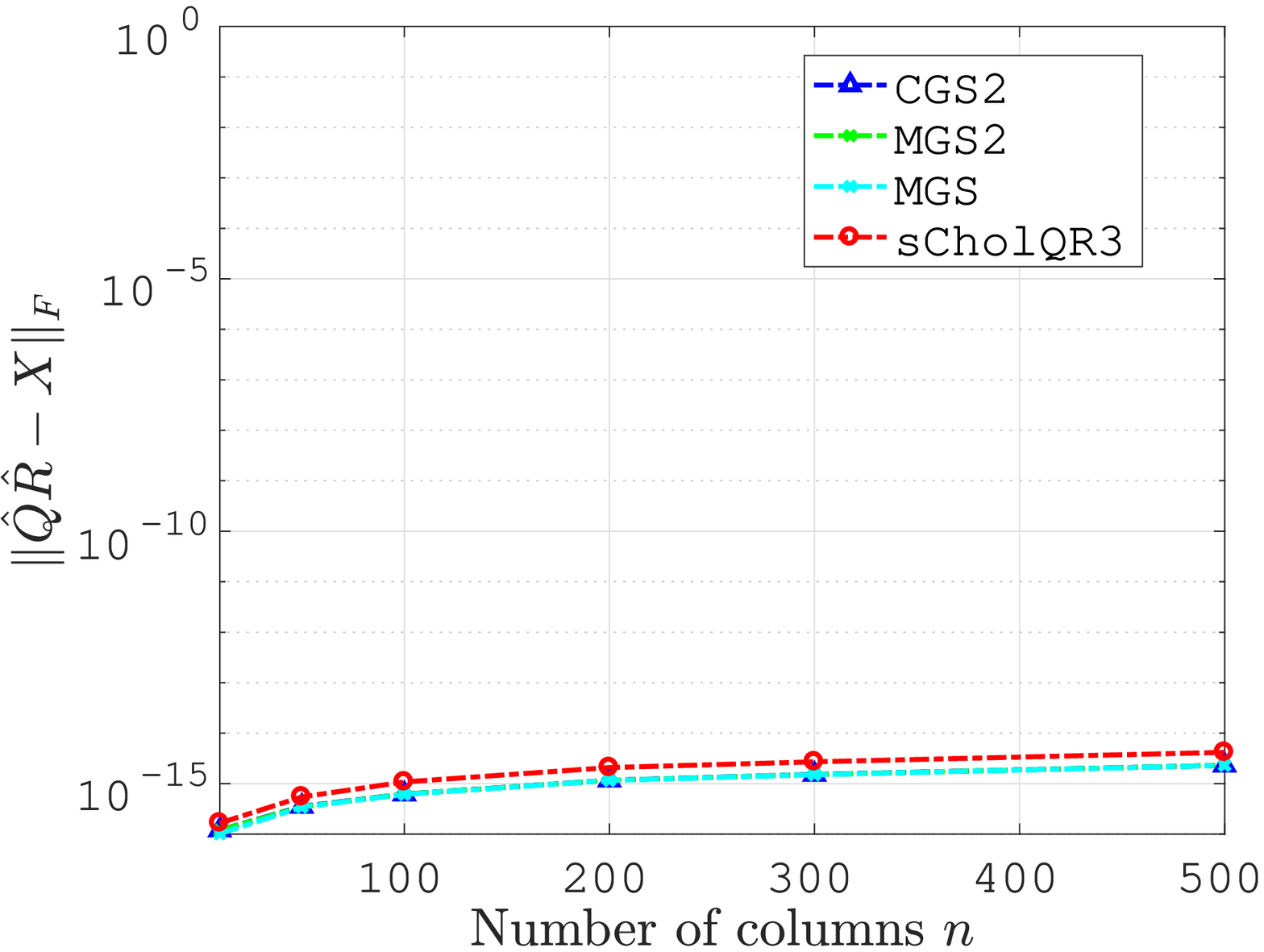}
      \caption{Orthogonality $\|\hat{Q}^{\top}B\hat{Q}-I \|_F$ and Residual $\|\hat{Q}\hat{R}-X \|_F$ for test matrices with $\kappa_2(X)=10^{8}$, $\kappa_2(B)=10^{10}$, $m=1000$, varying $n$.}
      \label{fig:varying_n_B}
   \end{center}
\end{figure}

\section{Runtime Performance}\label{sec:perf}
We next evaluate the runtime performance of \shiftcholqr3 in multi-core CPU environments for both the standard and the oblique case involving a sparse symmetric, positive definite matrix. 
\subsection{Standard inner product}
We used a compute node of the Laurel~2 supercomputer system installed at the Academic Center for Computing and Media Studies, Kyoto University, whose specifications are listed in Table~\ref{tab:spec_laurel2}. Here we focus on the standard $B=I$ case, to facilitate comparison with available algorithms.
Test matrices are generated as in the previous experiments, using \eqref{testmat}. 
Here, random orthogonal matrices are obtained by applying the LAPACK Householder QR routines ({\tt dgeqrf} and {\tt dorgqr}) to a random matrix. 
The code is written in Fortran90 and uses LAPACK and BLAS routines.

Table~\ref{tab:timing_xeon} presents the computational time of several methods, where the QR factorization of matrices whose condition number is $10^{11}$ is computed. 
Here, we compare \shiftcholqr3 ({\tt dgemm} and {\tt dsyrk} versions), Householder QR, and CGS2 (including its blocked version). 
It is worth noting that {\tt dgeqr} is a novel LAPACK routine that appropriately uses the TSQR algorithm. 
The block width in Block CGS2 was empirically tuned. 
Table~\ref{tab:timing_xeon} clearly shows that \shiftcholqr3 (both {\tt dgemm} and {\tt dsyrk} versions) outperforms other methods for all cases. 
Among methods besides \shiftcholqr3, {\tt dgeqr} is fastest, probably due to employing TSQR, but \shiftcholqr3 ({\tt dsyrk}) 
is more than $1.7$ times faster in every case. 
These results also indicate that, even if four iterations is required for ill-conditioned problems, iterated CholeskyQR (shown in Algorithm~\ref{alg12}) would still be faster than other methods. 

It is of interest to compare \shiftcholqr3 with mixedCholQR~\cite{yamazaki2015mixed}  from the viewpoint of computational time, but implementing and highly tuning double-double {\tt gemm} or {\tt syrk} routines (whose input matrices are in double precision) is generally difficult and requires significant effort.  
We thus estimate the computational cost of mixedCholQR routine; 
we only discuss the case where {\tt gemm} is used, but almost the same discussion is applicable when we use {\tt syrk}. 
Table~\ref{tab:blas_flops} presents the results of the benchmark for {\tt dgemm}, 
which computes $X^\top X$, where $X$ is an $m \times n$ matrix. 
From this table, we can assume that $700$ GFLOPS is a rough upper bound of the achieved performance. 
According to the paper on mixed precision CholeskyQR~\cite{yamazaki2015mixed}, the number of double precision operations required in {\tt ddgemm} (a double-double precision {\tt gemm} routine) is $8$ times (Cray-style) or $12.5$ times (IEEE-style) that required in {\tt dgemm}. 
Therefore, assuming that double precision operations in {\tt ddgemm} are performed at $700$ GFLOPS, we can estimate the time (sec.) of {\tt ddgemm} as 
$\frac{8\cdot2}{7}\cdot mn^2\cdot10^{-11}$ (Cray-style) or $\frac{12.5\cdot2}{7}\cdot mn^2\cdot10^{-11}$ (IEEE-style). 

Based on the above estimation and the breakdown of timing results of \shiftcholqr3, 
we compare \shiftcholqr3 and mixed precision CholeskyQR in Table~\ref{tab:time_mixed_chol_qr_n64}.
Here, we ignore the increasing cost for {\tt ddpotrf} because it is small relative to the total time. 
From the table, we can expect that \shiftcholqr3 is faster than mixed precision CholeskyQR in this computational environment. 
Considering this estimation and the fact that a well-tuned {\tt ddgemm} is currently rarely available, \shiftcholqr3 seems to be more practical than mixed precision CholeskyQR. 

Finally we briefly mention the performance of \shiftcholqr3 on large-scale distributed parallel systems. 
Based on our previous performance evaluation of CholeskyQR2 on the K computer (for details, see \cite{fukaya2014choleskyqr2}), we give a rough estimation of the computational time of \shiftcholqr3 in Table~\ref{tab:estimation_on_K}, where we estimate the computational time of \shiftcholqr3 as $1.5$ times that of CholeskyQR2, simply based on the number of iterations. 
From this table, \shiftcholqr3 is expected to be still significantly faster than Householder QR methods (both TSQR and ScaLAPACK routines) in large-scale parallel computation for matrices $\kappa_2(X)<\uu^{-1}$. 

\begin{table}[t]
\begin{center}
\caption{Specifications of the Laurel~2 system.}
\begin{tabular}{c|c}
\hline
Item & Specification\\
\hline \hline
CPU & Intel Xeon E5-2695 v4 (Broadwell, 2.1 GHz, 18 cores) \\
Number of CPUs / node & 2 \\
Memory size / node & 128 GB \\
Peak FLOPS / node & 1.21 TFLOPS (in double precision) \\
Compiler & Intel {\tt ifort} ver.~17.0.6 \\
Compile options & {\tt -mcmodel=medium}, {\tt -shared-intel}, {\tt -qopenmp}\\
	& {\tt -O3}, {\tt -ipo}, {\tt -xHost}\\
BLAS, LAPACK & Intel MKL ver.~2017.0.6 ({\tt -mkl=parallel})\\
\hline
\end{tabular}
\label{tab:spec_laurel2}
\end{center}
\end{table}

\begin{table}[t]
\begin{center}
\caption{Computational time on Laurel~2: $\kappa_2(X)=10^{11}$, $m=100,000$ and the number of threads is $36$.}
\begin{tabular}{c|cccc}
\hline
& \multicolumn{4}{c}{Time (sec.)}\\
\hline
Method & $n=32$ & $n=64$ & $n=128$ & $n=256$\\
\hline \hline
sCholQR3 ({\tt dgemm} ver.)    & $2.62 \times 10^{-3}$ & $9.16 \times 10^{-3}$ & $3.32 \times 10^{-2}$ & $1.16 \times 10^{-1}$\\
sCholQR3 ({\tt dsyrk} ver.)    & $2.39 \times 10^{-3}$ & $7.44 \times 10^{-3}$ & $2.45 \times 10^{-2}$ & $8.20 \times 10^{-2}$\\
{\tt dgeqrf} $+$ {\tt dorgqr} & $6.31 \times 10^{-2}$ & $8.94 \times 10^{-2}$ & $1.19 \times 10^{-1}$ & $1.96 \times 10^{-1}$ \\
{\tt dgeqr} $+$ {\tt dgemqr}  & $4.39 \times 10^{-3}$ & $1.30 \times 10^{-2}$ & $4.39 \times 10^{-2}$ & $1.68 \times 10^{-1}$ \\
CGS2                          & $9.60 \times 10^{-3}$ & $2.31 \times 10^{-2}$ & $1.40 \times 10^{-1}$ & $1.24$\\
Block CGS2                    & $9.62 \times 10^{-3}$ & $2.49 \times 10^{-2}$ & $7.88 \times 10^{-2}$ & $2.05 \times 10^{-1}$\\
\hline
\end{tabular}
\label{tab:timing_xeon}
\end{center}
\end{table}

\begin{table}[t]
\begin{center}
\caption{Achieved performance of {\tt dgemm}: $m = 100,000$ and the number of threads is $36$.}
\begin{tabular}{c|cccccccc}
\hline
$n$    & $32$  & $64$  & $128$ & $256$ & $512$ & $1,024$ & $4,096$ & $16,384$\\
\hline
GFLOPS & $476$ & $559$ & $539$ & $589$ & $600$ & $624$   & $654$   & $628$   \\
\hline
\end{tabular}
\label{tab:blas_flops}
\end{center}
\end{table}

\begin{table}[t]
\begin{center}
\caption{Comparison of \shiftcholqr3 with Mixed Precision CholeskyQR based on the performance estimation for {\tt ddgemm}: $m=100,000$ and $n=64$.}
\begin{tabular}{c|cc|ccc}
\hline
& \multicolumn{2}{c|}{\shiftcholqr3} & \multicolumn{3}{c}{Mixed Precision CholeskyQR}\\
\hline
&         &            &         & Cray-style & IEEE-style \\
& Routine & Time (sec.) & Routine & Time (sec.) & Time (sec.) \\
\hline \hline
sCholQR & {\tt dgemm}  & $1.90\times10^{-3}$ & -- & -- & --  \\
        & {\tt dpotrf} & $3.00\times10^{-5}$ & -- & -- & --  \\
        & {\tt dtrsm}  & $1.11\times10^{-3}$ & -- & -- & --  \\
\hline
CholQR  & {\tt dgemm}  & $1.68\times10^{-3}$ & {\tt ddgemm}  & $\ge 9.36\times10^{-3}$ & $\ge 1.46\times10^{-2}$\\
        & {\tt dpotrf} & $2.91\times10^{-5}$ & {\tt ddpotrf} & $\ge 2.91\times10^{-5}$ & $\ge 2.91\times10^{-5}$\\
        & {\tt dtrsm}  & $1.31\times10^{-3}$ & {\tt dtrsm}   & $1.31\times10^{-3}$ & $1.31\times10^{-3}$\\
        & {\tt dtrmm}  & $3.10\times10^{-5}$ & -- & -- &-- \\
\hline
CholQR  & {\tt dgemm}  & $1.81\times10^{-3}$ & {\tt dgemm}   & $1.81\times10^{-3}$ & $1.81\times10^{-3}$\\
        & {\tt dpotrf} & $3.79\times10^{-5}$ & {\tt dpotrf}  & $3.79\times10^{-5}$ & $3.79\times10^{-5}$\\
        & {\tt dtrsm}  & $1.24\times10^{-3}$ & {\tt dtrsm}   & $1.24\times10^{-3}$ & $1.24\times10^{-3}$\\
        & {\tt dtrmm}  & $2.91\times10^{-5}$ & {\tt dtrmm}   & $2.91\times10^{-5}$ & $2.91\times10^{-5}$\\
\hline
Misc.   &              & $1.36\times10^{-4}$ &               & $\ge 0$ & $\ge 0$\\
\hline
Total   &              & $9.34\times10^{-3}$ &               & $\ge 1.38\times10^{-2}$ & $\ge  1.91\times10^{-2}$\\
\hline
\end{tabular}
\label{tab:time_mixed_chol_qr_n64}
\end{center}
\end{table}

\begin{table}[t]
\begin{center}
\caption{Estimation of the computational time of \shiftcholqr3 on the K computer: the number of nodes ($=$ MPI processes) is $16,384$.}
\begin{tabular}{cc|ccc|c}
\hline
    &     & \multicolumn{4}{c}{Time (sec.)} \\
\hline
    &     & \multicolumn{3}{c|}{Measured} & Estimated \\
$m$ & $n$ & TSQR & {\tt pdgeqrf} $+$ {\tt pdorgqr} & CholQR2 & sCholQR3 \\
\hline \hline
4,194,304 & 16  & $1.64\times10^{-3}$ & $1.04\times10^{-2}$ & $8.02\times10^{-4}$ & $1.20\times10^{-3}$ \\
          & 64  & $7.42\times10^{-3}$ & $4.14\times10^{-2}$ & $2.52\times10^{-3}$ & $3.79\times10^{-3}$ \\
          & 256 & $2.32\times10^{-1}$ & $1.84\times10^{-1}$ & $3.05\times10^{-2}$ & $4.57\times10^{-2}$ \\
\hline
16,777,216& 16  & $1.84\times10^{-3}$ & $1.13\times10^{-2}$ & $9.06\times10^{-4}$ & $1.36\times10^{-3}$ \\
          & 64  & $8.82\times10^{-3}$ & $5.65\times10^{-2}$ & $3.13\times10^{-3}$ & $4.70\times10^{-3}$ \\
          & 256 & $2.40\times10^{-1}$ & $3.92\times10^{-1}$ & $3.38\times10^{-2}$ & $5.07\times10^{-2}$ \\
\hline
\end{tabular}
\label{tab:estimation_on_K}
\end{center}
\end{table}
\subsection{Oblique inner product}
In the case of the inner product defined by a positive definite matrix $B$, we 
 focus on large-sparse $B$, as arises commonly in applications. 
Owing to the sparsity of $B$, orthogonalizing $X$ such that $X^TBX = I$ depends heavily on the performance of sparse matrix-vector multiplication (SpMV). Both CGS2 and \shiftcholqr3 would therefore achieve only a fraction of the machine's peak performance compared with the case when $B$ is dense, where performance would be dictated by {\tt dgemm}. 
Optimizing the performance of \shiftcholqr3 for a sparse $B$ involves extracting performance from the computation of the inner product $X^TBX$, which in turn depends on
\begin{itemize}
	\item sparse matrix multiple-vector multiplication kernels \cite{ka13}, and
	\item the choice of whether $X^TBX$ is computed by forming $Y:= BX$ explicitly followed by computing $X^TY$ or in blocks. 
\end{itemize} 
By blocking we refer to forming the matrix $Y_{j}:= BX(:,jk:(j+1) k)$ for some block size $k \in [1,n]$ and $j \in (1,n/k)$ in succession, and subsequently calculating the matrix $X^TY_j$.
The choice of strategy is highly sensitive to both the sparsity structure of $B$, and the cache hierarchy of the architecture on which we execute. 
For an unbounded cache size, computing $Y$ explicitly will be the fastest strategy since the subsequent operation $Y^TX$ has excellent computational intensity.
Realistically however, for a small L1 cache or indeed a large $B$, forming $Y$ will cause the entries of $X$ to be evicted from cache, leading to unnecessary cache misses and a slower performance.
For a more detailed discussion and performance analysis of the different strategies, we refer the reader to \cite[sec. 6.4]{kannanthesis}; in this section we present the overall speed-up obtained over CGS2 implementations.

Our experiments for a sparse $B$ were carried out on a shared memory system with Intel Xeon E5-2670 (Sandy Bridge) symmetric multiprocessors.
There are 2 processors with 8 cores per processor, each with 2 hyperthreads that share a large 20 MB L3 cache.
Each core also has access to 32 KB of L1 cache and 256 KB of L2 cache and runs at 2.6 GHz.
They support the AVX instruction set with 256-bit wide SIMD registers, resulting in 10.4 GFlop/s of double precision performance per core or 83.2 GFlop/s per processor and a {\tt dgemm} performance of 59 GFlop/sec.
The implementation was parallelized using OpenMP, and compiled using Intel C++ Compiler version 14.0  with {\tt -O3} optimization level, with autovectorization turned on for both CPU. 
Tests are run by scheduling all threads on a single processor with 16 threads with OpenMP `compact' thread affinity, which is set using {\tt KMP\_SET\_AFFINITY}.

We use as test problems symmetric positive definite matrices from the University of Florida \cite{Davis2011} collection, as listed in Table \ref{tab-test-matrices}. 
The matrices are stored and operated upon $B$ using the Compressed Sparse Row (CSR) format and we avoid changing the format to favour performance although the results in \cite{ka13} strongly suggest that this is beneficial. 
The reason for this is that oblique QR factorization is usually part of a ``larger'' program, for example, a sparse generalized eigensolver \cite[chap. 4]{kannanthesis}, and hence the storage format needs may be governed by other operations in the parent algorithm, for example, a sparse direct solution, and such software may not exist for the new format.
Changing the sparse matrix format to accelerate the factorization may also slow down other parts of the calling program that are not optimized to work with a different format.

\begin{table}[]
	\caption{Test matrices used for benchmarking {\tt chol\_borth}.}
	\label{tab-test-matrices}
	\centering
	\begin{tabular}{p{2.2cm}|p{3cm}|c|c|c}
		Name & Application & Size & Nonzeros & Nonzeros/row \\
		\hline
		apache2 & finite difference & 715,176 & 2,766,523 &  3.86\\
		bairport &  finite element & 67,537 & 774,378 &  11.46\\
		bone010 & model reduction & 986,703 & 47,851,783 & 48.50 \\
		G3\_circuit & circuit simulation & 1,585,478 & 7,660,826 & 4.83 \\
		Geo\_1438 & finite element & 1,437,960 & 60,236,322 & 41.89 \\
		parabolic\_fem &  CFD & 525,825 & 2,100,225 & 3.99 \\
		serena & finite element & 1,391,349 &	64,131,971 	& 46.09 \\
		shipsec8 & finite element & 114,919 & 3,303,553 & 28.74  \\
		watercube & finite element & 68,598	& 1,439,940 & 20.99 \\
	\end{tabular}
\end{table}

We present the performance of \shiftcholqr3 on CPU by varying $n$, i.e., the size of $X$ and compare the results with those from CGS2 in Figure \ref{fig-chol-vs-cgs2}.
On the CPU, \shiftcholqr3 was faster than CGS2 by a minimum of 3.7 times for $n = 16$ and a maximum of 40 times for $n = 256$ vectors. 
The large speedup obtained is only a reflection of the reliance of CGS2 on matrix-vector products, which in the case of sparse matrices, has a particularly poor CPU utilization.

\begin{figure} 
\centering{
	\includegraphics[width=.70\textwidth]{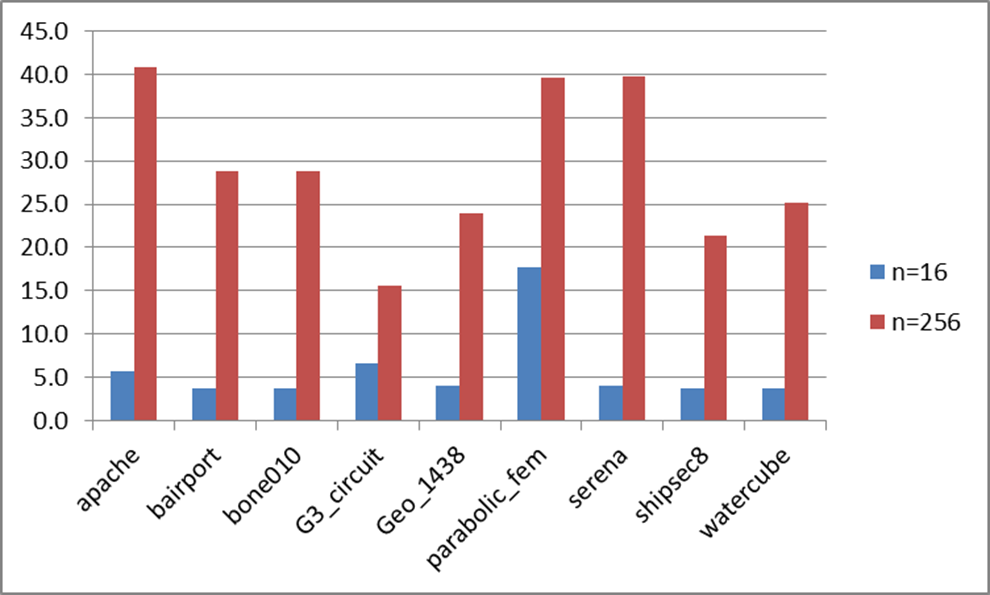}}
	\caption{Performance speed-ups of \shiftcholqr3 over CGS2 for sparse problems for varying sizes of $X$.}
	\label{fig-chol-vs-cgs2}
\end{figure}

\section{Conclusion and discussion}
Our algorithm \shiftcholqr3 combines speed, stability,  and versatility (applicable to $B\neq I$). We believe it offers an attractive alternative in high-performance computing to the conventional Householder-based QR factorziation algorithms when $B=I$, and can be the clear algorithm of choice when $B\neq I$. 

 \shiftcholqr3 as presented could benefit from further tuning, and this work suggests a few future directions. First, the choice of shift $s$ introduced in this paper is conservative, and severely so when $m,n$ are large. Our experiments suggest that a much smaller shift, such as $s=O(\uu\|A\|_2)$, is usually sufficient to avoid breakdown in $\mbox{chol}(A+s I )$, and as illustrated in Figure~\ref{shift_test}, a smaller shift results in improved conditioning, and hence smaller number of \shiftcholqr\ iterations. Introduction of a shift strategy that is both stable and efficient is an important remaining task. 

Our performance results hold promise for the competitiveness of \shiftcholqr3, and further work will focus on comparing it with other state-of-the-art implementations for QR factorziations such as TSQR in an HPC setting. 
In the case where $B \neq I$ for a sparse $B$, the performance benefits over CGS2 are remarkable and our algorithm is the clear choice for applications. 

Finally, for rank-deficient matrices, \shiftcholqr3 is inapplicable, because $AR$ is rank-deficient for any $R$\footnote{However, roundoff errors often map the zero singular values to $O(\uu)$, and so a few iterations of \shiftcholqr\ usually result in  $\kappa_2(Q)\leq \uu^{-1/2}$.}. This issue is not present in Householder-type methods, and a workaround for \shiftcholqr3 is much desired. 

\appendix 
\section{A sharper bound on the residual of CholeskyQR2}
In \cite{yamamotoetna2015}, the residual of CholeskyQR2 is bounded by $\|\hat{Z}\hat{U}-X\|_F \le 5n^2\sqrt{n}{\bf u}\|X\|_2$. In this appendix, we derive a sharper bound $5n^2{\bf u}\|X\|_2$.

In the CholeskyQR2 algorithm in floating-point arithmetic, the QR decomposition $X=ZU$ is computed as follows.
\begin{eqnarray}
&& \hat{A}=fl(X^{\top}X), \quad \hat{R}=fl({\rm chol}(\hat{A})), \quad \hat{Y}=fl(X\hat{R}^{-1}), \\
&& \hat{C}=fl(\hat{Y}^{\top}\hat{Y}), \quad \hat{S}=fl({\rm chol}(\hat{C})), \quad \hat{Z}=fl(\hat{Y}\hat{S}^{-1}), \quad \hat{U}=fl(\hat{S}\hat{R}).
\end{eqnarray}
Let us denote the residuals in the first and the second step by $\Delta X$ and $\Delta Y$, respectively, and the forward error in the computation of $\hat{U}$ by $\Delta U$. Then,
\begin{eqnarray}
X+\Delta X &=& \hat{Y}\hat{R}, \\
\hat{Y}+\Delta\hat{Y} &=& \hat{Z}\hat{S}, \\
\hat{U} &=& \hat{S}\hat{R}+\Delta U.
\end{eqnarray}
Using these quantities, the residual of CholeskyQR2 can be evaluated as
\begin{eqnarray}
\|\hat{Z}\hat{U}-X\|_F &=& \|\hat{Z}(\hat{S}\hat{R}+\Delta U)-\hat{Y}\hat{R}+\Delta X\|_F \nonumber \\
&=& \|\Delta\hat{Y}\hat{R}+\hat{Z}\Delta U+\Delta X\|_F \nonumber \\
&\le& \|\Delta\hat{Y}\|_F\|\hat{R}\|_2+\|\hat{Z}\|_2\|\Delta U\|_F+\|\Delta X\|_F.
\label{eq:residual_CQR2}
\end{eqnarray}

In \cite{yamamotoetna2015} the residual was bounded row-wise, then summed to bound $\|\hat{Z}\hat{U}-X\|_F$. The analysis below improves the bound by a factor $\sqrt{n}$ by directly bounding the matrix norm and using the refined analysis employed in Section~\ref{sec:stab}.
Now we bound each term in (\ref{eq:residual_CQR2}). In \cite{yamamotoetna2015}, $\|\hat{R}\|_2$ and $\|\Delta U\|_F$ (which is denoted as $\|E_5\|_F$ in \cite{yamamotoetna2015}) are evaluated as
\begin{eqnarray}
\|\hat{R}\|_2 &\le& 1.1\|X\|_2, \\
\|\Delta U\|_F &\le& 1.2n^2{\bf u}\|X\|_2.
\end{eqnarray}
$\|\hat{Z}\|_2$ can be bounded as in Eq.~(\ref{eq:hatZnorm}) of this paper. To bound $\|\Delta X\|_F$, we recall that the $i$th row of $\hat{Y}$, which we denote by $\hat{\bf y}_i^{\top}$, is computed from the $i$th row of $X$, which we denote by ${\bf x}_i^{\top}$, by triangular solution and therefore it holds that
\begin{equation}
\hat{\bf y}_i^{\top} = {\bf x}_i^{\top}(\hat{R}+\Delta\hat{R}_i)^{-1} \quad (i=1, 2, \ldots, m),
\end{equation}
where $\Delta \hat{R}_i$ is the backward error of the triangular solution. According to \cite{yamamotoetna2015}, $\|\Delta \hat{R}_i\|_2$ is bounded as
\begin{equation}
\|\Delta \hat{R}_i\|_2 \le 1.2n\sqrt{n}{\bf u}\|X\|_2.
\end{equation}
Hence, by denoting the $i$th row of $\Delta X$ by $\Delta{\bf x}_i^{\top}$ and noting the relationship $\Delta{\bf x}_i^{\top}=-\hat{\bf y}_i^{\top}\Delta\hat{R}_i$, we obtain
\begin{equation}
\|\Delta X\|_F
= \sqrt{\sum_{i=1}^m\|\hat{\bf y}_i^{\top}\Delta\hat{R}_i\|^2}
= \sqrt{\sum_{i=1}^m\|\hat{\bf y}_i^{\top}\|^2}\max_{1\le i\le m}\|\Delta\hat{R}_i\|_2
\le \|\hat{Y}\|_F\cdot 1.2n\sqrt{n}{\bf u}\|X\|_2.
\label{eq:DeltaXF}
\end{equation}
Since the singular values of $\hat{Y}$ are bounded by $\frac{\sqrt{69}}{8}$ (see the proof of Corollary 3.2 in \cite{yamamotoetna2015}), $\|\hat{Y}\|_F \le \sqrt{n}\|\hat{Y}\|_2 \le \frac{\sqrt{69}}{8}\sqrt{n}$. Inserting this into (\ref{eq:DeltaXF}), we have
\begin{equation}
\|\Delta X\|_F \le 1.3n^2{\bf u}\|X\|_2.
\end{equation}
A bound on $\|\Delta Y\|_F$ can be obtained by replacing $X$ and $\hat{Y}$ in (\ref{eq:DeltaXF}) with $\hat{Y}$ and $\hat{Z}$, respectively, and noting that $\|\hat{Y}\|_2 \le \frac{\sqrt{69}}{8}$ and $\|\hat{Z}\|_F \le \sqrt{n}\sqrt{1+6(mn{\bf u}+n(n+1){\bf u})} \le 1.1\sqrt{n}$ (see Theorem 3.3 in \cite{yamamotoetna2015}). The result is
\begin{equation}
\|\Delta\hat{Y}\|_F \le 1.4n^2{\bf u}.
\end{equation}
Putting all these together and inserting into (\ref{eq:residual_CQR2}), we finally have
\begin{eqnarray}
\|\hat{Z}\hat{U}-X\|_F &\le& 1.4n^2{\bf u}\cdot 1.1\|X\|_2 + \frac{\sqrt{76}}{8}\cdot 1.2n^2{\bf u}\|X\|_2 + 1.3n^2{\bf u}\|X\|_2 \nonumber \\
&\le& 5n^2{\bf u}\|X\|_2.
\end{eqnarray}

\def\noopsort#1{}\def\l{\char32l}\def\v#1{{\accent20 #1}}
  \let\^^_=\v\def\hbk{hardback}\def\pbk{paperback}

\end{document}